\documentclass[11pt, a4paper]{amsart}
\usepackage[colorlinks, linkcolor=blue, anchorcolor=blue, citecolor=green]{hyperref}
\usepackage{bbm}
\usepackage{quiver}
\usepackage{graphics,epic}
\usepackage{amsmath,amssymb, amsthm,mathrsfs}
\usepackage[all,2cell]{xy}
\usepackage{shorttoc}
\usepackage{tikz}
\usepackage{tikz-cd} 

\usepackage{setspace}
\newtheorem{theorem}{Theorem}[section]
\newtheorem*{theorem*}{Theorem}

\newtheorem{lemma}[theorem]{Lemma}
\newtheorem{proposition}[theorem]{Proposition}

\newtheorem*{conjecture*}{Conjecture}

\newtheorem{example}[theorem]{Example}
\newtheorem{remark}[theorem]{Remark}

\newtheorem{definition}[theorem]{Definition}

\newtheorem{thm}[theorem]{Theorem}
\newtheorem{lem}[theorem]{Lemma}
\newtheorem{prop}[theorem]{Proposition}

\newcommand{\ie}{{\em i.e.}\ }

\newcommand{\opname}[1]{\operatorname{\mathsf{#1}}}

\newcommand{\dimv}{\underline{\dim}\,}

\newcommand{\add}{\opname{add}\nolimits}

\newcommand{\ind}{\opname{ind}}

\newcommand{\old}{\opname{old}}
\newcommand{\new}{\opname{new}}

\newcommand{\cok}{\opname{cok}\nolimits}

\newcommand{\ra}{\rightarrow}

\newcommand{\id}{\mathbf{1}}

\newcommand{\bfA}{\mathbf{A}}
\newcommand{\Hom}{\opname{Hom}}

\newcommand{\Ext}{\opname{Ext}}

\newcommand{\End}{\opname{End}}

\newcommand{\mb}{\mathcal{B}}

\newcommand{\mpik}{\mathcal{P}_k}

\newcommand{\tauni}{\tau^{-1}}

\newcommand{\xra}{\xrightarrow}

\newcommand{\bbk}{\mathbbm{k}}

\newcommand{\Ini}{\mathbf{Ini}}
\newcommand{\Ter}{\mathbf{Ter}}
\newcommand{\din}{\mathbf{d_{in}}}
\newcommand{\dout}{\mathbf{d_{out}}}
\newcommand{\dinT}{\mathbf{d_{in,T}}}
\newcommand{\doutT}{\mathbf{d_{out,T}}}
\newcommand{\dinQ}{\mathbf{d_{in,Q}}}
\newcommand{\doutQ}{\mathbf{d_{out,Q}}}

\newcommand{\Ain}{\mathbf{A_{in}}}
\newcommand{\Aoin}{\mathbf{A^{\circ}_{in}}}
\newcommand{\Aout}{\mathbf{A_{out}}}
\newcommand{\Aoout}{\mathbf{A^{\circ}_{out}}}
\newcommand{\AinT}{\mathbf{A_{in,T}}}

\newcommand{\AoutT}{\mathbf{A_{out,T}}}

\newcommand{\AinQ}{\mathbf{A_{in,Q}}}
\newcommand{\AmutQ}{\mathbf{A_{Q}^{mut}}}

\newcommand{\AfixQ}{\mathbf{A_{Q}^{fix}}}
\newcommand{\AoinQ}{\mathbf{A^{\circ}_{in,Q}}}
\newcommand{\AoutQ}{\mathbf{A_{out,Q}}}

\newcommand{\rad}{\opname{rad}\nolimits}
\newcommand{\Irr}{\opname{Irr}\nolimits}

\newcommand{\IrrA}{\opname{Irr_{A}}\nolimits}
\newcommand{\IrrT}{\opname{Irr_T}\nolimits}
\newcommand{\IrrTo}{\opname{Irr_T^{\circ}}\nolimits}

\begin{document}

\title{Tilting mutation of gentle algebras and  its combinatorial description }\thanks{Partially supported by the National Natural Science Foundation of China (Grant No.  12171397, 12471037) }
\author[Deng]{Difan Deng}
\address{Difan Deng\\Department of Mathematics\\
	Southwest Jiaotong University\\
	610031 Chengdu \\
	P.R.China}
\email{difandeng@my.swjtu.edu.cn}

\author[Geng]{Shengfei Geng}
\address{Shengfei Geng\\Department of Mathematics\\
Sichuan University\\
610064 Chengdu\\
P.R.China}
\email{genshengfei@scu.edu.cn}

\author[Liu]{Pin Liu}
\address{Pin Liu\\Department of Mathematics\\
	Southwest Jiaotong University\\
	610031 Chengdu \\
	P.R.China}
\email{pinliu@swjtu.edu.cn}

\dedicatory{Dedicated to Professor Liangang Peng on the occasion of his 70th birthday}
\begin{abstract}

Tilting mutation provides a natural way to construct derived-equivalent algebras by replacing an indecomposable summand of a tilting object.  Holm, Schr\"oer and Zimmermann~\cite{HSZ01} (see also the appendix of~\cite{AA08}) have described combinatorial transformations of the quiver and relations of a gentle algebra that yield derived equivalences via tilting complexes. In this paper, we develop a self-contained module-theoretic framework for tilting mutation via generalized $\mathrm{BB}$-tilting modules.
For a gentle algebra $A=\bbk Q/\langle I\rangle$ and a vertex $k\in Q_0$, we first give a necessary and sufficient condition, expressed in terms of the arrows and relations incident with $k$, for the corresponding minimal left approximation of $P(k)$ to yield a tilting mutation. 
This criterion applies uniformly to vertices with or without loops. When the mutation exists, we determine the irreducible morphisms between the indecomposable summands of the mutated tilting module and use them to construct explicitly the Gabriel quiver and defining relations of its endomorphism algebra. In particular, we also provide a combinatorial mutation $(Q,I)\mapsto(Q',I')$ of gentle pairs such that $\mu_k^+(A)\cong \bbk Q'/\langle I'\rangle,$ so the resulting algebra is again gentle and derived equivalent to $A$. Finally, we  formulate the dual cotilting mutation and its combinatorial description via opposite gentle pairs.

\end{abstract}

\maketitle
\tableofcontents

\section{Introduction}
Tilting theory provides one of the fundamental mechanisms for constructing and studying derived equivalences between finite-dimensional algebras. Let $A$ be a basic finite-dimensional algebra and write
$A_A=P(k)\oplus\overline T$, where $P(k)$ is an indecomposable projective right $A$-module and $\overline T=\bigoplus_{\ell\neq k}P(\ell)$. Suppose that a minimal left $\add\overline T$-approximation of $P(k)$ fits into an exact sequence
\[
0\longrightarrow P(k)\longrightarrow M_k\longrightarrow P(k)^*\longrightarrow0,
\qquad M_k\in\add\overline T,
\]
and that $T=\overline T\oplus P(k)^*$ is a tilting $A$-module. In this situation we consider the algebra $\mu_k^+(A)=\End_A(T)$, which is derived equivalent to $A$. This leads naturally to two questions: when does such a tilting mutation exist, and how can the resulting endomorphism algebra be described explicitly from the original algebra?

The replacement of indecomposable summands of tilting objects has a long history. The origins of tilting mutation can be traced back to the reflection functors of Bernstein, Gelfand, and Ponomarev for hereditary algebras \cite{BGP},
which were subsequently generalized to arbitrary finite-dimensional algebras by Auslander, Platzeck, and Reiten~\cite{APR} through the construction of APR-tilting modules. Brenner and Butler~\cite{BB1980} further extended this framework by allowing the replacement of any indecomposable projective non-injective module, leading to the so-called BB-tilting modules. 
The theory of exchange complements  for almost complete tilting modules was developed further in \cite{ HU1989, RS91,CHU94,HU1998}. These ideas naturally evolved into the mutation theories of cluster-tilting and silting objects~\cite{BMRRT, BMR, BIRS, IY, IR, KY2011, AI12}, which now play a central role in modern representation theory and cluster algebras.

Explicit descriptions of endomorphism algebras under mutation have also been studied. Oppermann gave a quiver description for silting mutation in the dg-setting \cite{Opp17}, while Fosse developed a purely combinatorial procedure for tilting mutation of suitable algebras on loop-free vertices via tilting complexes \cite{Fosse}, and Aihara introduced mutations of SB quivers and used them to give a combinatorial description of tilting mutation for symmetric special biserial algebras \cite{A}. These developments provide the general mutation-theoretic background for the problem considered here.

For gentle algebras, several more specific results are closely related to the present construction. 
Gentle algebras arose naturally in the study of iterated tilted algebras~\cite{AS87}, and this class is closed under derived equivalence; see~\cite{S99, SZ03}. 
Holm, Schröer and Zimmermann~\cite{HSZ01} described combinatorial transformations of the quiver and relations of a gentle algebra that yield derived equivalences (see also the appendix of~\cite{AA08}). 
Avella-Alaminos and Geiss~\cite{AG08} showed that the number of arrows and the number of oriented cycles remain constant in each derived equivalence class, and introduced combinatorial invariants that are also derived invariants. 
Their framework implies that gentle algebras with two cycles split into ``nondegenerate'' and ``degenerate'' subclasses; Avella-Alaminos~\cite{AA08} and Bobi\'nski--Malicki~\cite{BM08} provided complete derived classifications for the nondegenerate case. 
In the loop-free setting, Bobi\'nski and Buan~\cite{BB12} gave an explicit combinatorial description of Brenner--Butler tilting for gentle quivers: under a local compatibility condition at the mutation vertex, a reflection of the bound quiver realizes the endomorphism algebra of the corresponding BB-tilting module. 
Ladkani~\cite{Ladkani11} further characterized positive and negative algebra mutations of loop-free gentle algebras in terms of maximal nonzero paths and related these to BB-tilting. 
More recently, Saleh~\cite{Saleh26} introduced an involutive mutation operation on admissible ideals, providing a combinatorial framework for producing new gentle bound quiver algebras, though the relation of this construction to derived equivalence remains open.

From a geometric perspective, Baur and Coelho Sim\~{o}es constructed a surface model for the module category of an arbitrary gentle algebra~\cite{BS21}. 
Building on the geometric model of the derived category developed in~\cite{OPS}, Chang and Schroll gave a geometric interpretation of silting mutation via transformations of graded arcs~\cite{CSc23}, and Chang subsequently studied complements and tilting completion for gentle algebras~\cite{Chang24}.

Continuing our study of tilting and derived structures of gentle algebras~\cite{DGL26}, which provided a geometric description of the endomorphism algebras of tilting modules, the present paper shifts focus to the local algebraic mechanics of mutation. 
We develop a self-contained module-theoretic framework for the tilting mutations introduced by Holm, Schr\"oer, and Zimmermann~\cite{HSZ01}, formulated via generalized $\mathrm{BB}$-tilting modules. 

While Ni~\cite[Lemma 5.3]{Ni} also considered this direction, his treatment is necessarily concise, as it serves primarily as a stepping stone for broader classification objectives. 
In contrast, our methodology is systematically constructive and homologically transparent. 
By introducing the combinatorial notion of \emph{$k$-paths}, we explicitly characterize the minimal left $\add(\overline{T})$-approximation $i \colon P(k) \to M_k$. Then based on  the associated short exact sequence, we determine all irreducible morphisms between the indecomposable summands of $T$, and recover the defining relations through explicit zero-composition patterns. 
Crucially, this path-theoretic procedure applies uniformly to vertices with loops as well as loop-free vertices.


Our first main result gives a complete existence criterion. Let $i:P(k)\to M_k$ be the minimal left $\add\overline T$-approximation determined by the relevant nonzero paths ending at $k$. Then $\overline T\oplus\cok i$ is a tilting module if and only if $i$ is a monomorphism. For a gentle algebra this condition admits the following purely local form: for every outgoing arrow $\beta$ at $k$, there exists an incoming arrow $\alpha$  at $k$ such that $\alpha\beta\notin I$, see Theorem~\ref{t: conditions for approximation is injective0}. Thus the existence of the tilting mutation can be decided directly from the arrows and relations incident with $k$. In particular, when $k$ carries a loop, the criterion reduces to the condition that exactly one non-loop arrow enters $k$.

Assume now that the mutation exists and write $T=\overline T\oplus P(k)^*$. We determine all irreducible morphisms between the indecomposable summands of $T$. These morphisms are described explicitly by paths in $Q$ and fall into four classes, according to whether the source or target is the exchanged summand $P(k)^*$. This description yields the Gabriel quiver $Q_T$ of $\End_A(T)$. We then determine all quadratic zero compositions (see Definition~\ref{d:quadratic relations}) among the corresponding irreducible morphisms and prove that they generate the defining ideal. Consequently, the mutated endomorphism algebra has an explicit bound-quiver presentation
\[
\mu_k^+(A)=\End_A(T)\cong \bbk Q_T/\langle I_\mb\rangle,
\]
where both $Q_T$ and $I_\mb$ are obtained directly from paths and relations in $(Q,I)$, see Theorem~\ref{t: qt-ib is a gentle pair}. In particular, the resulting algebra is again gentle. 

The intrinsic description of $Q_T$ and $I_\mb$ can be reformulated as a direct combinatorial operation on gentle pairs. We therefore obtain a mutation
\[
\mu_k^+:(Q,I)\longmapsto(Q',I')
\]
whose associated bound-quiver algebra is isomorphic to $\mu_k^+(A)$, see Theorem~\ref{t: tilitng mutation in general}. The construction is valid both at loop-free vertices and at vertices carrying loops, and we give explicit local forms according to whether the mutation vertex belongs to a loop and/or a $2$-cycle. 

Finally, we consider the dual cotilting construction. Passing to the opposite gentle pair transforms cotilting mutation into tilting mutation. This gives a combinatorial cotilting mutation
\[
\mu_k^-(Q,I)=\bigl(\mu_k^+(Q^{\operatorname{op}},I^{\operatorname{op}})\bigr)^{\operatorname{op}},
\]
and identifies the corresponding cotilting mutation algebra with the bound-quiver algebra determined by this pair. Hence the tilting mutation developed in the main part of the paper also provides, by duality, an explicit combinatorial description of cotilting mutation for gentle algebras.

The paper is organized as follows. Section~\ref{s:s2 pre} fixes conventions concerning path morphisms, relative irreducible morphisms, and tilting mutation. Section~\ref{s:s3 def of gentle} recalls gentle algebras and establishes the path-theoretic lemmas used throughout the paper. Section~\ref{s:s4 tilting mutation} determines the minimal left approximation of $P(k)$, proves the necessary and sufficient condition for the existence of tilting mutation, and constructs the exchange exact sequence. Section~\ref{s:s5 irreducible} determines the  irreducible morphisms and the relations  among the indecomposable summands of the mutated tilting module. Section~\ref{s:s6 endo} uses these morphisms to compute the Gabriel quiver and defining  all quadratic zero compositions of $\End_A(T)$. Section~\ref{s:s7 ex} reformulates this intrinsic description as an explicit combinatorial tilting mutation of gentle pairs and analyzes the relevant local configurations. Moreover, we provided some examples in this section. Finally, Section~\ref{s:s8 cotilting} introduces cotilting mutation and obtains its combinatorial description from tilting mutation of the opposite gentle pair.

\section{Preliminaries}\label{s:s2 pre}
\subsection{Conventions}

Throughout the paper, $\bbk$ denotes an algebraically closed field. By an algebra we mean a basic finite-dimensional $\bbk$-algebra. All modules are finitely generated right modules unless otherwise stated.
For a finitely generated $A$-module $M$, we denote by $|M|$ the number of isomorphism classes of indecomposable direct summands in a basic decomposition of $M$.

Let $A=\bbk Q/J$ be a basic finite-dimensional algebra, where $Q=(Q_0,Q_1)$ be a finite quiver and $J$ is an admissible ideal of $\bbk Q$.
For an arrow $\alpha\in Q_1$, we denote its source and target by $s(\alpha)$ and $t(\alpha)$, respectively. 
An arrow $\alpha$ is called a {\it loop} if $s(\alpha)=t(\alpha)$, else will be called a {\it non-loop} arrow.

Paths are read from left to right. Thus, if
\[
a\xrightarrow{\alpha}b\xrightarrow{\beta}c,
\]
then the corresponding path from $a$ to $c$ is denoted by $\alpha\beta$. We denote the source and target of a path $p$ by $s(p)$ and $t(p)$, respectively. A path $p$ is called  \emph{nonzero} if $p\notin J$.

Let $p=\alpha_1\cdots\alpha_s$ be a  path with $\alpha_i\in Q_1$ for each $1\leq i\leq s$. $s$ is called the length of $p$. $p$ is called a trivial path if $s=0$. For $s>0$,
 we denote its initial and terminal arrows by $\Ini(p)$ and $\Ter(p)$, respectively.    Let $v\in Q_0$. We call $v$ an \emph{internal vertex} of $p$ if there exist paths $p_1:s(p)\leadsto v$ and $p_2:v\leadsto t(p)$, both of positive length, such that $p=p_1p_2$.
In particular, for $a,b\in Q_0$, we denote by  \begin{itemize}
    \item $\mathcal{P}_k(a,b)$: the set of all nonzero paths of positive length from $a$ to $b$ whose internal vertices, if any, are all equal to $k$.
\end{itemize}
This set will play a central role in our combinatorial description of the mutation at $k$.

For $k\in Q_0$, we use the following notation:
\begin{itemize}
  \item $\Ain(k)$ is the set of incoming arrows at $k$;
  \item $\Aoin(k)$ is the set of non-loop incoming arrows at $k$;
  \item $\Aout(k)$ is the set of outgoing arrows at $k$;
  \item $\Aoout(k)$ is the set of non-loop outgoing arrows at $k$;
  \item $\din(k)$ is the number of incoming arrows at $k$;
  \item $\dout(k)$ is the number of outgoing arrows at $k$;
  \item $N^-(k):=\{s(\alpha)\mid \alpha\in\Aoin(k)\}$ is the set of direct predecessors of $k$ distinct from $k$.
  \item $N^+(k):=\{t(\beta)\mid \beta\in\Aoout(k)\}$ is the set of direct successors of $k$ distinct from $k$.
\end{itemize}

\subsection{Path morphisms}
\label{ss:path-morphisms}
 Let $\{e_i\mid i\in Q_0\}$ be the complete set of primitive orthogonal idempotents corresponding to the vertices of $Q$. For $i\in Q_0$, set $P(i):=e_iA$. Thus $P(i)$ is the indecomposable projective right $A$-module corresponding to $i$. 

For $i,j\in Q_0$, there is a natural $\bbk$-linear isomorphism
\[
\Hom_A(P(i),P(j)) \cong e_jAe_i.
\]
Consequently, every nonzero path $p:j\leadsto i$ determines an $A$-module homomorphism
\[
\varphi_p:P(i)\longrightarrow P(j)
\]
given by left multiplication: $\varphi_p(x)=px$ for $x\in e_iA$. We refer to $\varphi_p$ as the \emph{path morphism induced by $p$}.

The correspondence between paths and path morphisms is compatible with composition. More precisely, suppose that $p_1:j\leadsto i$ and $p_2:\ell\leadsto j$ are paths. Then
\[
\varphi_{p_2p_1}=\varphi_{p_2}\varphi_{p_1}=\varphi_{p_2}\circ\varphi_{p_1}.
\]
Thus, if a path $p$ admits a factorization $p=qr$, then
$\varphi_p=\varphi_q\varphi_r.$
Here the composition of morphisms denotes ordinary composition of maps. 

\subsection{Relative irreducible morphisms}
\label{ss:relative-irreducible-morphisms}

Let $M,N$ be finitely generated $A$-modules. We denote by $\rad_A(M,N)$ the radical of $\Hom_A(M,N)$.

Let $T$ be a finitely generated $A$-module. Define
\[
\rad_T^2(M,N)
:=
\sum_{X\in\ind(\add T)}
\rad_A(X,N)\circ\rad_A(M,X).
\]
Equivalently, $\rad_T^2(M,N)$ is the subspace of $\rad_A(M,N)$ spanned by all compositions $M\xrightarrow{f}X\xrightarrow{g}N$, where $X\in\add T$ and both $f$ and $g$ are radical morphisms.

The space of irreducible morphisms relative to $T$ is defined by
\[
\Irr_T(M,N)
:=
\rad_A(M,N)/\rad_T^2(M,N).
\]
For convenience, we also write
\[
\Irr_T^\circ(M,N)
:=
\rad_A(M,N)\setminus\rad_T^2(M,N).
\]
Thus $\Irr_T^\circ(M,N)$ is a set of morphisms rather than a vector space. A morphism belonging to $\Irr_T^\circ(M,N)$ will be called \emph{$T$-irreducible}.
For a morphism $f\in\rad_A(M,N)$, we denote its residue class in $\Irr_T(M,N)$ by $\overline f$.

Let $\mathcal P_A:=\add A$ be the category of finitely generated projective $A$-modules. For projective modules $P$ and $Q$, define
\[
\rad_{\mathcal P_A}^2(P,Q)
:=
\sum_{X\in\operatorname{ind}(\mathcal P_A)}
\rad_A(X,Q)\circ\rad_A(P,X),
\]
and
\[
\Irr_A(P,Q)
:=
\rad_A(P,Q)/\rad_{\mathcal P_A}^2(P,Q).
\]
Similarly, we write
\[
\Irr_A^\circ(P,Q)
:=
\rad_A(P,Q)\setminus\rad_{\mathcal P_A}^2(P,Q).
\]

For vertices $i,j\in Q_0$, the space $\Irr_A(P(i),P(j))$ is naturally identified with the corresponding arrow space of the Gabriel quiver of $A$. In particular, an arrow $\alpha:j\longrightarrow i$ corresponds to the irreducible morphism $\varphi_\alpha:P(i)\longrightarrow P(j)$.

\subsection{Tilting mutation}
\label{ss:tilting-mutation}

Recall that a finitely generated $A$-module $T$ is a \emph{tilting module} if
\begin{enumerate}
    \item $\operatorname{pd}_A T\leq1$;
    \item $\Ext_A^1(T,T)=0$;
    \item there exists an exact sequence
    \[
    0\longrightarrow A
    \longrightarrow T_0
    \longrightarrow T_1
    \longrightarrow0
    \]
    with $T_0,T_1\in\add T$.
\end{enumerate}
For a basic finite-dimensional algebra $A$, assuming conditions~(1) and~(2), condition~(3) is equivalent to $|T|=|A|$.

\begin{definition}\label{def:tilting-mutation}
Let $k\in Q_0$, set $P(k)=e_kA$, and write
\[
A_A=P(k)\oplus\overline T,
\qquad
\overline T:=\bigoplus_{\ell\neq k}P(\ell).
\]
Suppose that there is an exact sequence
\[
0\longrightarrow P(k)
\xrightarrow{\,f\,}
M_k
\longrightarrow P(k)^*
\longrightarrow0,
\]
where $f$ is a minimal left $\add\overline T$-approximation,
$M_k\in\add\overline T$, and
$T:=\overline T\oplus P(k)^*$ is a basic tilting $A$-module. We define
\[
\mu_k^+(A):=\End_A(T)
\]
and call it the \emph{tilting mutation of $A$ at $k$}. The module $P(k)^*$ is called the \emph{exchange complement} of $P(k)$.
\end{definition}


\begin{remark}\label{r: pk* is tauni Sk}
   By  \cite[Proposition 7.4]{KY2014},  $P(k)^*\cong \tauni S_k^+$, where $S_k^+=D(A/A(1-e_k)A)$. When  there are no loops in the quiver of $Q$ at the $k$, this specializes to the classical $\mathrm{BB}$-tilting module \cite{BB1980} and $\mathrm{APR}$-tilting module \cite{APR}. In view of this structural correspondence, we refer to $T$ as a \emph{generalized $\mathrm{BB}$-tilting module}. 
\end{remark}


\section{Gentle algebras}\label{s:s3 def of gentle}

\subsection{Gentle algebras}
A finite-dimensional $\bbk$-algebra $A$ is \emph{gentle} if it admits a presentation $A=\bbk Q/\langle I\rangle$ satisfying the following conditions:
\begin{itemize} 
    \item[(G1)] Each vertex of $Q$ is the source of at most two arrows and the target of at most two arrows;
    \item[(G2)] For each arrow $\alpha$, there is at most one arrow $\beta$ such that $\alpha\beta\in I$, and at most one arrow $\gamma$ such that $\gamma\alpha\in I$;
    \item[(G3)] For each arrow $\alpha$, there is at most one arrow $\beta$ such that $\alpha\beta\notin I$, and at most one arrow $\gamma$ such that $\gamma\alpha\notin I$;
    \item[(G4)] $I$ is generated by a finite set of paths of length two.
\end{itemize}
The pair $(Q,I)$ is called a \emph{gentle pair}.

If $\rho$ is a loop at a vertex $k$, since $A$ is finite dimensional,  (G4) implies $\rho^2\in  I$. Moreover, for every non-loop incoming arrow $\alpha$ and every non-loop outgoing arrow $\beta$ at $k$, the gentle conditions imply
\[
\alpha\rho\notin  I,
\qquad
\rho\beta\notin I,
\qquad
\alpha\beta\in I.
\]

\subsection{Some elementary properties of paths}
Let $A=KQ/\langle I\rangle$ be a gentle algebra. We begin with several elementary properties of paths that will be used repeatedly below. First, we have the following known result.
\begin{lem}\label{l: p_1+p_2 in I infer pi in I}
    Let $a,b\in Q_0$, and let $p_1,\ldots,p_t$ be pairwise distinct paths in $Q$ from $a$ to $b$. Let $\lambda_1,\ldots,\lambda_t\in\bbk^\times$. If
\[
\lambda_1p_1+\cdots+\lambda_tp_t\in\langle I\rangle,
\]
equivalently, if \[
\lambda_1\varphi_{p_1}+\cdots+\lambda_t\varphi_{p_t}=0,
\]
then $p_i\in\langle I\rangle$ for every $1\leq i\leq t$.
   \end{lem}
It follows that if $p,~q$ are nonzero paths representing elements of $e_aAe_b$ such that $\varphi_p=\varphi_q\neq 0$, then $p=q$.

\begin{definition}
    Let $p$ and $q$ be two nonzero paths in $Q$. 
    \begin{enumerate}
        \item We say that \emph{$q$ ends with $p$} if there exists a path $w$ such that $q = w p$.
        \item We say that \emph{$q$ starts with $p$}  if there exists a path $w$ such that $q = p w$.
    \end{enumerate}
\end{definition}

\begin{lem}\label{l:p is the shortest}
    Let $a,b,k\in Q_0$ with $k\notin\{a,b\}$, let $\alpha:b\to k$ be an arrow, and let $p,q$ be nonzero paths from $a$ to $b$ such that $$p\alpha, q\alpha\in\langle I\rangle.$$ Then
    \begin{enumerate}
        \item either $p$ ends with $q$ or $q$ ends with $p$;
        \item if $p\in\mathcal{P}_k(a,b)$, then $q$ ends with $p$.
    \end{enumerate}
\end{lem}

\begin{proof}
Let $\beta_p$ and $\beta_q$ be the terminal arrows of $p$ and $q$. Since $p$ and $q$ are nonzero, they contain no subpath belonging to $I$. Since $I$ is generated by a finite set of paths of length two, the relations $p\alpha,q\alpha\in\langle I\rangle$ imply
\[
\beta_p\alpha,\beta_q\alpha\in I.
\]
By~(G2), we have $\beta_p=\beta_q$.

Write $p=p'\beta$ and $q=q'\beta$ for this common terminal arrow $\beta$. If both $p'$ and $q'$ have positive length, let $\delta_p$ and $\delta_q$ be their terminal arrows. Since $p$ and $q$ are nonzero, neither $\delta_p\beta$ nor $\delta_q\beta$ belongs to $\langle I\rangle$. By~(G3), we have $\delta_p=\delta_q$.

Repeating this argument backwards, the two paths agree from the right until one of them is exhausted. Hence one ends with of the other.

For~(2), suppose that $p\in\mathcal P_k(a,b)$. By~(1), either $q=wp$ or $p=wq$. If $p=wq$ with $w$ nontrivial, since $p$ and $q$ have the same source $a$, the path $w$ is a nontrivial path from $a$ to $a$. Thus $a$ occurs as an internal vertex of $p$, contradicting $p\in\mathcal P_k(a,b)$ because $a\neq k$. Hence $q=wp$, in particular, $q$ ends with $p$.
\end{proof}


\section{Existence of tilting mutation}\label{s:s4 tilting mutation}
\label{sec:tilting-mutation}
Throughout this section, assume that $Q$ is connected and that $|Q_0|\ge 2$.
Let $k\in Q_0$ and set $$\overline{T}=\bigoplus\limits_{l\neq k}P(l).$$
We first determine the minimal left $\add\overline T$-approximation of $P(k)$ and then characterize when it is a monomorphism.
 \subsection{The minimal left approximation}
For $a,k\in Q_0$, recall that 
\begin{itemize}
    \item $\mathcal{P}_k(a,k)$ is the set of all nonzero paths with length at least one from $a$ to $k$ that have no internal vertices other than possibly $k$.
\end{itemize}
 The following lemma describes the minimal left $\add\overline T$-approximation of $P(k)$.
\begin{lem}\label{lem:min-left-approx-form}
\begin{enumerate}
    \item For each $a\neq k$,  the residue classes
    $$\left\{
\overline{\varphi_p}
\ \middle|\
p\in\mathcal P_k(a,k)
\right\}$$
form a $\bbk$-basis of $\Irr_{\overline T}(P(k),P(a))$.
\item Let
$\bigcup_{a\in N^-(k)}\mathcal P_k(a,k)=\{p_1,\ldots,p_t\},
\quad
a_j:=s(p_j).$
Then $$i:=(\varphi_{p_j})_{j=1}^t:
P(k)\longrightarrow
M_k:=\bigoplus_{j=1}^tP(a_j)$$
is the minimal left $\add\overline{T}$-approximation of $P(k)$. Moreover, $t\leq2$.
\end{enumerate}

\end{lem}

\begin{proof}
For (1), $\Hom_A(P(k),P(a))$ is spanned by all nonzero paths from $a$ to $k$.  If a nonzero path $p:a\leadsto k$ does not belong to $\mathcal{P}_k(a,k)$, then $p$ has an internal vertex $v\neq k$, and therefore $\varphi_p\in\rad^2_{\overline{T}}(P(k),P(a))$. It follows that the residue classes
$\left\{
\overline{\varphi_p}
\ \middle|\
p\in\mathcal P_k(a,k)
\right\}$
span $\Irr_{\overline T}(P(k),P(a))$. Their linear independence follows from Lemma~\ref{l: p_1+p_2 in I infer pi in I}. Thus they form a $\bbk$-basis.

For (2), the fact that $i$ is a minimal $\add\overline{T}$-approximation can be checked directly.

 It remains to show that $t\leq2$. Suppose first that there is no loop at $k$. Then a path in $\mathcal P_k(a,k)$ cannot have an internal vertex and therefore must be an  arrow at $k$. Hence $t\leq2$ by~(G1).

Now suppose that there is a loop $\rho$ at $k$, so $\rho^2\in I$. By~(G1), there is at most one non-loop incoming arrow $\alpha:a\to k$. If such an arrow exists, the only possible paths occurring in the above union are $\alpha$ and $\alpha\rho$. Hence again $t\leq2$.
\end{proof}

    

\subsection{Criteria for the existence of tilting mutation}


\begin{thm}\label{t: conditions for approximation is injective0}
   Let $i: P(k)\ra M_k$ be the minimal left $\add \overline{T}$-approximation.
   Then the following statements are equivalent:
\begin{enumerate}
    \item $\overline{T}\oplus \mathbf{coker} i$ is a tilting $A$-module;
    \item $i$ is a monomorphism;
     \item for every $\beta\in\Aout(k)$, there exists $\alpha\in\Ain(k)$ such that $\alpha\beta\notin I$. 
\end{enumerate}
In particular, if there is a loop at $k$, then $\overline{T}\oplus \mathbf{coker} i$ is a tilting $A$-module if and only if $\din(k)=2$.
\end{thm}

\begin{proof}
The equivalence (1) $\Leftrightarrow$ (2) is well known. It remains to prove the equivalence of~(2) and~(3).
The module $P(k)$ has a $\bbk$-basis consisting of the nonzero paths starting at $k$. Hence $i$ is a monomorphism if and only if $i(p) \neq 0$ for every nonzero path $p \in e_k A$.

$(3) \Rightarrow (2)$. We first note that $\Ain(k)\neq\varnothing$. Indeed, if $k$ has an outgoing arrow, condition~(3) forces the existence of an incoming arrow. If $k$ has no outgoing arrows, then connectedness of $Q$ and the assumption $|Q_0|\geq2$ imply that $k$ has an incoming arrow.

Suppose first that there is a loop $\rho$ at $k$. Since $\rho^2\in I$, by assumption, we obtain a non-loop incoming arrow $\alpha:a\to k$ such that $\alpha\rho\notin I$.
 Then 
$i$ is $$ P(k) \xra{i={\scriptstyle\left(\begin{smallmatrix} i_1 \\ i_2 \end{smallmatrix}\right)}}P(a)\oplus P(a)$$ 
where $i_1=\varphi_{\alpha}$ and $i_2=\varphi_{\alpha}\varphi_{\rho}=\varphi_{\alpha\rho}$. Let $0\neq p\in e_kA$. If $p$ begins with $\rho$, then $p=\rho\omega$ for some nonzero path $\omega$, and hence $i_1(p)=\varphi_{\alpha}(p)=\alpha p=\alpha\rho\omega\neq 0$. If $p$ does not begin with $\rho$, including the case $p=e_k$, then $i_2(p)=\alpha\rho p\neq 0$. Thus $i(p)\neq0$ for every nonzero $p\in e_kA$, and hence $i$ is injective.

Now suppose that there is no loop at $k$.
Then $$i = (\varphi_\alpha)_{\alpha\in\Ain(k)} : P(k) \to M_k.$$ Since the set $\Ain(k)$ is nonempty, so $i\neq 0$. Let $p \in e_k A$ be a nonzero path. If $p = e_k$, then $i(e_k) = (\alpha)_{\alpha\in\Aoin(k)} \neq 0$. If $p$ has positive length, write $p = \beta p_1$ where $\beta=\Ini(p)$. By condition (3), there exists an incoming arrow $\alpha$ at $k$ such that $\alpha\beta \notin I$. Consequently, $\alpha p = \alpha\beta p_1 \neq 0$ in $A$, which implies the $\alpha$-component of $i(p)$ is nonzero. Therefore $i(p)\neq0$ for every nonzero $p\in e_kA$, so $i$ is injective.

$(2) \Rightarrow (3)$.
Suppose that condition (3) fails. Then there exists an outgoing arrow $\beta: k \to j$ such that $\alpha\beta \in I$ for every incoming arrow $\alpha: a \to k$. Hence $$i(\beta)=(\varphi_\alpha(\beta))_{\alpha\in\Aoin(k)}=(\alpha\beta)_{\alpha\in\Aoin(k)}=0,$$
 contradicting the injectivity of $i$. If there are no incoming arrows at $k$, then $i = 0$, which is likewise not injective. Therefore, $(2)$ implies $(3)$.

 When there is a loop $\rho$ at $k$, the condition (3), applied to $\beta = \rho$, gives an incoming arrow $\alpha$ such $\alpha\rho \notin I$. Since $\rho\rho = \rho^2 \in I$, we have $\alpha\ne\rho$. This yields $\din(k) = 2$, completing the proof.
\end{proof}
\begin{remark}
    If there is no loop at $k$, condition (3)  coincides with the local compatibility condition given by Bobi\'nski and Buan in \cite{BB12}. Moreover, in this case, $\overline{T}\oplus \mathbf{coker} i$ is a BB-tilting $A$-module.
\end{remark}

\subsection{The exchange exact sequence} \label{ss: exchange exact sequence}
\begin{definition}
Whenever the equivalent conditions of Theorem~\ref{t: conditions for approximation is injective0} hold, set
\[
P(k)^*:=\cok i,
\qquad
T:=\overline T\oplus P(k)^*.
\]
Then the minimal left approximation fits into the exact sequence 
\begin{equation}
0\longrightarrow P(k)
\xrightarrow{\,i\,}
M_k
\xrightarrow{\,\pi\,}
P(k)^*
\longrightarrow0.
\tag{$\bullet$}\label{eq:exchange-general}
\end{equation}
We call this the \emph{exchange exact sequence at $k$}.
\end{definition}
\noindent By the standard exchange theory of almost complete tilting modules, $P(k)^*$ is indecomposable and $\pi$ is a minimal right $\add\overline T$-approximation; see \cite{HU1989, CHU94, HU1998}. In particular, 
by Remark~\ref{r: pk* is tauni Sk}, we know $P(k)^*\cong \tauni S_k^+$, where $S_k^+=D(A/A(1-e_k)A)$. Thus,
$T$ is a basic generalized $\mathrm{BB}$-tilting module and
\[
\mu_k^+(A)=\End_A(T).
\]

In what follows, if there is no loop at $k$, we fix all the incoming arrows at $k$ by
\[
\alpha_j:a_j\to k,
\qquad
1\leq j\leq \din(k).
\]
 Then the exchange exact sequence at $k$ is 
\begin{align}\label{g: exact seq for case without loop}
 0\ra P(k) \xra{i}\bigoplus\limits_{1\leq j\leq \din(k)}P(a_j)\xra{\pi} P(k)^*\ra 0,   
\end{align}
where $i=(i_j)_{1\leq j\leq \din(k)}=(\varphi_{\alpha_j})_{\alpha_j:a_j\ra k}$, $\pi=(\pi_j)_{1\leq j\leq \din(k)}$.  

If $k$ carries a loop $\rho$. Then $\din(k)=2$. 
We fix the other incoming arrow at $k$ by $\alpha: a\ra k$. 
 Then the exchange exact sequence at $k$ is
\begin{align}\label{g: exchange exact seq for the case with loop}
 0\ra P(k) \xra{i={\scriptstyle\left(\begin{smallmatrix} i_1 \\ i_2 \end{smallmatrix}\right)}}P(a)\oplus P(a)\xra{\pi=(\pi_1,\pi_2)} P(k)^*\ra 0,   
\end{align}
where $i_1=\varphi_\alpha$ and
$i_2=i_1\varphi_{\rho}=\varphi_{\alpha}\varphi_{\rho}=\varphi_{\alpha\rho}$.

\section{Irreducible morphisms under tilting mutation}\label{s:s5 irreducible}
Let $A=\bbk Q/\langle I\rangle$ be a gentle algebra, and let $k\in Q_0$ be a vertex at which the tilting mutation exists.
We retain the notation
\[
\overline T=\bigoplus_{\ell\neq k}P(\ell),
\qquad
T=\overline T\oplus P(k)^*,
\]
and the exchange exact sequences in Subsection~\ref{ss: exchange exact sequence}. Having established the existence criterion, we now determine the tilting mutation $\mu_k^+(A)$ explicitly. Using the exchange exact sequence, we describe the irreducible morphisms and the relations between the indecomposable summands of $T$.

\subsection{A criterion for T-irreducibility}
In order to get $T$-irreducible morphisms, we need the following results.
First, by the corresponding definitions, it follows
\begin{lem}
   Let $c, d \in Q_0$ with $k \notin \{c, d\}$ and $p$ be a nonzero path from $c$ to $d$.  Then 
   $$\varphi_p \in \operatorname{Irr}_{\overline{T}}^\circ(P(d), P(c))\iff p\in\mpik(c,d).$$
\end{lem}
        
The following lemma  tells us a path morphism $\varphi_p$ for some path $p$ is $T$-irreducible if and only if $\varphi_p$ is $\overline{T}$-irreducible and $\varphi_p$ is $P(k)^*$-irreducible.
\begin{lemma}\label{l: only need to check whether pass through P_k*}
    Let $c, d \in Q_0$ with $k \notin \{c, d\}$, and let $p$ be a nonzero path from $c$ to $d$. 
    Then $$\varphi_p \in \operatorname{Irr}_T^\circ(P(d), P(c))$$ if and only if the following two conditions hold:
    \begin{enumerate}
        \item[(1)] $\varphi_p \in \operatorname{Irr}_{\overline{T}}^\circ(P(d), P(c))$;
        
        \item[(2)] $\varphi_p \in \operatorname{Irr}^\circ_{P(k)^*}(P(d), P(c))$.
    \end{enumerate}
\end{lemma}

\begin{proof}
 If $\varphi_p\in\Irr_T^\circ(P(d),P(c))$, then both conditions hold because $T=\overline T\oplus P(k)^*$.
 
 On the other hand, suppose that (1) and (2) hold, if  $\varphi_{p}\notin \IrrTo(P(d),P(c))$, so $\varphi_{p}\in \rad_{{T}}^2(P(d),P(c))$, then $$\varphi_{p}=g_2g_1+h_2h_1$$ for $g_1\in\rad_{A}(P(d),\overline{T}),g_2\in \rad_{A}(\overline{T}, P(c)), h_1\in \rad_{A}(P(d),P(k)^*)$ and $h_2\in \rad_{A}(P(k)^*,P(c))$. Moreover, we can suppose that $h_2h_1\notin \rad_{\overline{T}}^2(P(d),P(c))$.
Since $p$ is a single path from $c$ to $d$ and the path morphisms are linearly independent in $\operatorname{Hom}_A(P(d), P(c))$ (cf. Lemma~\ref{l: p_1+p_2 in I infer pi in I}), either  $g_2g_1=0$ or $h_2h_1=0$. Then $\varphi_{p}=h_2h_1$ or $\varphi_{p}=g_2g_1$ respectively, a contradiction with (2) or (1). 
\end{proof}

\subsection{The morphisms incident to $P(k)^*$}
Consider the exchange exact sequences (\ref{g: exact seq for case without loop}) and (\ref{g: exchange exact seq for the case with loop}) at $k$,  we introduce the following notations to denote the  morphisms of $P(k)^*$.
\begin{definition} Let $\alpha$ be a non-loop incoming arrow at $k$.
If there is no loop at $k$ and $\alpha=\alpha_j$ for some $1\leq j\leq \din(k)$,     set $$\varphi_{\alpha}^*:=\pi_j.$$ 
   If there is a loop at $k$, set $$\varphi_{\alpha}^*:=\pi_2.$$
\end{definition}
\begin{definition}
    Suppose that $k$ carries a loop $\rho$. Consider the exchange exact sequence (\ref{g: exchange exact seq for the case with loop}).  
It is clear that  $i_2\varphi_{\rho}=\varphi_{\alpha\rho}\varphi_{\rho}=\varphi_{\alpha\rho^2}=0$, so $${\scriptstyle\left(\begin{smallmatrix} 0&1 \\ 0&0\end{smallmatrix}\right)}{\scriptstyle\left(\begin{smallmatrix} i_1 \\ i_2 \end{smallmatrix}\right)}={\scriptstyle\left(\begin{smallmatrix} i_2 \\ 0 \end{smallmatrix}\right)}={\scriptstyle\left(\begin{smallmatrix} i_1\varphi_{\rho} \\ 0 \end{smallmatrix}\right)}={\scriptstyle\left(\begin{smallmatrix} i_1 \\ i_2 \end{smallmatrix}\right)}\varphi_{\rho}.$$
Take 
    $g={\scriptstyle\left(\begin{smallmatrix} 0&1 \\ 0&0\end{smallmatrix}\right)}$.
 Then there exists a unique morphism $f\in\Hom_A(P(k)^*,P(k)^*)$ such that the following diagram commutes:\[\xymatrix{
0\ar[r]&P(k)\ar[d]^{\varphi_{\rho}}\ar[r]^{{\scriptstyle\left(\begin{smallmatrix} i_1 \\ i_2 \end{smallmatrix}\right)}\ \ \ \ }&P(a)\oplus P(a)\ar[r]^{\ \ \ \ (\pi_1,\pi_2)}\ar[d]_{g}&P(k)^*\ar[r]\ar[d]^{f}&0\\
0\ar[r]&P(k)\ar[r]^{{\scriptstyle\left(\begin{smallmatrix} i_1 \\ i_2 \end{smallmatrix}\right)}\ \ \ \ }&P(a)\oplus P(a)\ar[r]^{\ \ \ \ (\pi_1,\pi_2)}&P(k)^*\ar[r]&0}
\]
 We set 
 $$\varphi_{\rho}^*:=f.$$
\end{definition}

\begin{definition}
    Let $\alpha \colon a \to k$ be a non-loop incoming arrow, and let $p \in \mathcal{P}_k(c,a)$ be a path with $c\neq k$ such that $p\alpha \in \langle I \rangle$. Then by Proposition~\ref{p: path ka is T irreducible}, the morphism $\varphi_p$ factors uniquely through the cokernel of $P(k)^*$. 
    We denote this unique factorization morphism by $$\varphi_{[p\alpha]} \in \Hom_A(P(k)^*, P(c)).$$ 
\end{definition}
As will be verified later, $\varphi_{[p\alpha]}$ is irreducible, i.e., $\varphi_{[p\alpha]} \in \IrrTo(P(k)^*, P(c))$.
Geometrically, $\varphi_{[p\alpha]}$ corresponds to the new arrow $[p\alpha] \colon c \to k^*$ in the mutated quiver $Q_T$.

\subsection{T-irreducible morphisms  from $P(k)^*$ to  $P(k)^*$}
\begin{prop}\label{p: relation if exist loop}
    $\dim\IrrT(P(k)^*,P(k)^*)=
        \dim\IrrA(P(k),P(k))$. 
        Furthermore, if there is a loop $\rho$ at $k$,  then 
        \begin{enumerate}
            \item  $\varphi_{\rho}^*\in \IrrTo(P(k)^*,P(k)^*)$ and $\{\overline{\varphi_{\rho}^*}\}$  is a basis of $\IrrT(P(k)^*,P(k)^*)$;
            \item $\pi_1=\varphi_{\rho}^*\pi_2=\varphi_{\rho}^*\varphi_{\alpha}^*,~ \varphi_{\rho}^*\varphi_{\rho}^*=0$;
            \item $\varphi_{\alpha}^*=\pi_2\in\IrrTo(P(a),P(k)^*).$
        \end{enumerate}
     \end{prop}
\begin{proof}
Let $\delta\in\rad_A(P(k),P(k))$. Since $i$ is a minimal left $\add\overline{T}$-approximation, one can find morphisms $g$ and $\sigma$ such that the following diagram commutes:
\[\xymatrix{
0\ar[r]&P(k)\ar[d]^{\delta}\ar[r]^{i }&M_k\ar[r]^{ \pi}\ar[d]_{g}&P(k)^*\ar[r]\ar[d]^{\sigma}&0\\
0\ar[r]&P(k)\ar[r]^{i }&M_k\ar[r]^{ \pi}&P(k)^*\ar[r]&0}
\]
Since $P(k)$ and $P(k)^*$ are indecomposable and $\pi$ is a minimal right $\add\overline T$-approximation of $P(k)^*$, it follows that $\sigma \in \rad_A(P(k)^*,P(k)^*)$ (otherwise $\sigma$ is an isomorphism, then $g$ is an isomorphism since $\pi$ is right minimal,  which forces $\delta$ to be an isomorphism, a contradiction).
Similarly, for any $\sigma \in \rad_T(P(k)^*,P(k)^*)$, there exist $g$ and $\delta \in \rad_A(P(k),P(k))$ completing the diagram.
Thus one can infer that $\delta\in \rad_{P(k)}^2(P(k),P(k))$ if and only if $\sigma\in \rad_{P(k)^*}^2(P(k)^*,P(k)^*)$.

Next, we show $\delta \in \rad_{\mathcal{P}_A}^2(P(k),P(k)) \iff \sigma \in \rad_T^2(P(k)^*,P(k)^*)$.
Let $\sigma \in \rad_T^2(P(k)^*,P(k)^*)$,
 since $T=P(k)^*\oplus \overline{T}$, there is $M\oplus P_k^*$ and $f_1\in\rad(P(k)^*,M),f_2\in\rad(M,P(k)^*)$, $g_1\in\rad(P(k)^*,P_k^*), g_2\in\rad(P_k^*,P(k)^*)$ such that $$\sigma=f_2f_1+g_2g_1$$ where $M\in\add\overline{T}$, $P_k^*\in\add P(k)^*$. Since  $M$ is projective, there is 
$f_3\in\Hom_A(M,M_k)$ such that $$f_2=\pi f_3$$ and  there is $g'\in \Hom_A(M_k,M_k)$ such that $$\pi g'=g_2g_1\pi.$$
Then there is $\delta'\in \Hom_A(P(k),P(k))$ such that $$i\delta'=g'i.$$ Because $g_2g_1\in \rad_{P(k)^*}^2(P(k)^*,P(k)^*)$, by the discussion above, we know $$\delta'\in \rad_{P(k)}^2(P(k),P(k)).$$
Moreover, by
$$\pi\circ (g-g'-f_3f_1\pi)=\pi g-\pi g'-\pi f_3f_1\pi=\pi g-g_2g_1\pi-f_2f_1\pi=\pi g-\sigma\pi=0,$$  there is $h\in\Hom_A(M_k, P(k))$ such that $ih=g-g'-f_3f_1\pi$. Then we have 
$$i(\delta-\delta'-hi)=i\delta-i\delta'-ihi=i\delta-i\delta'-(g-g'-f_3f_1\pi)i=i\delta-i\delta'-gi+g'i+f_3f_1\pi i=0.$$
Because $i$ is a monomorphism, we have $$\delta=\delta'+hi\in \rad_{\mathcal{P}_A}^2(P(k),P(k)).$$
\[\xymatrix{
0\ar[r]&P(k)\ar@/_4pt/[dd]_{\delta}\ar@/^4pt/[dd]^{\delta'}\ar[rr]^{i}&&M_k\ar@{.>}[ddll]_{h}\ar[rr]^{\pi}\ar@/_4pt/[dd]_{g}\ar@/^4pt/[dd]^{g'}&&P(k)^*\ar[dl]^{{\scriptstyle\left(\begin{smallmatrix} f_1 \\ g_1 \end{smallmatrix}\right)}}\ar[r]\ar[dd]^{\sigma}&0\\
&&&&M\oplus P_k^* \ar[dr]^{(f_2,g_2)}\ar[dl]^{f_3}&&\\
0\ar[r]&P(k)\ar[rr]^{i}&&M_k\ar[rr]^{\pi}&&P(k)^*\ar[r]&0}
\]
 The converse follows similarly.

Consequently, the correspondence $\delta \mapsto \sigma$ induces a well-defined linear isomorphism
\[
\rad_A(P(k),P(k))/\rad_{\mathcal{P}_A}^2(P(k),P(k)) \xrightarrow{\sim} \rad_T(P(k)^*,P(k)^*)/\rad_T^2(P(k)^*,P(k)^*),
\]
which implies $\dim\IrrA(P(k),P(k)) = \dim\IrrT(P(k)^*,P(k)^*)$.

Now assume there is a loop $\rho$ at $k$, 
so  $\dim\IrrT(P(k)^*,P(k)^*)=
        \dim\IrrA(P(k),P(k))=1$.
        By the definition of $\varphi_\rho^*$, there is a commutative diagram
   \[\xymatrix{
0\ar[r]&P(k)\ar[d]^{\varphi_{\rho}}\ar[r]^{{\scriptstyle\left(\begin{smallmatrix} i_1 \\ i_2 \end{smallmatrix}\right)}\ \ \ \ }&P(a)\oplus P(a)\ar[r]^{\ \ \ \ (\pi_1,\pi_2)}\ar[d]_{{\scriptstyle\left(\begin{smallmatrix} 0&1 \\ 0&0\end{smallmatrix}\right)}}&P(k)^*\ar[r]\ar[d]^{\varphi_{\rho}^*}&0\\
0\ar[r]&P(k)\ar[r]^{{\scriptstyle\left(\begin{smallmatrix} i_1 \\ i_2 \end{smallmatrix}\right)}\ \ \ \ }&P(a)\oplus P(a)\ar[r]^{\ \ \ \ (\pi_1,\pi_2)}&P(k)^*\ar[r]&0.}
\]

For (1), since $\varphi_\rho\in\Irr_A(P(k),P(k))$, the preceding argument shows that $\varphi_{\rho}^*\in \IrrTo(P(k)^*,P(k)^*)$ and $\{\overline{\varphi_{\rho}^*}\}$  is a basis of $\IrrT(P(k)^*,P(k)^*)$ since $\dim\IrrT(P(k)^*,P(k)^*)=1$.

For (2),  by $$(\varphi_{\rho}^*\pi_1,\varphi_{\rho}^*\pi_2)=\varphi_{\rho}^*(\pi_1,\pi_2)=(\pi_1,\pi_2){\scriptstyle\left(\begin{smallmatrix} 0&1 \\ 0&0\end{smallmatrix}\right)}=(0,\pi_1),$$
We obtain $\pi_1=\varphi_{\rho}^*\pi_2$ and $\varphi_{\rho}^*\pi_1=0$. 
Hence $$(\varphi_{\rho}^*)^2\pi = \big((\varphi_{\rho}^*)^2\pi_1,\, (\varphi_{\rho}^*)^2\pi_2\big) = \big(\varphi_{\rho}^*(\varphi_{\rho}^*\pi_1),\, \varphi_{\rho}^*(\varphi_{\rho}^*\pi_2)\big) = (\varphi_{\rho}^*0,\, \varphi_{\rho}^*\pi_1) = (0,0) = 0.$$
Since $\pi$ is an epimorphism, $(\varphi_{\rho}^*)^2=0$. 

For (3), suppose  $\pi_2\notin\IrrTo(P(a),P(k)^*)$, so $\pi_2\in\rad_T^2(P(a),P(k)^*)$, then there is $M\oplus P_k^*$ and $f_1\in\rad(M_k,M),f_2\in\rad(M,P(k)^*)$, $g_1\in\rad(M,P_k^*), g_2\in\rad(P_k^*,P(k)^*)$ such that $$\pi_2=f_2f_1+g_2g_1$$ where $M\in\add\overline{T}$, $P_k^*\in\add P(k)^*$. If $f_2f_1\neq 0$, then there must exist an indecomposable direct summand $P(c)$ of $M$ and a nonzero morphism $\beta_c\in \IrrTo(P(c),P(k)^*)$.  Because  $M$ is projective, $P(c)$ is projective as well. Hence there is a morphism $\alpha_c\in \Hom_A(P(c),M_k)$ such that $\beta_c=\pi \alpha_c$. Because $\pi_1=\varphi_{\rho}^*\pi_2, \pi_2\in \rad_T^2(P(a),P(k)^*)$, it follows that $\beta_c\in\rad_T^2(P(a),P(k)^*)$, a contradiction with $\beta_c\in \IrrTo(P(c),P(k)^*)$. So $f_2f_1=0$. Thus $g_2g_1\neq 0$. However, because $g_2\in\rad(P_k^*,P(k)^*)$, using $\pi_1=\varphi_{\rho}^*\pi_2$, and $(\varphi_{\rho}^*)^2=0$,  it follows that $\pi_1=0$, a contradiction. So we must have $\pi_2\in\IrrTo(P(a),P(k)^*)$, \ie $\varphi_{\alpha}^*\in \IrrTo(P(a),P(k)^*).$
    \end{proof}

\subsection{T-irreducible morphisms  from $P(a)$ to  $P(k)^*$}

Since $\pi$ is a minimal right $\operatorname{add}(\overline{T})$-approximation of $P(k)^*$, we have the following precise description of the irreducible morphism spaces.
\begin{prop}\label{p: basis for ak*}
  Let $a \in Q_0$ with $a \neq k$.
  \begin{enumerate}
      \item If $a \notin N^-(k)$, then $\dim_\bbk \IrrT(P(a), P(k)^*) = 0$. 
      \item If $a \in N^-(k)$, then 
the set
    \[\mathcal{B}_{ak^*}=
      \{\overline{\varphi_{\alpha}^*} \mid \alpha: a \to k\}\]
    forms a $\bbk$-basis of $\IrrT(P(a), P(k)^*)$.
  \end{enumerate}
\end{prop}
\begin{proof}
    (1) is clear.
    (2) is straightforward if there is no loop at $k$.  We only consider the case where there is a loop $\rho$ at $k$.
    Consider the exchange exact sequence (\ref{g: exchange exact seq for the case with loop}).

Since ~$\varphi_{\alpha}^*=\pi_2\in\IrrTo(P(a),P(k)^*)$,
it is clear that $\mathcal{B}_{ak^*}=
      \{\overline{\varphi_{\alpha}^*} \}$ and  $\mathcal{B}_{ak^*}$ is linearly independent.
Now let $f\in \IrrTo(P(a), P(k)^*)$. Since $P(a)$ is projective, there is $g={\scriptstyle\left(\begin{smallmatrix} g_1 \\ g_2 \end{smallmatrix}\right)}\in \Hom_A(P(a),P(a)\oplus P(a))$ such that $$f=\pi g=(\pi_1,\pi_2){\scriptstyle\left(\begin{smallmatrix} g_1 \\ g_2 \end{smallmatrix}\right)}=\pi_1g_1+\pi_2g_2.$$
    Because $\pi_1=\varphi_{\rho}^*\pi_2\in \rad^2_T(P(a),P(k)^*)$ and $\pi_2=\varphi_{\alpha}^*$, so $$\overline{f}=\lambda \overline{\pi_2}=\lambda \overline{\varphi_{\alpha}^*}$$ for some $\lambda\in \bbk$. Thus $\mathcal{B}_{ak^*}$ spans
 $\IrrT(P(a), P(k)^*)$. Hence, $\mathcal{B}_{ak^*}$ is a basis of
 $\IrrT(P(a), P(k)^*)$. 
\end{proof}

\subsection{T-irreducible morphisms  from  $P(k)^*$ to  $P(c)$}
Recall that $\Ini(p)$ denotes the initial arrow of a nonzero path $p$.
The following lemma describes the relations involving $\varphi_{\alpha}^*$.
\begin{lem}\label{l: condition for a*p=0}
Let $\alpha: a\ra k$ be a non-loop incoming arrow at $k$. Let $p=\lambda_1p_1+\cdots+\lambda_tp_t$ for some nonzero paths $p_i:a\leadsto b$ with  $b\in Q_0$, $b\neq k$, and $\lambda_i\neq 0$ for each $i$.
Then the following are equivalent
\begin{enumerate}
    \item $\varphi_{\alpha}^*\varphi_{p}=0$.
    \item  For each $1\leq i\leq t$,  $\varphi_{\alpha}^*\varphi_{p_i}=0$.
    \item For each $1\leq i\leq t$,  
    $\alpha=\Ini(p_i)$.
\end{enumerate}
Moreover, if there is a loop $\rho$ at $k$, then $$\varphi_\alpha^*\varphi_p=0\iff \varphi_\rho^*\varphi_\alpha^*\varphi_p=0.$$


\end{lem}

\begin{proof}
 If there is no loop at $k$, consider the exchange exact sequence (\ref{g: exact seq for case without loop}) and take $\alpha=\alpha_1$. If there is a loop $\rho$ at $k$, consider the exchange exact sequence (\ref{g: exchange exact seq for the case with loop}). We only consider the case $\din(k)=2$. The case $\din(k)=1$ is similar.

\noindent \textbf{$(2)\Rightarrow (1)$}  is clear.

\noindent \textbf{$(1)\Rightarrow (3)$} 
By assumption, if there is no loop at $k$, then $\pi_1 \varphi_p=\varphi_{\alpha_1}^*\varphi_p=\varphi_{\alpha}^*\varphi_p=0$, while if there is a loop at $k$, then
$\pi_1 \varphi_p=\varphi_{\rho}^*\varphi_{\alpha}^*\varphi_p=0$. So in either case, we have $\pi_1 \varphi_p=0$, which means $$(\pi_1,\pi_2){\scriptstyle\left(\begin{smallmatrix}  \varphi_{p}\\ 0 \end{smallmatrix}\right)}=0.$$
Then there exists $g\in \Hom_A(P(b),P(k))$ such that $${\scriptstyle\left(\begin{smallmatrix}  \varphi_{p}\\ 0 \end{smallmatrix}\right)}={\scriptstyle\left(\begin{smallmatrix} i_1 \\ i_2 \end{smallmatrix}\right)}g={\scriptstyle\left(\begin{smallmatrix} i_1g \\ i_2g \end{smallmatrix}\right)},$$
 \[
\xymatrix{
&&& P(b)\ar@{.>}[dll]_{g}\ar[d]^{{\scriptstyle\left(\begin{smallmatrix}  \varphi_{p}\\0 \end{smallmatrix}\right)}}&&\\
0\ar[r]&P(k)\ar[rr]^{i={\scriptstyle\left(\begin{smallmatrix} i_1 \\ i_2 \end{smallmatrix}\right)}\ \ \ \ }&&P(a_1)\oplus P(a_2)\ar[rr]^{\ \ \ \ \pi=(\pi_1,\pi_2)}&&P(k)^*\ar[r]&0
}
\]
so $$\varphi_{p}=i_1g=\varphi_{\alpha}g=\varphi_{\alpha q}$$ for some nonzero  $q\in e_kAe_b$ such that $\varphi_q=g$. By Lemma~\ref{l: p_1+p_2 in I infer pi in I}, we obtain $p=aq$, thus $\Ini(p_i)=\alpha$  for each $i$. 

\noindent \textbf{$(3)\Rightarrow (2)$.} 
Choose some $i$ with $1\leq i\leq t$,  let $p_i=\alpha \omega$ for some path $\omega: k\leadsto b$, since $k\neq b$, so $\omega$ has length at least one.

\noindent \textbf{Case 1:} There is no loop at $k$,
 then  $\varphi_{p_i}=\varphi_{\alpha \omega}=\varphi_{\alpha_1 \omega}$.
Thus $\alpha_1\omega=p_i\notin \langle I\rangle$, condition (G3) implies that $\alpha_2\omega\in \langle I\rangle$. 
Consider the exchange exact sequence (\ref{g: exact seq for case without loop}), by $i_1=\varphi_{\alpha_1},~i_2=\varphi_{\alpha_2}$,  we obtain $i_2\varphi_{\omega}=\varphi_{\alpha_2}\varphi_{\omega}=0$ and $$0=\pi i \varphi_{\omega}=\pi_1i_1\varphi_{\omega}+\pi_2i_2\varphi_{\omega}=\pi_1\varphi_{\alpha_1\omega}+\pi_2\varphi_{\alpha_2\omega}=\pi_1\varphi_{\alpha_1\omega}=\varphi_{\alpha_1^*}\varphi_{p_i}=\varphi_{\alpha^*}\varphi_{p_i}.$$

\[
\xymatrix{
& P(b)\ar[d]^{\varphi_{\omega}}\ar[drr]^{{\scriptstyle\left(\begin{smallmatrix}  \varphi_{\alpha_1\omega} \\0\end{smallmatrix}\right)}}&&&\\
0\ar[r]&P(k)\ar[rr]^{i={\scriptstyle\left(\begin{smallmatrix} i_1 \\ i_2 \end{smallmatrix}\right)}\ \ \ \ }&&P(a_1)\oplus P(a_2)\ar[rr]^{\ \ \ \ \pi=(\pi_1,\pi_2)}&&P(k)^*\ar[r]&0
}
\]

\noindent \textbf{Case 2:}
 There is a loop $\rho$ at $k$. We have the exchange exact sequence (\ref{g: exchange exact seq for the case with loop}), where $i_1=\varphi_{\alpha},~i_2=\varphi_{\alpha\rho}$. 
 Since  $\alpha\omega=p_i\notin \langle I\rangle,~\rho^2\in I$, it follows that $\alpha\rho\notin I,~\omega=\rho\omega'$ for some path $\omega': k\leadsto b$ such that
$\rho\neq \Ini(\omega')$ and $p=\alpha\rho\omega'$. Thus $\alpha\omega'\in\langle I\rangle$, so  $i_1\varphi_{\omega'}=\varphi_{\alpha}\varphi_{\omega'}=\varphi_{\alpha\omega'}=0$. 
Then $$0=\pi i \varphi_{\omega'}=\pi_1i_1\varphi_{\omega'}+\pi_2i_2\varphi_{\omega'}=\pi_2i_2\varphi_{\omega'}=\pi_2\varphi_{\alpha\rho}\varphi_{\omega'}=\pi_2\varphi_{\alpha\rho\omega'}=\varphi_{\alpha^*}\varphi_{p_i}=\varphi_{\alpha}^*\varphi_{p_i}.$$

\[
\xymatrix{
& P(b)\ar[d]^{\varphi_{\omega'}}\ar[drr]^{{\scriptstyle\left(\begin{smallmatrix} 0\\ \varphi_{\alpha\omega'} \end{smallmatrix}\right)}}&&&\\
0\ar[r]&P(k)\ar[rr]^{i={\scriptstyle\left(\begin{smallmatrix} i_1 \\ i_2 \end{smallmatrix}\right)}\ \ \ \ }&&P(a)\oplus P(a)\ar[rr]^{\ \ \ \ \pi=(\pi_1,\pi_2)}&&P(k)^*\ar[r]&0
}
\] 

For the last part,
suppose that there is a  loop at $k$. It is clear that  $\varphi_\rho^*\varphi_\alpha^*\varphi_p=0$  if  $\varphi_\alpha^*\varphi_p=0$. Conversely, suppose that $\varphi_\rho^*\varphi_\alpha^*\varphi_p=0$, by Proposition \ref{p: relation if exist loop}(2), this exactly $\pi_1\varphi_p=0$, the same hypothesis from which proof of $(1)\Rightarrow(3)$, hence we obtain $\Ini(p_i)=\alpha$ for each $i$. By the above, we have $\varphi_\alpha^*\varphi_p=0$.
\end{proof}

 Now we can describe the irreducible morphisms starting from $P(k)^*$ and the corresponding relations with other  irreducible morphisms. 
\begin{prop}\label{p: path ka is T irreducible}
 Let $\alpha: a\ra k$ be a non-loop incoming arrow at $k$. If there exists a path $p\in\mathcal{P}_k(c,a)$ such that  $p\alpha\in\langle I\rangle$ for some $c\in Q_0$ with $c\neq k$, then there exists a unique morphism
    $\varphi_{[p\alpha]} \in \IrrTo(P(k)^*,P(c))$
  such that  $$\varphi_{p}=\varphi_{[p\alpha]}\varphi_{\alpha}^*, \text{ if there is no loop at }k$$
  or  $$\varphi_{p}=\varphi_{[p\alpha]}\varphi_{\rho}^*\varphi_{\alpha}^*,  \ 0=\varphi_{[p\alpha]}\varphi_{\alpha}^*,\text{ if there is a loop } \rho \text{ at } k.$$
  Moreover, \begin{enumerate}
  \item If  $c=a$, then $\alpha=\Ini(p)$ 
  and $\varphi_{\alpha}^*\varphi_{[p\alpha]}=0$.
  \item If there is another non-loop incoming arrow $\alpha'$ at $k$, then $\varphi_{[p\alpha]}\varphi_{\alpha'}^*=0$. In addition, if $s(\alpha')=c$, then $\varphi_{\alpha'}^*\varphi_{[p\alpha]}= 0$ if and only if $\alpha'=\Ini(p)$. 
  \item For every $A$-module $M$ and every $g\in \Hom_A(P(c),M)$, $g\varphi_p=0$ if and only if $g\varphi_{[p\alpha]}=0$.
  In particular, for every nonzero path $q$ with $t(q)=c$,   $q p\in \langle I\rangle$ if and only if  $\varphi_q \varphi_{[p\alpha]}=0$.

\item For every morphism $h\in\Hom_A(P(d), P(a))$ with $d\neq k$, $\varphi_ph=0$ if and only if $\varphi_{\alpha}^*h=0$.


\end{enumerate} 
\end{prop}

\begin{proof}
 If there is no loop at $k$, consider the exchange exact sequence (\ref{g: exact seq for case without loop}) and  take
 $\alpha=\alpha_1$, thus $a=a_1, i_1=\varphi_{\alpha}$. If there is a loop at $k$, consider the exchange exact sequence (\ref{g: exchange exact seq for the case with loop}) and let $a_2=a$ in this case. We prove the case $\din(k)=2$; the case $\din(k)=1$ is analogous and simpler.

Since  $p\alpha\in\langle I\rangle$, so $\varphi_{p} i_1=\varphi_{p}\varphi_{\alpha}=\varphi_{p\alpha} = 0$, which implies
\[
(\varphi_{p}, 0) \begin{pmatrix} i_1 \\ i_2 \end{pmatrix} = \varphi_{p} i_1 = 0.
\]
So there exists a unique morphism $f_1 \in \Hom_A(P(k)^*, P(c))$ such that
\[
(\varphi_{p}, 0) = f_1 (\pi_1, \pi_2) = (f_1\pi_1, f_1\pi_2).
\]
Hence
\begin{align}\label{g:fipii=p}
    \varphi_{p} = f_1\pi_1 \text{ and }  f_1\pi_2 = 0.
\end{align}
 The corresponding diagram is:
\[\xymatrix{
0\ar[r]&P(k)\ar[r]^{{\scriptstyle\left(\begin{smallmatrix} i_1 \\ i_2 \end{smallmatrix}\right)}\ \ \ \ }&P(a)\oplus P(a_2)\ar[r]^{\ \ \ \ (\pi_1,\pi_2)}\ar[d]_{(\varphi_{p},0)}&P(k)^*\ar[r]\ar@{.>}[dl]^{f_1}\ar[d]^{(g_1,h_1)}&0\\
&&P(c)&M\oplus P_k^*\ar[l]^{{\scriptstyle\left(\begin{smallmatrix} g_2 \\ h_2 \end{smallmatrix}\right)}\ \ \ \ }&}
\]
We claim $f_1 \in \IrrTo(P(k)^*, P(c))$. Suppose otherwise, then $f_1 \in \rad_T^2(P(k)^*, P(c))$, thus there is $M\oplus P_k^*$ and $g_1\in\rad_A(P(k)^*,M),g_2\in\rad_A(M,P(c))$, $h_1\in\rad_A(P(k)^*,P_k^*),h_2\in\rad_A(P_k^*,P(c))$ such that $$f_1=g_2g_1+h_2h_1$$ where $M\in\add\overline{T}$, $P_k^*\in\add P(k)^*$. Moreover, we can suppose that $h_2h_1$ doesn't factor through $\add\overline{T}$.


If there is no loop at $k$, then $\dim \IrrT(P(k)^*, P(k)^*) = 0$, it follows that $h_2h_1=0$.
Thus $$\varphi_{p} = f_1\pi_1 = g_2 g_1 \pi_1.$$ 
Else if there is a loop at $k$, we can suppose that $h_1={\scriptstyle\left(\begin{smallmatrix} \varphi_{\rho}^*\\ \vdots \\ \varphi_{\rho}^* \end{smallmatrix}\right)}$ since $P_k^*\in\add P(k)^*$. Note that in this case $\pi_1=\varphi_{\rho}^*\varphi_{\alpha}^*$. By $(\varphi_{\rho}^*)^2=0$, we must have $h_2h_1\pi_1=0$.
Thus $$\varphi_p=f_{1}\pi_1=g_2g_1\pi_1+h_2h_1\pi_1=g_2g_1\pi_1.$$
 So in any case, we obtain that 
 $\varphi_p=g_2g_1\pi_1$
factoring through $M$,  which contradicts $p \in \mathcal{P}_k(c, a)$ since $M\in \add\overline{T}$. Hence \[f_1 \in  \IrrTo(P(k)^*, P(c)) \text{ and } \varphi_{p} = f_1\pi_1.\]
 Now we set $$\varphi_{[p\alpha]}:=f_1.$$
Hence, if there is no loop at $k$, by $\pi_1=\varphi_{\alpha_1}^*=\varphi_{\alpha}^*$, we have
\begin{align}\label{eq: no loop}
\varphi_{p}=f_1\pi_1=\varphi_{[p\alpha]}\varphi_{\alpha}^*,\quad \text{and }0=f_1\pi_2.
\end{align}
Else if there is a loop $\rho$ at  $k$, by $\pi_1=\varphi_{\rho}^*\varphi_{\alpha}^*$, $\pi_2=\varphi_{\alpha}^*$, we have
\begin{align}\label{eq: a loop at k}
\varphi_{p}=f_1\pi_1=\varphi_{[p\alpha]}\varphi_{\rho}^*\varphi_{\alpha}^*, \text{ and } 0=f_1\pi_2=\varphi_{[p\alpha]}\varphi_{\alpha}^*.
\end{align}

For (1), assume that $c=a$, so   $p\in\mpik(a,a)$. Since $p\in\mpik(a,a)$ and  $a\ne k$, the first arrow of $p$ is a non-loop arrow $\alpha':a\to k$. 
Denote by $\alpha'=\Ini(p)$, so $\alpha'$ is also an arrow from $a$ to $k$. If $\alpha\neq \Ini(p)$, then $\alpha\neq \alpha'$. Let $p=\alpha'\omega$ for some nonzero path $\omega:k\leadsto a$. Since $p\alpha\in\langle I\rangle$, we have 
 $p\alpha'=\alpha'\omega\alpha'\notin \langle I\rangle$. Hence $\alpha'\omega$ and $\omega\alpha'$ are nonzero, then we have an infinite nonzero path $\alpha'\omega\alpha'\omega\alpha'\omega\alpha'\cdots,$ a contradiction with $A$ is finite dimensional. Thus  $\alpha=\Ini(p)$. 
 
 Since $\alpha=\Ini(p)$, by Lemma~\ref{l: condition for a*p=0}, we have $\varphi_{\alpha}^*\varphi_{p}=0$. By 
 \eqref{eq: no loop} we have $$0=(\varphi_{\alpha}^*\varphi_{p},0)=(\varphi_{\alpha}^*f_1\pi_1,f_1\pi_2)=(\varphi_{\alpha}^*f_1\pi_1,\varphi_{\alpha}^*f_1\pi_2)=\varphi_{\alpha}^*f_1(\pi_1,\pi_2)=\varphi_{\alpha}^*f_1\pi=\varphi_{\alpha}^*\varphi_{[p\alpha]}\pi.$$
Because $\pi$ is an epimorphism, we have $\varphi_{\alpha}^*\varphi_{[p\alpha]}=0$.

 For (2), if  there exists another non-loop incoming arrow $\alpha'$ at $k$, obviously, there is no loop at $k$, hence $\alpha'=\alpha_2$. Since $f_1\pi_2=0$, $$0=f_1\pi_2=\varphi_{[p\alpha]}\varphi_{\alpha_2}^*=\varphi_{[p\alpha]}\varphi_{\alpha'}^*.$$
 If $s(\alpha')= c$, then $s(\alpha_2)=a_2= c$. 
 Suppose $\varphi_{\alpha'}^*\varphi_{[p\alpha]}= 0$, thus $\pi_2f_1=\varphi_{\alpha'}^*\varphi_{[p\alpha]}=0$, 
then we have $$(\pi_1,\pi_2){\scriptstyle\left(\begin{smallmatrix} 0\\f_1 \end{smallmatrix}\right)}=\pi_2f_1=0.$$
Then there is $g\in \Hom_A(P(k)^*,P(k))$ such that $${\scriptstyle\left(\begin{smallmatrix} 0\\f_1  \end{smallmatrix}\right)}={\scriptstyle\left(\begin{smallmatrix} i_1 \\ i_2 \end{smallmatrix}\right)}g={\scriptstyle\left(\begin{smallmatrix} i_1g \\ i_2g \end{smallmatrix}\right)}.$$
\[
\xymatrix{
0\ar[r]&P(k)\ar[r]^{{\scriptstyle\left(\begin{smallmatrix} i_1 \\ i_2 \end{smallmatrix}\right)}\ \ \ \ }&P(a)\oplus P(c)\ar[r]^{\ \ \ \ (\pi_1,\pi_2)}&P(k)^*\ar[r]\ar[dl]^{{\scriptstyle\left(\begin{smallmatrix}0 \\ f_1 \end{smallmatrix}\right)}}\ar@{.>}[dll]_{\exists g}&0\\
0\ar[r]&P(k)\ar[r]_{{\scriptstyle\left(\begin{smallmatrix} i_1 \\ i_2 \end{smallmatrix}\right)}\ \ \ \ }&P(a)\oplus P(c)\ar[r]_{\ \ \ \ (\pi_1,\pi_2)}&P(k)^*\ar[r]&0.
}
\]
Hence, $f_1=i_2g$, combining $i_2=\varphi_{\alpha_2}$ and  \eqref{eq: no loop}, thus we can obtain $$\varphi_{p}=f_1\pi_1=i_2g\pi_1=\varphi_{\alpha_2q}$$ for some path $q$. Therefore we must have $\alpha_2=\Ini(p)$, specifically, $\alpha'=\alpha_2=\Ini(p)$.

 On the other hand, suppose that $\alpha'=\Ini(p)$. Then 
by Lemma~\ref{l: condition for a*p=0}, we have $\varphi_{\alpha'}^*\varphi_p=0$. By $\pi_2=\varphi_{\alpha'}^*$ and \eqref{eq: no loop}, 
we have $$(0,0)=(\varphi_{\alpha'}^*\varphi_p,\pi_2f_1\pi_2)=(\pi_2f_1\pi_1,\pi_2f_1\pi_2)=\pi_2f_1(\pi_1,\pi_2)=\pi_2f_1\pi.$$
Because $\pi$ is an epimorphism, we have $\pi_2f_1=0$, specifically,  $\varphi_{\alpha'}^*\varphi_{[p\alpha]}=0$.

For (3), by the preceding argument, we obtain $\varphi_{[p\alpha]}=f_1$. Let $M$ be an $A$-module, $g\in\Hom_A(P(c),M)$. If $g\varphi_{[p\alpha]}=0$, since $\varphi_p=f_1\pi_1=\varphi_{[p\alpha]}\pi_1$,  it is clear that $g\varphi_{p}=g\varphi_{[p\alpha]}\pi_1=0$. On the other hand, if $g\varphi_p=0$,  composing with $\pi$ and using $f_1\pi_2=0$ yield
\[
g\varphi_{[p\alpha]}\pi=g f_1 \pi = (g f_1\pi_1, g f_1\pi_2) = (g\varphi_{p}, 0)=(0, 0).
\]
Because $\pi$ is an epimorphism,  $g\varphi_{[p\alpha]}=gf_1=0$.

For (4), it is clear that we can suppose $$0\neq h=\lambda_1\varphi_{\omega_1}+\cdots +\lambda_t\varphi_{\omega_t}$$ for some nonzero different paths $\omega_1,\cdots,\omega_t: a\leadsto d$ with $\lambda_i\neq 0$ for each $1\leq i\leq t$.

If $\varphi_{\alpha}^*h=0$, since either $\varphi_p=\varphi_{[p\alpha]}\varphi_{\alpha}^*$ or  $\varphi_p=\varphi_{[p\alpha]}\varphi_\rho^*\varphi_{\alpha}^*$, it is clear that  
$\varphi_ph=0$.

On the other hand,
if $\varphi_ph=0$, 
then $$\varphi_ph=\varphi_{\lambda_1p\omega_1+\cdots +\lambda_tp\omega_t}=0.$$
Then we can obtain $\lambda_1p\omega_1+\cdots +\lambda_tp\omega_t\in \langle I\rangle$. Thus for each $1\leq i\leq t$,
$p\omega_i\in \langle I\rangle$.  Since $p\alpha\in \langle I\rangle$, we must have  $\alpha=\Ini(\omega_i)$ for each $i$. Then by Lemma~\ref{l: condition for a*p=0}, we have $$\varphi_{\alpha}^*\varphi_{\omega_i}=0$$ for each $i$. Hence, $\varphi_{\alpha}^*h=0$. 
\end{proof}

\begin{proposition}\label{p:basis for k*c}
 Let $c\in Q_0$ with $c\neq k$, and set
 $$\mathcal{S}_{k^*c}=\bigcup_{a\in N^-(k)}\left\{p\in\mathcal{P}_k(c,a)\mid \exists \alpha: a\to k \text{ such that } p\alpha\in \langle I\rangle\right\}.$$
 Then$$\mathcal{B}_{k^*c}:=\{\overline{\varphi_{[p\alpha]}}\mid p\in \mathcal{S}_{k^*c}\}$$
is a $\bbk$-basis of $\IrrT(P(k)^*,P(c))$. 
\end{proposition}

\begin{proof}
 If there is no loop at $k$, consider the exchange exact sequence (\ref{g: exact seq for case without loop}). If there is a loop at $k$, consider the exchange exact sequence (\ref{g: exchange exact seq for the case with loop}). We prove the case $\din(k)=2$; the case $\din(k)=1$ is analogous and simpler.
 Since $\din(k)=2$, we have $|\mathcal{B}_{k^*c}|=|\mathcal{S}_{k^*c}| \leq 2$.


\noindent\textbf{Case $|\mathcal{B}_{k^*c}|=2$.}  Then there must be two non-incoming arrows at $k$,  so there is no loop at $k$ and $\dim\IrrT(P(k)^*,P(k)^*)=0$.
Moreover, we have $\mathcal{B}_{k^*c} = \{\overline{\varphi_{[p_1\alpha_1]}}, \overline{\varphi_{[p_2\alpha_2]}}\}$ for some paths ${p_1}\in \mathcal{P}_k(c,a_1)$ and ${p_2}\in \mathcal{P}_k(c,a_2)$.
Keep the notations in the proof of Proposition~\ref{p: path ka is T irreducible}, let $$f_i=\varphi_{[p_i\alpha_i]}$$ for $1\leq i\leq 2$ (note that in the proof of Proposition~\ref{p: path ka is T irreducible}, $\alpha=\alpha_1$). So 
$$f_j\pi_j=\varphi_{p_j},\ \ f_j\pi_l=0$$ for $\{j,l\}=\{1,2\}.$

\noindent\textbf{Linear independence}.
Assume 
\[
f := \lambda_1 f_1 + \lambda_2 f_2 \in \rad_T^2(P(k)^*, P(c)).
\]
By $\varphi_{p_1}=f_1\pi_1$ and $f_2\pi_1=0$, we have $$\lambda_1\varphi_{p_1}=\lambda_{1} f_1\pi_1=\lambda_{1} f_1\pi_1+\lambda_{2} f_2\pi_1=f\pi_1\in \rad_T^2(P(k)^*, P(c))\rad_{T}(P(a_1),P(k)^*).$$
Because $\dim\IrrT(P(k)^*,P(k)^*)=0$, we have \[\lambda_1\varphi_{p_1}=f\pi_1\in \rad^2_{\overline{T}}(P(a_1),P(c)).\] However,
${p_1}\in \mathcal{P}_k(c,a_1)$, which means $\varphi_{p_1}\notin \rad^2_{\overline{T}}(P(a_1),P(c))$, then we must have $\lambda_1 = 0$. 
Composing similarly with $\pi_2$ yields $\lambda_2 = 0$. 
Hence $\mathcal{B}_{k^*c}$ is linearly independent. 

\noindent\textbf{Spanning}.
Let $f\in \IrrTo(P(k)^*,P(c))$.  
Note that $$0=f\pi i=f(\pi_1,\pi_2){\scriptstyle\left(\begin{smallmatrix} i_1 \\ i_2 \end{smallmatrix}\right)}=f\pi_1 i_1+f\pi_2i_2=f\pi_1\varphi_{\alpha_1}+f\pi_2\varphi_{\alpha_2}\in \Hom_A(P(k),P(c)).$$
Because $\alpha_1\neq \alpha_2$, so the corresponding terminal arrows of $f\pi_1 i_1$ and $f\pi_2i_2$ are different, then using Lemma~\ref{l: p_1+p_2 in I infer pi in I}, one can get
\begin{align}\label{g: fpi1=fpi2=0}
    f\pi_1 i_1=f\pi_2i_2=0.
\end{align}
Since $\pi$ is an epimorphism and $f\neq 0$, one can get that $$0\neq f\pi=f(\pi_1,\pi_2)=(f\pi_1,f\pi_2).$$
Thus at least one of   $f\pi_1$ and $f\pi_2$ is not zero.


 Suppose that $f\pi_1\neq 0$ and  $$f\pi_1=\lambda_1'\varphi_{q_1}+\cdots+\lambda_t'\varphi_{q_t},$$ for some nonzero different paths $q_1,\cdots, q_t$ from $c$ to $a_1$ with $\lambda_i'\in\bbk^\times$ for each $1\leq i\leq t$.  
Because $$0=f\pi_1 i_1=f\pi_1 \varphi_{\alpha_1}=\lambda_1'\varphi_{q_1\alpha_1}+\cdots+\lambda_t'\varphi_{q_t\alpha_1},$$
we have $q_i\alpha_1\in\langle I\rangle$ for each $1\leq i\leq t$ by Lemma~\ref{l: p_1+p_2 in I infer pi in I}.

Note that 
$p_1\in\mathcal{P}_k(c,a_1)$ and $p_1\alpha_1\in\langle I\rangle$, by Lemma~\ref{l:p is the shortest}, we know that every $q_i$ ends with $p_1$, furthermore $\varphi_{q_i}=h_i'\varphi_{p_1}$, for some $h_i'\in\Hom_A(P(c),P(c))$. 
Thus we can write 
\begin{align}\label{g: f=h1f1pi1}
    f\pi_1= \sum_{i=1}^t \lambda_i' h_i' \varphi_{p_1} = h_1\varphi_{p_1}= h_1f_1\pi_1,
\end{align}
for $h_1 := \sum_{i=1}^t \lambda_i' h_i' \in\Hom_A(P(c),P(c))$.


Suppose that $f\pi_2\neq 0$. By the similar argument, we obtain $$f\pi_2=h_2\varphi_{p_2}= h_2f_2\pi_2$$
 for some  $h_2\in\Hom_A(P(c),P(c))$.

Assume that $f\pi_1\neq 0$ and $f\pi_2\neq 0$.
Because $f_1\pi_2=0=f_2\pi_1$, then $$f\pi_1=h_1f_1\pi_1=(h_1f_1+ h_2f_2)\pi_1$$ and $$f\pi_2=h_2f_2\pi_2=(h_1f_1+ h_2f_2)\pi_2.$$  Then 
$$f\pi=(f\pi_1,f\pi_2)=(h_1f_1+ h_2f_2)\pi.$$ 
 Because $\pi$ is an epimorphism, we must have $$f=h_1f_1+ h_2f_2.$$
 Note that  each  $h_i \in \Hom_A(P(c), P(c))$ admits a decomposition $h_i = \mu_i e_c + r_i$, where $\mu_i \in \bbk$ and $r_i \in \rad_A(P(c), P(c))$. 
 Using $f\in\IrrTo(P(k)^*,P(c))$, we must have $$\overline{f}=\mu_1\overline{f_{1}}+\mu_2\overline{f_{2}}.$$ If $f\pi_i = 0$ for some $i$, the corresponding coefficient $\mu_i$ is simply zero. This shows that $\mathcal{B}_{k^*c}$ spans $\IrrT(P(k)^*, P(c))$. 
 
 Together with the linear independence established earlier, $\mathcal{B}_{k^*c}$ forms a $\bbk$-basis of $\IrrT(P(k)^*, P(c))$.

\noindent\textbf{Case $|\mathcal{B}_{k^*c}|=1$.} 
 Say $\mathcal{B}_{k^*c} = \{\overline{\varphi_{[p\alpha]}}\}$ for some path ${p}\in \mathcal{P}_k(c,a)$ and some arrow $\alpha: a\to k$. 
Keep the notations as in the proof of Proposition~\ref{p: path ka is T irreducible},  we also write $$f_1=\varphi_{[p\alpha]}.$$ By (\ref{g:fipii=p}),   we know $f_1\pi_1=\varphi_{p},~f_1\pi_2=0$.

It is trivial $\mathcal{B}_{k^*c}$ is linearly independent.
Moreover, 
let $f \in \IrrTo(P(k)^*, P(c))$. If there is no loop, similar as (\ref{g: fpi1=fpi2=0}),  we have $f\pi_1 i_1=f\pi_2i_2=0.$
If there is a loop $\rho$ at $k$, we also have
$$0=f\pi i=f(\pi_1,\pi_2){\scriptstyle\left(\begin{smallmatrix} i_1 \\ i_2 \end{smallmatrix}\right)}=f\pi_1 i_1+f\pi_2i_2=f\pi_1\varphi_{\alpha}+f\pi_2\varphi_{\alpha}\varphi_{\rho}\in \Hom_A(P(k),P(c)).$$
Because $\rho\neq \alpha$, so the corresponding terminal arrows of $f\pi_1 i_1$ and $f\pi_2i_2$ are different, then using Lemma~\ref{l: p_1+p_2 in I infer pi in I}, one can get
$f\pi_1 i_1=f\pi_2i_2=0.$ So for any case, we have $$f\pi_1 i_1=f\pi_2i_2=0.$$
 Since  $\pi$ is an epimorphism and $f\neq 0$, it follows that $$0\neq f\pi=f(\pi_1,\pi_2)=(f\pi_1,f\pi_2).$$
Thus at least one of  $f\pi_1$ and $f\pi_2$ is not zero.

Suppose first that $f\pi_1 \neq 0$. Write $f\pi_1 = \sum_{i=1}^t \lambda_i' \varphi_{q_i}$ for distinct nonzero paths $q_i: c \leadsto a$. The relation $f\pi_1i_1=f\pi_1 \varphi_{\alpha} = 0$ implies $q_i \alpha \in \langle I\rangle $ for all $i$.  Since $p\alpha\in\langle I\rangle$ and $p\in \mathcal{P}_k(c,a)$,  by Lemma \ref{l:p is the shortest}, one can get  every $q_i$ ends with $p$.  Thus as (\ref{g: f=h1f1pi1}), we can rewrite
\[
f\pi_1 = h_1f_1\pi_1
\] for some $h_1\in\Hom_A(P(c),P(c))$.

Now suppose that $f\pi_2\neq 0$.
  Write $f\pi_2= \sum_{i=1}^s \lambda_i'' \varphi_{\omega_i}$ 
for distinct nonzero paths $\omega_i: c \leadsto a_2$ with $\lambda_i''\neq 0$ for each $1\leq i\leq s$.
We consider the following cases:

 If there is a loop at $k$, then we know $i_2=\varphi_{\alpha\rho}$. Thus $$0=f\pi_2i_2=\sum_{i=1}^s \lambda_i'' \varphi_{\omega_i\alpha\rho}.$$
Then we must have $\omega_i\alpha\rho\in\langle I\rangle$ for each $i$. Because $\alpha\rho\notin \langle I\rangle$, by the definition of gentle algebras, we have $\omega_i\alpha\in\langle I\rangle$ for each $i$.  
 Using $p\alpha\in\langle I\rangle$ and $p\in \mathcal{P}_k(c,a)$,   by Lemma \ref{l:p is the shortest}, one also can get  every $\omega_i$ ends with $p$. 
 As above, we can rewrite
\[
f\pi_2 = h_2f_1\pi_1
\] for some $h_2\in\Hom_A(P(c),P(c))$.  By $\pi_1=\varphi_{\rho}^*\pi_2$, we have 
\[
f\pi_2 = h_2f_1\varphi_{\rho}^*\pi_2.
\]
Using $f_1\pi_2=0$ and $\varphi_{\rho}^*\pi_1=\varphi_{\rho}^*\varphi_{\rho}^*\pi_2=0$, one can get $$f\pi_1=h_1f_1\pi_1=h_1f_1\pi_1+h_2f_1\varphi_{\rho}^*\pi_1 \text{ and }  f\pi_2 = h_2f_1\varphi_{\rho}^*\pi_2=h_1f_1\pi_2+h_2f_1\varphi_{\rho}^*\pi_2.$$
Thus we have $$f\pi=(f\pi_1,f\pi_2)=(h_1f_1\pi_1+h_2f_1\varphi_{\rho}^*\pi_1,h_1f_1\pi_2+h_2f_1\varphi_{\rho}^*\pi_2)=(h_1f_1+h_2f_1\varphi_{\rho}^*)\pi.$$
Because $\pi$ is an epimorphism, we have $$f=h_1f_1+h_2f_1\varphi_{\rho}^*.$$ It is clear that $h_2f_1\varphi_{\rho}^*\in \rad_T^2(P(k)^*,P(c))$, then we can get $\overline{f}=\mu_1\overline{f_{1}}$ for some $\mu_1\in \bbk$.

 Now suppose that there is no loop at $k$. 
The relation $0=f\pi_2i_2=\sum_{i=1}^s \lambda_i'' \varphi_{\omega_i\alpha_2} $ forces $\omega_i\alpha_2\in\langle I\rangle $ for each $i$. Thus there is $d\neq k$ and a path $\omega\in \mathcal{P}_k(d,a_2)$
such that $\omega\alpha_2\in\langle I\rangle $. By Lemma~\ref{l:p is the shortest},  each $\omega_i$ $(1\leq i\leq s)$ ends with $\omega$.  
Since $\omega\alpha_2\in\langle I\rangle $, by Proposition~\ref{p: path ka is T irreducible}, there is  a unique  morphism $$\varphi_{[\omega\alpha_2]}\in\Hom_A(P(k)^*,P(d)),$$ simply denoted by $f_2$, such that $f_2\pi_2=\varphi_{\omega}$ and $f_2\pi_1=0$. Similarly to the above,
we can write $$f\pi_2=h_2'f_2\pi_2$$ for some  $h_2'\in\Hom_A(P(d),P(c))$.  We claim $h_2'f_2\in \rad^2_{T}(P(k)^*,P(c))$.
Note that if $d=c$, since $\omega\alpha_2\in\langle I\rangle $, thus $f_2\in\IrrTo(P(k)^*,P(c))$, $\overline{f_2}\in \mathcal{B}_{k^*c}$, a contradiction with $|\mathcal{B}_{k^*c}|=1$. So we must have 
$d\neq c$,  so $h_2'\in \rad_A(P(d),P(c))$. Thus $h_2'f_2\in \rad^2_{T}(P(k)^*,P(c))$.

Because $f_2\pi_1=0=f_1\pi_2$, similar discussion as above, we also have $$f=h_1f_1+h_2'f_2.$$
By $h_2'f_2\in \rad_T^2(P(k)^*,P(c))$, $f\in\IrrTo(P(k)^*,P(c))$, we can get $\overline{f}=\mu_1\overline{f_{1}}$ for some $\mu_1'\in \bbk$.

This proves that $\mathcal{B}_{k^*c}$ spans $\IrrT(P(k)^*, P(c))$. Combined with linear independence, $\mathcal{B}_{k^*c}$ is a $\bbk$-basis.

\noindent\textbf{Case $|\mathcal{B}_{k^*c}|=0$.} 
It is equivalent to prove that  $\dimv\Hom_A(P(k)^*,P(c))=0$. Suppose there is    $f \in \IrrTo(P(k)^*, P(c))$. By the preceding argument, 
we know  $f\pi_1 i_1=f\pi_2i_2=0$ and at least one of   $f\pi_1$ and $f\pi_2$ is not zero. Without loss of generality, we can assume  $f\pi_1\neq 0$, by above, there must exist $p\in\mpik(c,a)$ and  $\alpha: a\to k$ such that $p\alpha\in\langle I\rangle$. A contradiction with $|\mathcal{B}_{k^*c}|=0$. Thus  $\dimv\Hom_A(P(k)^*,P(c))=0$.
\end{proof}

\begin{example}
Let $(Q,I)$ be the gentle pair \[\begin{tikzcd}[sep=small]
	{Q=} & a && k && b
	\arrow["{\alpha_1}", shift left=3, from=1-2, to=1-4]
	\arrow["{\beta_1}", shift left=3, from=1-4, to=1-2]
	\arrow["{\alpha_2}"', tail reversed, no head, from=1-4, to=1-6]
\end{tikzcd}\] with $I=\{\beta_1\alpha_1,\alpha_2\beta_1\}.$
The exchange exact sequence is 
\begin{align}
 0\ra P(k) \xra{i={\scriptstyle\left(\begin{smallmatrix} i_1 \\ i_2 \end{smallmatrix}\right)}}P(a)\oplus P(b)\xra{\pi=(\pi_1,\pi_2)} I(k)\ra 0,   
\end{align}
where $i_1=\varphi_{\alpha_1},$  $i_2=\varphi_{\alpha_2},$ 
$\pi_1=\varphi_{\alpha_1}^*$, $\pi_2=\varphi^*_{\alpha_2}.$
Thus $P(k)^*=I(k)$. For $\alpha_1\colon a\to k$, 
let $p=\alpha_1\beta_1$, then $p\in\mathcal{P}_k(a,a)$ and $p\alpha_1\in\langle I\rangle$.  By Proposition~\ref{p: path ka is T irreducible}, there is  $\varphi_{[p\alpha_1]}\in \IrrTo(P(k)^*,P(a))$ such that  $$\varphi_{p}=\varphi_{[p\alpha_1]}\varphi_{\alpha_1}^*,\ \ \ 0=\varphi_{[p\alpha_1]}\varphi_{\alpha_2}^* \text{ and } 0=\varphi_{\alpha_1}^*\varphi_{[p\alpha_1]}.$$
Hence Proposition~\ref{p:basis for k*c} gives $$\mathcal{B}_{k^*a}=\{\overline{\varphi_{[p\alpha_1]}}\}.$$
\end{example}

\begin{example}
Let $(Q,I)$ be the gentle pair \[\begin{tikzcd}[sep=small]
	{Q=} & c && a && k
	\arrow["p"', from=1-2, to=1-4]
	\arrow["\alpha"', from=1-4, to=1-6]
	\arrow["\rho", from=1-6, to=1-6, loop, in=55, out=125, distance=10mm]
\end{tikzcd}\]
with $ I =\{\rho^2,p\alpha\}.$
The exchange exact sequence is 
\begin{align}
 0\ra P(k) \xra{i={\scriptstyle\left(\begin{smallmatrix} i_1 \\ i_2 \end{smallmatrix}\right)}}P(a)\oplus P(a)\xra{\pi=(\pi_1,\pi_2)} I(k)\ra 0,   
\end{align}
where $i_1=\varphi_{\alpha},$  $i_2=\varphi_{\alpha}\varphi_{\rho},$ 
$\pi_1=\varphi_{\rho}^*\varphi_{\alpha}^*$, $\pi_2=\varphi^*_{\alpha}.$
Thus $P(k)^*=I(k)$. Since $p\in \mpik(c,a)$ and $p\alpha\in \langle I\rangle $, by Proposition~\ref{p: path ka is T irreducible}, 
there is a morphism $\varphi_{[p\alpha]}\in \IrrTo(P(k)^*,P(c))$ such that  $$\varphi_{p}=\varphi_{[p\alpha]}\varphi_{\rho}^*\varphi_{\alpha}^*,\ \ \ 0=\varphi_{[p\alpha]}\varphi_{\alpha}^*.$$
Hence $$\mathcal{B}_{k^*c}=\{\overline{\varphi_{[p\alpha]}}\}.$$
\end{example}

\subsection{T-irreducible morphisms  from $P(b)$ to  $P(a)$}

\begin{lem}\label{l: p in rad2 means pa in I}
   Let $a,b\in Q_0$ with $ k\notin \{a,b\}$, and $p\in\mathcal{P}_k(a,b)$. If $\varphi_p\in\rad^2_{P(k)^*}(P(b),P(a))$, then there is an arrow $\alpha: b\ra k$ such that $p\alpha\in\langle I\rangle$.
  \end{lem}
\begin{proof}
Because $\varphi_p\in\rad^2_{P(k)^*}(P(b),P(a))$ and $p$ is a single path in $\mathcal{P}_k(a,b)$, we can suppose that there is $h_1 \in \rad_A(P(b), P(k)^*)$ and $h_2 \in \rad_A(P(k)^*, P(a))$ such that  $\varphi_p = h_2 h_1$.
Since $P(b)$ is projective, there exists $g \in \Hom_A(P(b), M_k)$ such that $h_1 = \pi g$.  
It follows that $$\varphi_p = h_2 \pi g,$$ which means $\varphi_p$ factors through $M_k$. Since $p\in\mathcal{P}_k(a,b)$,  so $g\notin\rad_A(P(b), M_k)$,  thus $P(b)$ must be an indecomposable direct summand of $M_k$.

 \[\xymatrix{&&P(b)\ar@{.>}[d]^{g}\ar[dr]^{h_1}\ar[r]^{\varphi_{p}}&P(a)&
 \\0\ar[r]&P(k)\ar[r]^{i}&M_k\ar[r]^{ \pi}&P(k)^*\ar[r]\ar[u]^{h_2}&0
}\]
If $|M(k)|=1$, then  $M(k)=P(b)$ and $i=\varphi_{\alpha}$ for some arrow $\alpha: b\to k$. By $g\notin\rad_A(P(b), P(b))$, $g$ is an isomorphism. Using $\varphi_p = h_2 \pi g$ and 
$p$ is a single path, we can suppose that  $g = \id_{P(b)}$. Then $\varphi_p=h_2\pi$ and 
 $$\varphi_{p\alpha}=\varphi_p \varphi_{\alpha}=\varphi_p i=h_2\pi i=0.$$
Thus  $p\alpha\in\langle I\rangle$.

If $|M(k)|=2$, we suppose that  $M(k)=P(a_1)\oplus P(a_2)$, and the exchange exact sequence is \[\xymatrix{0\ar[r]&P(k)\ar[r]^{{\scriptstyle\left(\begin{smallmatrix} i_1 \\ i_2 \end{smallmatrix}\right)}\ \ \ \ }&P(a_1)\oplus P(a_2)\ar[r]^{\ \ \ \ (\pi_1,\pi_2)}&P(k)^*\ar[r]&0.
 }
\]
Take  $g={\left(\begin{array}{cc}
         g_1 \\
         g_2 
    \end{array}\right)}$. 
  \[\xymatrix{&&P(b)\ar@{.>}[d]^{{\scriptstyle\left(\begin{smallmatrix} g_1\\g_2\end{smallmatrix}\right)}}\ar[dr]^{h_1}\ar[r]^{\varphi_{p}}&P(a)&\\0\ar[r]&P(k)\ar[r]^{{\scriptstyle\left(\begin{smallmatrix} i_1 \\ i_2 \end{smallmatrix}\right)}\ \ \ \ }&P(a_1)\oplus P(a_2)\ar[r]^{\ \ \ \ (\pi_1,\pi_2)}&P(k)^*\ar[r]\ar[u]^{h_2}&0
 }
\]
Then $$\varphi_{p}=h_2\pi g=h_2(\pi_1,\pi_2){\left(\begin{array}{cc}
         g_1 \\
         g_2 
    \end{array}\right)}=h_2\pi_1g_1+h_2\pi_2g_2.$$ 
 Since  $p$ is a single path, either $h_2\pi_1g_1=0$ or $h_2\pi_2g_2=0$. Without loss of generality, assume $h_2 \pi_1 g_1 = 0$. Then $\varphi_p=h_2\pi_2g_2 $ and 
   $g_2$ is an isomorphism, so $a_2=b$. Thus, using $p$ is a single path, we also can assume $g_2 = \id_{P(b)}$. Consequently, $\varphi_p = h_2 \pi_2$.
Because $$0=h_2\pi i=h_2\pi_1i_1+h_2\pi_2i_2=h_2\pi_1i_1+\varphi_{p}i_2,$$
associated with $i_1$ and  $i_2$ are induced by different paths, so
$\varphi_{p}i_2=0$. 

If $i_2=\varphi_{\alpha_2}$ for some arrow $\alpha_2:b\to k$,  then $$0=\varphi_{p}i_2=\varphi_{p\alpha_2},$$
so $p\alpha_2\in \langle I\rangle.$

If $i_2=\varphi_{\alpha\rho}$ for a non-loop incoming arrow $\alpha: b\to k$ and a loop $\rho$ at $k$, then we  get $p\alpha\rho\in \langle I\rangle$. Note that $\alpha\rho\notin \langle I\rangle$, we must have $p\alpha\in \langle I\rangle$. 
\end{proof}

\begin{prop}\label{p:path ab is T irreducible}
 Let $a,b\in Q_0$ with $k\notin\{a,b\}$, and let $p\in\mathcal P_k(a,b)$. Then
\[
\varphi_p\in\Irr_{T}^\circ(P(b),P(a)) \text{ if and only if } p\alpha\notin\langle I\rangle  \text{ for every non-loop arrow $\alpha:b\to k$.}
\]

\end{prop}

\begin{proof}
Suppose first that $p\alpha\in\langle I\rangle$ for some non-loop arrow $\alpha:b\to k$. By Proposition~\ref{p: path ka is T irreducible}, there is a morphism $\varphi_{[p\alpha]}:P(k)^*\to P(a)$. If there is no loop at $k$, then
\[
\varphi_p
=
\varphi_{[p\alpha]}\varphi_\alpha^*,
\]
so $\varphi_p\notin\Irr_{T}^\circ(P(b),P(a))$. If there is a loop $\rho$ at $k$, then
\[
\varphi_p
=
\varphi_{[p\alpha]}\varphi_\rho^*\varphi_\alpha^*,
\]
so $\varphi_p\notin\Irr_{T}^\circ(P(b),P(a))$. 
Both cases contradict the assumption that $\varphi_p\in\Irr_{T}^\circ(P(b),P(a))$. 

Conversely, suppose that $p\alpha\notin\langle I\rangle$ for every non-loop arrow $\alpha:b\to k$. If $\varphi_p \notin \IrrTo(P(b), P(a))$, because $p\in\mathcal{P}_k(a,b)$, which means $\varphi_p \in \Irr_{\overline{T}}^{\circ}(P(b), P(a))$, by Lemma~\ref{l: only need to check whether pass through P_k*},  $\varphi_p\in\rad^2_{P(k)^*}(P(b),P(a))$. By Lemma~\ref{l: p in rad2 means pa in I},  there is an arrow $\alpha: b\ra k$ such that $p\alpha\in\langle I\rangle$. A contradiction with the assumption. Thus, we must have $\varphi_p \in \IrrTo(P(b), P(a)).$
\end{proof}




\begin{proposition}\label{p:basis for ba}
    Let $a,b\in Q_0$ with $k\notin\{a,b\}$, and set$$\mathcal{S}_{ba}=\{p\mid p\in \mathcal{P}_k(a,b), p\alpha\notin\langle I\rangle \text{ for every } \alpha: b\ra k\}.$$
    Then $$\mathcal{B}_{ba}=\{\overline{\varphi_p}\mid p\in \mathcal{S}_{ba}\} $$ 
    is a $\bbk$-basis of $\IrrT(P(b),P(a))$.
\end{proposition}
\begin{proof}

Let $p_1,\cdots,p_t$ be all the paths in $\mathcal{S}_{ba}$. By Lemma~\ref{l: p_1+p_2 in I infer pi in I}, $\varphi_{p_1}, \cdots,\varphi_{p_t}$ are linearly independent. By Proposition~\ref{p:path ab is T irreducible},  $\varphi_{p_j}\in \IrrTo(P(b),P(a))$ for $1\le j\le t$.

Suppose $\sum_{i=1}^t \lambda_i \overline{\varphi_{p_i}} = 0$ in $\Irr_T(P(b), P(a))$ for some $\lambda_i\in\bbk$, which means
$$f=\lambda_1\varphi_{p_1}+\cdots+\lambda_t\varphi_{p_t}\in \rad_T^2(P(b),P(a)).$$ 
Then there is $M\oplus P_k^*$ and $g_1\in\rad_A(P(b),M),g_2\in\rad_A(M,P(a))$, $h_1\in\rad_A(P(b),P_k^*),h_2\in\rad_A(P_k^*,P(a))$ such that $$f-g_2g_1-h_2h_1=0$$ where $M\in\add\overline{T}$, $P_k^*\in\add P(k)^*$. By Proposition \ref{p:path ab is T irreducible}, we know for each $j$, $\varphi_{p_j}\notin \rad_{\overline{T}}^2(P(b),P(a))$ and $\varphi_{p_j}\notin \rad_{P(k)^*}^2(P(b),P(a))$, so by  Lemma~\ref{l: p_1+p_2 in I infer pi in I}, we must have $$\lambda_j\varphi_{p_j}=0.$$
 Thus $\lambda_j=0$ for each $j$. Then we get $\mathcal{B}_{ba}$ is linearly independent.


On the other hand,  let $p$ be a path from $a$ to $b$ such that $\varphi_p\in \IrrTo (P(b),P(a))$.
Thus $\varphi_p\in \Irr_{\overline{T}}(P(b),P(a))$, so we must have $p\in \mathcal{P}_k(a,b)$. Then by Proposition \ref{p:path ab is T irreducible}, one can get that  for each arrow $\alpha: b\ra k$,  $p\alpha\notin \langle I\rangle $,  thus $p\in \mathcal{S}_{ba}$, so
$\overline{\varphi_p}\in \mathcal{B}_{ba}$. Since any morphism from $P(b)$ to $P(a)$ is a linear combination of paths morphism from $a$ to $b$, so $\mathcal{B}_{ba}$ spans $\IrrT(P(b),P(a))$. 

 Combining linear independence and spanning, $\mathcal{B}_{ba}$ forms a $\bbk$-basis of $\IrrT(P(b), P(a))$.
\end{proof}

\section{Endomorphism algebra of the new tilting module}\label{s:s6 endo}
 Throughout this section, let $A=\bbk Q/\langle I\rangle$ be a gentle algebra, and let $k\in Q_0$ be a vertex at which the tilting mutation exists.
We retain the notation
\[
\overline T=\bigoplus_{\ell\neq k}P(\ell),
\qquad
T=\overline T\oplus P(k)^*.
\]


\subsection{Basis}
For vertices $a,b,c \in Q_0$ with $k \notin \{a,b,c\}$, we recall the bases obtained in the preceding section:
\begin{align*}
    &\mathcal{B}_{k^*k^*} = \{\overline{\varphi_{\rho}^*} \mid \rho \text{ is a loop at } k\} , \\
    &\mathcal{B}_{ak^*} = \{\overline{\varphi_{\alpha}^*} \mid \alpha \colon a \to k\}, \\
    &\mathcal{B}_{k^*c} = \{\overline{\varphi_{[p\alpha]}} \mid p \in \mathcal{P}_k(c,a),\, \alpha \colon a \to k,\, p\alpha \in \langle I \rangle, a\in N^-(k)\}, \\
    &\mathcal{B}_{ba} = \{\overline{\varphi_p} \mid p \in \mathcal{P}_k(a,b),\, p\alpha \notin \langle I \rangle \text{ for every } \alpha \colon b \to k\}.
\end{align*}
Set
\begin{align*}
    \mathcal{B} &= \mathcal{B}_{k^*k^*} \cup \biggl(\bigcup_{a\in N^-(k)} \mathcal{B}_{ak^*}\biggr) \cup \biggl(\bigcup_{c\in N^+(k)} \mathcal{B}_{k^*c}\biggr) \cup \biggl(\bigcup_{\substack{a \neq k \\ b \neq k}} \mathcal{B}_{ba} \biggr).
\end{align*}
Then $\End_A(T)$ is generated as an algebra by $\mathcal{B}$ together with the idempotents $\{\xi_x\}_{x \in (Q_T)_0}$, which form a complete set of primitive orthogonal idempotents. Here, $\xi_{k^*}$ corresponds to $P(k)^*$, while $\xi_x$ corresponds to $P(x)$ for $x \neq k$. 
  
\subsection{The Gabriel quiver}

Set
\[
B:=\mu_k^+(A)=\End_A(T)=\bbk  Q_T/\langle I_T\rangle.
\]
The vertex set of the Gabriel quiver $Q_T$ of $B$ is
\[
(Q_T)_0=(Q_0\setminus\{k\})\cup\{k^*\},
\]
where $k^*$ corresponds to $P(k)^*$.
By the preceding propositions, $(Q_T)_1$ consists of the following four types of arrows.

\begin{enumerate}
\item[\bf(T1)]
If $k$ carries a loop $\rho$, then there is a loop
\[
\rho^*:k^*\to k^*
\]
corresponding to $\overline{\varphi_\rho^*}$.

\item[\bf(T2)]
For every non-loop incoming arrow $\alpha:a\to k$, there is an arrow
\[
\alpha^*:k^*\to a
\]
corresponding to $\overline{\varphi_\alpha^*}$.

\item[\bf(T3)]
For every non-loop incoming arrow $\alpha:a\to k$ and every
$p\in\mathcal P_k(c,a)$ satisfying $p\alpha\in\langle I\rangle$, there is an arrow
\[
[p\alpha]:c\to k^*
\]
corresponding to $\overline{\varphi_{[p\alpha]}}$.

\item[\bf(T4)]
For every $p\in\mathcal P_k(a,b)$ satisfying
\[
p\alpha\notin\langle I\rangle
\]
for every non-loop arrow $\alpha:b\to k$,  there is an arrow
\[
[p]:a\to b
\]
corresponding to $\overline{\varphi_p}$.
\end{enumerate}
For convenience, we also denote by $\mathbf{(Ti)}$ the set of all the arrows of type $\mathbf{(Ti)}$ in $(Q_T)_1$ for each $1\leq i\leq 4$.
For each arrow $\eta$ of $Q_T$, for clarity, we denote the corresponding irreducible morphism by $\phi_\eta$; thus
\[
\phi_{\rho^*}=\varphi_\rho^*,
\quad
\phi_{\alpha^*}=\varphi_\alpha^*,
\quad
\phi_{[p\alpha]}=\varphi_{[p\alpha]},
\quad
\phi_{[p]}=\varphi_p.
\]
 To describe the construction more explicitly, we introduce the following notation.
\begin{definition}
Let $\alpha \colon a \to k$ be a non-loop incoming arrow at $k$. We denote by
\begin{itemize}
    \item $p_{\alpha}^-:$ the shortest nonzero path such that $p_{\alpha}^- \alpha \in \langle I\rangle $ and $s(p_{\alpha}^-) \neq k$, if such a path exists.
     \item $p_{\alpha}^+:$ the shortest path such that $\alpha p_{\alpha}^+\notin\langle I\rangle$,  $t(p_{\alpha}^+)\neq k$, and $p_{\alpha}^+ \beta \notin \langle I\rangle $ for every arrow $\beta: t(p_{\alpha}^+)\to k$, if such a path exists.
\end{itemize}
\end{definition}
\begin{example}
\begin{enumerate}
    \item  In Example~\ref{eg:mutation in without loop and 2-cycle}, $p_{\alpha_1}^-=\gamma$, $p_{\alpha_1}^+=\beta_2$.
    \item   In Example~\ref{eg:mutation in with loop and no 2-cycle}, $p_{\alpha}^-=\gamma$, $p_{\alpha}^+=\rho\beta$.
    \item  In Example~\ref{eg:mutation in without loop and have 2-cycle}, $p_{\alpha_1}^-=\alpha_1\beta_2$, $p_{\alpha_1}^+$ doesn't exist.
\end{enumerate}

\end{example}

\begin{remark}\label{r: arrows in QT}

 \begin{enumerate}
 \item By the definition of gentle algebra, whenever $p_\alpha^{\pm}$ exists, it is uniquely determined. 
 \item For a non-loop incoming arrow $\alpha$ at $k$, $[p_{\alpha}^-\alpha]\in \bf(T3)$, while $[\alpha p_{\alpha}^+]\in \bf(T4)$.
 \item Let $[p]\in \bf(T4)$, it is easy to get that  either $[p]=[\alpha p_{\alpha}^+]$ for some non-loop incoming arrow $\alpha$ at $k$ or $[p]=[\gamma]$  for an arrow $\gamma:a\to b$ in $Q_1$ with $a,b \in Q_0 \setminus \{k\}$ satisfying  $\gamma\alpha \notin  I $ for every arrow $\alpha \colon b \to k$ in $Q_1$.
 \end{enumerate}
     \end{remark}
Thus, for each non-loop  incoming  arrow $\alpha$ at $k$, there are  at most three special types of arrows in $(Q_T)_1$:
$$\alpha^*,[p_{\alpha}^-\alpha],[\alpha p_{\alpha}^+].$$
Denote the set of arrows in $Q_1$ that are directly modified during the mutation at $k$ by
$$\AmutQ=\AinQ(k)\cup \AoutQ(k)\cup\{\gamma\in Q_1\mid \exists \alpha\in \AoinQ(k)) \text{ such that  } \gamma\alpha\in I\}.$$
The arrows that remain completely unchanged form the complementary set
\[
\AfixQ:= Q_1 \setminus \AmutQ.
\]
Moreover, set 
\begin{itemize}
    \item $\bfA_{T}^{\new}=\{\alpha^*,[p_{\alpha}^-\alpha],[\alpha p_{\alpha}^+]\mid \alpha \in\Aoin(k)\}\cup\{\rho^*\mid \rho \text{ is a loop at $k$}\}$;
    \item $\bfA_T^{\old}=\{[\gamma]\mid \gamma\in Q_1, \gamma\in \AfixQ\}.$
\end{itemize}

\begin{prop}\label{p:Gabriel-quiver}
The quiver $Q_T$ described above is the Gabriel quiver of $\mu_k^+(A)=\End_A(T)$. Moreover, $$(Q_T)_1=\bfA_{T}^{\new}\cup \bfA_{T}^{\old}.$$
\end{prop}

\begin{proof}
The vertices of $Q_T$ correspond to the indecomposable direct summands of the basic tilting module $T$. Propositions~\ref{p: relation if exist loop}, \ref{p: basis for ak*}, \ref{p:basis for k*c}, and \ref{p:basis for ba} provide bases of all irreducible morphism spaces between these summands. Hence the four classes {\rm(T1)--(T4)} give exactly the arrows of the Gabriel quiver of $\End_A(T)$.

 By definition, $\bfA_{T}^{\new}\subseteq (Q_T)_1$. Let $[\gamma]\in \bfA_{T}^{\old}$ for an arrow $\gamma : a\to b$ in  $ \AfixQ$, we have $k\notin\{a,b\}$ and $\gamma\alpha\notin I$ for every arrow $\alpha: b\to k$. Hence $[\gamma]\in{\bf(T4)}$, so $\bfA_{T}^{\old}\subset (Q_T)_1$. Conversely, the construction process of $(Q_T)_1$ and  Remark~\ref{r: arrows in QT} give
    $(Q_T)_1\subseteq \bfA_{T}^{\new}\cup \bfA_{T}^{\old}$, which proves the result.
\end{proof}

\subsection{Quadratic zero relations}
 In this subsection, we will give a precise description of quadratic zero relations in $Q_T$.

\begin{definition}\label{d:quadratic relations}
We define the set of \emph{quadratic relations} $I_{\mb}$ in  $Q_T$ as follows.
\begin{enumerate}
    \item \textbf{Loop-free case:} 
     If there is no loop  at $k$, 
\begin{align*}
 I_\mb&=\{[\gamma_1][\gamma_2]\mid\gamma_1:a\to b,\gamma_2:b\to c,  \gamma_1,\gamma_2\in \AfixQ, \gamma_1\gamma_2\in I,a,b,c\in Q_0\setminus\{k\} \}\\
  &\cup \{[p_{\alpha}^-\alpha]\alpha'^*\mid   \alpha\neq \alpha',  \alpha,\alpha'\in\AinQ(k)\} \\
   &\cup \{\alpha^*[\alpha p_{\alpha}^+]\mid \alpha\in\AinQ(k)\}\\
 &\cup\{[\gamma][p_{\alpha}^-\alpha]\mid  \alpha\in\AinQ(k), \gamma\in \AfixQ, t(\gamma)=s(p_{\alpha}^-), \gamma p_{\alpha}^-\in  \langle I \rangle\}\\
 &\cup \{\alpha^*[p_{\alpha'}^-\alpha']\mid    \alpha,\alpha'\in\AinQ(k),  \alpha=\Ini(p_{\alpha'}^-)\} \\
    &\cup \{[\alpha' p_{\alpha'}^+][p_{\alpha}^-\alpha]\mid \alpha',\alpha\in\AinQ(k), t(p_{\alpha'}^+)=s(p_{\alpha}^-), p_{\alpha'}^+p_{\alpha}^-\in  \langle I \rangle\} \\
    &\cup \{[\alpha p_{\alpha}^+][\gamma]\mid \alpha\in\AinQ(k), \gamma\in \AfixQ, t(p_{\alpha}^+)=s(\gamma), p_{\alpha}^+\gamma\in  \langle I \rangle\}.
\end{align*}
In general,
\begin{align*}
I_{\mathcal{B}} &= \underbrace{\{ [p][q] \mid [p], [q]\in{\bf(T4)}, pq \in \langle I \rangle \} \cup \{ [p][p_{\alpha}^-\alpha] \mid [p] \in{\bf(T4)}, \alpha \in \AinQ(k),pp_{\alpha}^- \in \langle I \rangle \}}_{\text{Inherited from } I} \\
&\cup \underbrace{\{ [p_{\alpha}^-\alpha]\alpha'^* \mid \alpha \neq \alpha' \in \AinQ(k) \} \cup \{ \alpha^*[\alpha p_{\alpha}^+] \mid \alpha \in \AinQ(k) \} \cup \{ \alpha^*[p_{\alpha'}^-\alpha'] \mid  \alpha = \Ini(p_{\alpha'}^-) \}}_{\text{Mutation relations at } k^*}.
\end{align*} 
    \item \textbf{Loop case:} If there is a loop $\rho$ at $k$ and $\alpha$ is the unique non-loop incoming arrow,
    \begin{align*}
 I_\mb&=\{[\gamma_1][\gamma_2]\mid\gamma_1:a\to b,\gamma_2:b\to c,  \gamma_1,\gamma_2\in \AfixQ,  \gamma_1\gamma_2\in I,a,b,c\in Q_0\setminus\{k\} \}\\
  &\cup \{[\alpha p_{\alpha}^+][\gamma]\mid  \gamma\in \AfixQ, t(p_{\alpha}^+)=s(\gamma), p_{\alpha}^+\gamma\in  \langle I \rangle\} \\
 &\cup\{[\gamma][p_{\alpha}^-\alpha]\mid  \gamma\in \AfixQ, t(\gamma)=s(p_{\alpha}^-), \gamma p_{\alpha}^-\in  \langle I \rangle \}\\
  &\cup \{\alpha^*[p_{\alpha}^-\alpha]\mid    \alpha=\Ini(p_{\alpha}^-)\} \\
&\cup \{[p_{\alpha}^-\alpha]\alpha^*, \alpha^*[\alpha p_{\alpha}^+], {(\rho^*)}^2\}. 
\end{align*}
\end{enumerate}
\end{definition}
By the construction of $Q_T$, we have the following basic relation between $\varphi$ and $\phi$.
\begin{lem} \label{l: morpshism relations of phi and varphi}
Let $\alpha\in\AoinQ(k)$. Then
\begin{itemize}
    \item $\phi_{[p_{\alpha}^-\alpha]}\phi_{\alpha^*}=\varphi_{[p_{\alpha}^-\alpha]}\varphi_{\alpha}^*=\varphi_{p_{\alpha}^-}$ if there is no loop at $k$;
    \item  $\phi_{[p_{\alpha}^-\alpha]}\phi_{\rho^*}\phi_{\alpha^*}=\varphi_{[p_{\alpha}^-\alpha]}\varphi_{\rho}^*\varphi_{\alpha}^*=
    \varphi_{p_{\alpha}^-}$ if there is a loop $\rho$ at $k$;
    \item $\phi_{[\alpha p_{\alpha}^+]}=\varphi_{\alpha p_{\alpha}^+}$.
\end{itemize}
\end{lem}
\begin{proposition}\label{p: construction of Ib}
The set $I_{\mb}$ is precisely the set of quadratic zero relations in $\End_A(T)$, \ie 
\[
I_{\mb} = \{\omega_1\omega_2 \mid \omega_1,\omega_2\in (Q_T)_1,\ \phi_{\omega_1}\phi_{\omega_2}=0\}.
\]
\end{proposition}

\begin{proof}
Let $\omega_1,\omega_2\in(Q_T)_1$ be two arrows in $(Q_T)_1$, we are going to prove that  $$\phi_{\omega_1}\phi_{\omega_2}=0\iff \omega_1\omega_2\in I_\mb.$$
If either $\omega_1$ or $\omega_2$ is of type ${\bf(T1)}$, then $k$ carries a loop. Let $\alpha$ be the unique non-loop incoming arrow at $k$. By Proposition~\ref{p: relation if exist loop}, the only such zero composition is
$\phi_{\rho^*}\phi_{\rho^*}=0.$ Thus $$\phi_{\omega_1}\phi_{\omega_2}=0\iff \omega_1=\rho^*,\omega_2=\rho^*\iff \omega_1\omega_2\in I_\mb.$$ In the remainder of the proof, assume that neither $\omega_1$ nor $\omega_2$ is of type ${\bf(T1)}$.  Let $\alpha,\alpha'$ denote non-loop incoming arrows at $k$ (here, we allow $\alpha=\alpha'$).

\noindent\textbf{Case 1}: $\omega_1$ is of type $\bf(T2)$. Suppose that $\omega_1=\alpha^*$. Since $\phi_{\omega_1}\phi_{\omega_2}=\phi_{\omega_1\omega_2}$, we have $s(\omega_2)=t(\alpha^*)\neq k$, and $\omega_2$ can only be of type ${\bf(T3)}$ or ${\bf(T4)}$.
\begin{itemize}
    \item If $\omega_2$ is of type ${\bf(T3)}$, say $\omega_2=[p_{\alpha'}^-\alpha']$, then Proposition~\ref{p: path ka is T irreducible} gives
    $$0=\phi_{\alpha^*}\phi_{[p_{\alpha'}^-\alpha']}=\varphi_{\alpha}^*\varphi_{[p_{\alpha'}^-\alpha']}\iff \alpha=\Ini(p_{\alpha'}^-)\iff {\alpha}^*[p_{\alpha'}^-\alpha']\in I_\mb.$$
    \item If $\omega_2=[q]$ is of type ${\bf(T4)}$, then $\Ini(q)=\alpha$ if and only if $q=\alpha p_\alpha^+$. By Lemma~\ref{l: condition for a*p=0}, for any nonzero path $p\colon a\leadsto b$ in $Q$ with $b\neq k$, we have $\varphi_{\alpha}^*\varphi_p=0$ if and only if $\alpha=\Ini(p)$. Hence 
    $$0=\phi_{\alpha^*[q]}=\phi_{\alpha}^*\phi_{[q]}\iff \Ini(q)=\alpha\iff[q]=[\alpha p_\alpha^+]\iff \alpha^*[\alpha p_\alpha^+]\in I_\mb.
    $$
\end{itemize}

\noindent\textbf{Case 2:} $\omega_1$ is of type ${\bf(T3)}$. Suppose that $\omega_1=[p_{\alpha}^-\alpha]$ for a non-loop arrow $\alpha:a\to k$. Since $t([p_{\alpha}^-\alpha])=k^*$, the arrow $\omega_2$ must be of type ${\bf(T2)}$.
Suppose that $\omega_2=\alpha'^*$. 
If there is no loop at $k$,
Proposition~\ref{p: path ka is T irreducible}(2) gives $$0=\phi_{[p_\alpha^-\alpha]}\phi_{\alpha'^*}=\varphi_{[p_\alpha^-\alpha]}\varphi_{\alpha'}^* \iff \alpha\neq \alpha'\iff [p_\alpha^-\alpha]{\alpha'}^*\in I_\mb.$$ If there is a loop at $k$, then $\alpha'=\alpha$, and Proposition~\ref{p: path ka is T irreducible} gives $$\phi_{[p_\alpha^-\alpha]}\phi_{\alpha^*}=\varphi_{[p_\alpha^-\alpha]}\varphi_{\alpha}^*=0 \iff [p_\alpha^-\alpha]{\alpha}^*\in I_\mb.$$

\noindent\textbf{Case 3:} $\omega_1$ is of type ${\bf(T4)}$. Suppose that $\omega_1=[q]$ for some $q\in\mathcal P_k(a,b)$ with $k\notin\{a,b\}$ satisfying
$q\alpha\notin\langle I\rangle$
for every non-loop arrow $\alpha:b\to k$. Then $\omega_2$ can only be of type ${\bf(T3)}$ or ${\bf(T4)}$.
\begin{itemize}
    \item If $\omega_2$ is of type ${\bf(T3)}$, say $\omega_2=[p_{\alpha}^-\alpha]$, then Proposition~\ref{p: path ka is T irreducible}(3) gives
    $$0=\phi_q\phi_{[p_{\alpha}^-\alpha]}=\varphi_q\varphi_{[p_{\alpha}^-\alpha]}=0\iff\varphi_q\varphi_{p_{\alpha}^-}=0\iff qp_{\alpha}^-\in\langle I\rangle\iff [q][p_{\alpha}^-\alpha]\in I_\mb.$$ 
    \item If $\omega_2$ is of type ${\bf(T4)}$, say $\omega_2=[q']$ for some $q'$ satisfying
$q'\alpha\notin\langle I\rangle$
for every non-loop arrow $\alpha:b\to k$.
Then $$0=\phi_{[q]}\phi_{[q']}=\varphi_q\varphi_{q'}\iff qq'\in \langle I\rangle\iff [q][q']\in  I_\mb.$$
\end{itemize}
Then using (3) of Remark~\ref{r: arrows in QT}, one can get this result.
\end{proof}

\begin{remark}
    Let $$I^{\mathrm{mut}}_Q = \bigcup_{\alpha \in \AinQ(k)} \Bigl(
    \{\alpha\beta \in I\}
    \cup \{\beta\alpha \in I\}
    \cup \{\gamma\,\Ini(p_{\alpha}^-) \in I\}
    \cup \{\Ter(p_{\alpha}^+)\,\gamma \in I\}
\Bigr).$$
The relations listed in $I_Q^{\mathrm{mut}}$ are precisely the relations introduced by the mutation at $k$; none of them belongs to the inherited part of $I_B$.
\end{remark}

\begin{lem}\label{l: Q' satisfy g2-g3}
  Let $v\in (Q_T)_0$.
    \begin{itemize}
    \item[(g1)] There are at most two incoming arrows and  at most two outgoing arrows at $v$;
    \item[(g2)] For each incoming arrow $\alpha$ at $v$, there is at most one outgoing arrow $\beta$ at $v$ such that $\alpha\beta\in I_\mb$, and at most one outgoing arrow $\gamma$ at $v$ such that $\alpha\gamma\notin I_\mb$;
    \item[(g3)] For each outgoing arrow $\alpha$ at $v$, there is at most one incoming arrow $\beta$ at $v$ such that $\beta\alpha\in I_\mb$, and at most one incoming arrow $\gamma$ at $v$ such that $\gamma\alpha\notin I_\mb$.
   \end{itemize}
   \end{lem}   
\begin{proof} 
By construction, $$\AoutT(k^*)=\{\alpha^*\mid\alpha\in\AinQ(k)\},\ \ \AinT(k^*)=\{[p_{\alpha}^-\alpha]\mid\alpha\in\AinQ(k)\},$$ and so
\[\dinT(k^*)\leq \dinQ(k)\leq 2 \text{ and } \doutT(k^*)=\dinQ(k)\leq 2.\]
For $a,b,c\in Q_0\setminus\{k\}$,
\begin{itemize}
    \item  ${\bf(T1)}$ and ${\bf(T2)}$ mean that  the arrow  $\alpha:a\ra k$ is replaced by $\alpha^*:k^*\ra a$;
    \item ${\bf(T3)}$ mean that the path $p_\alpha^-:c\leadsto a$ is replaced by $[p\alpha]\alpha^*:c\to k^*$ for $\alpha\in\AoinQ(k)$;
    \item   ${\bf(T4)}$ mean that some path $p:a\leadsto b$ is replaced by $[p]:a\to b$.
\end{itemize}
Together with (G2),  one can get that for $a\in Q_0\setminus\{k\}$, 
\[\dinT(a)\leq\max\{\dinQ(a),\doutQ(a)\}\leq2
\qquad\text{and}\qquad \doutT(a)\leq\doutQ(a)\leq2.\]
Hence (g1) holds.

    

We now prove (g2) and (g3). For every incoming arrow $\alpha\in\AoinQ(k)$, we verify (g2) and (g3) by cases on $v$:

\vspace{0.2cm}
\noindent\textbf{Case 1:  $v=k^*$.} 
If $k$ carries a loop $\rho$,  denote the unique non-loop incoming arrow at $k$ by $\alpha$. Then
    $$\AoutT(k^*)=\{\alpha^*,\rho^*\},\ \ \ \AinT(k^*)=\{[p_{\alpha}^-\alpha],\rho^*\}.$$
Since we  only have $(\rho^*)^2\in I_\mb$ and $[p_{\alpha}^-\alpha]\alpha^*\in I_\mb$,
the arrows at $k^*$ satisfy (g2) and (g3).

If there is no loop at $k$, then $$\AoutT(k^*)=\{\alpha^*\mid\alpha\in\AinQ(k)\},\ \ \AinT(k^*)=\{[p_{\alpha}^-\alpha]\mid\alpha\in\AinQ(k)\}.$$ Let $\alpha,\alpha'$ be the incoming arrows at $k$.
Since $[p_{\alpha}^-\alpha]\alpha'^*\in I_\mb$ if and only if $\alpha\neq\alpha'$, the arrows at $k^*$ satisfy (g2) and (g3).

\noindent\textbf{Case 2:} $v\neq k^*$ and $v\notin N^-(k)$.
Then  $$\AinT(v)=\{[p]\mid [p]\in{\bf(T4)}, t(p)=v\}$$
and
$$\AoutT(v)=\{[q]\mid q\in{\bf(T4)},s(q)=v\}\cup \{[p_{\alpha}^-\alpha]\mid s(p_{\alpha}^-)=v,  \alpha\in \AoinQ(k)\}.$$
Let $[p] \in \AinT(v)$ and $[q], [p_{\alpha}^- \alpha]\in \AoutT(v)$.
By the definition of $I_\mb$,
$$[p][q] \in I_\mb \iff pq \in \langle I\rangle,\ \ [p][p_{\alpha}^- \alpha] \in I_\mb \iff p p_{\alpha}^- \in \langle I\rangle.$$
Thus, since the arrows of $Q$ at $v$ satisfy (G2) and (G3), the arrows of $Q_T$ at $v$ satisfy (g2) and (g3).

\noindent\textbf{Case 3:} $v\in N^-(k)$.
Then  $$\AinT(v)=\{[p]\mid [p]\in{\bf(T4)}, t(p)=v\}\cup \{\alpha^*\mid \alpha: v\to k\}$$
and
$$\AoutT(v)=\{[q]\mid [q]\in{\bf(T4)},s(q)=v\}\cup \{[p_{\alpha'}^-\alpha']\mid s(p_{\alpha'}^-)=v,  \alpha'\in \AoinQ(k)\}.$$
Let $[p],\alpha^*\in\AinT(v)$ and $[q],[p_{\alpha'}^-\alpha']\in\AoutT(v)$. Then
by Lemma~\ref{l: condition for a*p=0} and the definition of $I_\mb$, we have
\begin{itemize}
\item $[p][q] \in I_\mb \iff pq \in \langle I\rangle$;
\item  $[p][p_{\alpha'}^- \alpha'] \in I_\mb \iff p p_{\alpha'}^- \in \langle I\rangle$;
 \item $\alpha^*[q] \in I_\mb \iff \alpha=\Ini(q) \iff p_\alpha^-q \in \langle I\rangle \text{ (if  $p_{\alpha}^-$ exists)}$;
\item $\alpha^*[p_{\alpha'}^- \alpha'] \in I_\mb \iff \alpha=\Ini(p_{\alpha'}^-)\iff  p_{\alpha}^-p_{\alpha'}^-\in\langle I\rangle \text{ (if  $p_{\alpha}^-$ exists)}$.
\end{itemize}
Note that if  $ pq \in \langle I\rangle$ and $\alpha=\Ini(q) $ both happen, then we must have $p=p_{\alpha}^-$, a contradiction with $[p]\in{\bf(T4)}$. Thus $\alpha^*[q] \in I_\mb$ and $[p][q] \in I_\mb$ cannot both happen. Similar one can get $[p][p_{\alpha'}^- \alpha'] \in I_\mb$ and $\alpha^*[p_{\alpha'}^- \alpha'] \in I_\mb$ cannot both happen. 
Then using the arrows of $Q$ at $v$ satisfy (G2) and (G3), one can get that  the arrows of $Q_T$ at $v$ satisfy (g2) and (g3).
\end{proof}

\subsection{Minimal relations}
In this subsection, we prove that $\langle I_T\rangle=\langle I_\mb\rangle$.

\begin{lem}\label{l: minimal relation 3 means 2}
Let $\alpha, \alpha'$ be incoming arrows at $k$, and let $[q] \in \mathbf{(T4)}$.
\begin{enumerate}
    \item  Assume that  there is no loop at $k$.
    \begin{enumerate}
        \item[(a1)] If $[p_{\alpha}^-\alpha]\alpha^*[q] \in \langle I_T \rangle$, then $\alpha^*[q] \in I_{\mb}$.
        \item[(b1)] If $[p_{\alpha}^-\alpha]\alpha^*[p_{\alpha'}^-\alpha']\alpha'^* \in \langle I_T \rangle$, then $\alpha^*[p_{\alpha'}^-\alpha'] \in I_{\mb}$.
        \item[(c1)] If $[q][p_{\alpha}^-\alpha]\alpha^* \in \langle I_T \rangle$, then $[q][p_{\alpha}^-\alpha] \in I_{\mb}$.
    \end{enumerate}
    
    \item  Assume that  there is a loop $\rho$ at $k$.
    \begin{enumerate}
        \item[(a2)] If $[p_{\alpha}^-\alpha]\rho^*\alpha^*[q] \in \langle I_T \rangle$, then $\alpha^*[q] \in I_{\mb}$.
        \item[(b2)] If $[p_{\alpha}^-\alpha]\rho^*\alpha^*[p_{\alpha'}^-\alpha']\rho^*\alpha'^* \in \langle I_T \rangle$, then $\alpha^*[p_{\alpha'}^-\alpha'] \in I_{\mb}$.
        \item[(c2)] If $[q][p_{\alpha}^-\alpha]\rho^*\alpha^* \in \langle I_T \rangle$, then $[q][p_{\alpha}^-\alpha] \in I_{\mb}$.
    \end{enumerate}
\end{enumerate}
\end{lem}
\begin{proof}
We only prove (1), (2) is similar.

For (a1), by Proposition~\ref{p: path ka is T irreducible}, we have $\varphi_{[p_{\alpha}^-\alpha]}\varphi_{\alpha}^*=\varphi_{p_{\alpha}^-}$ if $[p_{\alpha}^-\alpha]\alpha^*[q] \in \langle I_T \rangle$,
then 
$$0=\phi_{[p_{\alpha}^-\alpha]}\phi_{\alpha^*}\phi_{[q]}=\varphi_{[p_{\alpha}^-\alpha]}\varphi_{\alpha}^*\varphi_{q}=\varphi_{p_{\alpha}^-}\varphi_{q}=\varphi_{p_{\alpha}^-q}.$$ 
So $p_{\alpha}^-q\in\langle I\rangle.$
Let $\beta_1=\Ter(p_{\alpha}^-)$ and   $\beta_2=\Ini(q)$. Since $p_{\alpha}^-$ and $q$ are nonzero, we must have $$\beta_1\beta_2\in I.$$
Since  $[p_{\alpha}^-\alpha]\in\bf(T3)$, we have $p_{\alpha}^-\alpha\in\langle I\rangle$, thus $$\beta_1\alpha\in I.$$ Hence $\beta_2=\alpha$. Since $[q]\in\bf(T4)$, we have $k\notin\{s(q),t(q)\}$. Proposition~\ref{l: condition for a*p=0} then gives $\varphi_{\alpha}^*\varphi_{q}=0$, and therefore $\alpha^*[q]\in I_\mb$.

For (b1), Proposition~\ref{p: path ka is T irreducible} similarly gives
$p_{\alpha}^-p_{\alpha'}^-\in\langle I\rangle$. Arguing as in (a1), we obtain $\varphi_{\alpha}^*\varphi_{p_{\alpha'}^-}=0$. Proposition~\ref{p: path ka is T irreducible}(3) then gives $\varphi_{\alpha}^*\varphi_{[p_{\alpha'}^-\alpha']}=0$. Hence $\alpha^*[p_{\alpha'}^-\alpha']\in I_\mb$.

For (c1),  Proposition~\ref{p: path ka is T irreducible} similarly gives
$qp_{\alpha}^-\in\langle I\rangle$, \ie $\varphi_{q}\varphi_{p_{\alpha}^-}=0$. Then by (3) of Proposition~\ref{p: path ka is T irreducible}, we have $\varphi_{q}\varphi_{[p_{\alpha}^-\alpha]}=0$, then $\phi_{q}\phi_{[p_{\alpha}^-\alpha]}=\varphi_{q}\varphi_{[p_{\alpha}^-\alpha]}=0$, so $[q][p_{\alpha}^-\alpha] \in I_{\mb}$ by Proposition~\ref{p: construction of Ib}.
\end{proof}

\begin{lem}\label{l:p=0-> ab=0}
Let $\omega=\gamma_1\cdots\gamma_s$ be a path in $Q_T$ such that $\omega\in\langle I_T\rangle$. Then there exists $1\leq i\leq s-1$ such that $\gamma_i\gamma_{i+1}\in I_{\mathcal{B}}$.
\end{lem}
\begin{proof}
The condition $\omega\in\langle I_T\rangle$ is equivalent to
$$\phi_\omega=\phi_{\gamma_1}\cdots\phi_{\gamma_s}=0.$$ If $s=2$, then $\phi_{\gamma_1}\phi_{\gamma_2}=0$, and hence $\gamma_1\gamma_2\in I_\mb$ by Proposition~\ref{p: construction of Ib}. Assume now that $s>2$. Note that  $s(\gamma_1)=k^*$ if and only if $\gamma_1=\alpha^*$ for some $\alpha\in\AoinQ(k)$ or $\gamma_1=\rho^*$, $t(\gamma_s)=k^*$ if and only if $\gamma_s=[p_{\alpha}^-\alpha]$  some $\alpha\in\AinQ(k)$  or $\gamma_s=\rho^*$.

\noindent\textbf{Case 1.} There is no loop at $k$. 
If there exists $1\leq i\leq s-1$ such that $\gamma_i\gamma_{i+1}=[p_{\alpha}^-\alpha]\alpha'^*$ for $\alpha,\alpha'\in\AinQ(k)$ with $\alpha\neq\alpha'$, then $\phi_{\gamma_i}\phi_{\gamma_{i+1}}=0$, so $\gamma_i\gamma_{i+1}\in I_\mb$. Hence we may assume that no such pair occurs. 

\noindent\textbf{Subcase 1.1.} Suppose that $s(\gamma_1)\neq k^*$ and $t(\gamma_s)\neq k^*$. Then $\gamma_1\neq\alpha^*$ and $\gamma_s\neq[p_{\alpha'}^-\alpha']$ for all $\alpha,\alpha'\in\AinQ(k)$. Thus, for each $\alpha\in\AinQ(k)$, the pair $[p_{\alpha}^-\alpha]\alpha^*$ must occur consecutively whenever either $[p_{\alpha}^-\alpha]$ or $\alpha^*$ occurs.
For each $\gamma_{i-1}\gamma_i=[p_\alpha^-\alpha]\alpha^*$,
Proposition~\ref{p: path ka is T irreducible} gives $\phi_{\gamma_{i-1}}\phi_{\gamma_i}=\varphi_{[p_\alpha^-\alpha]}\varphi_{\alpha}^*=\varphi_{p_\alpha^-}$.
For each arrow $\gamma_i=[p]\in {\bf(T4)}$, we have $\phi_{[p]}=\varphi_p$. Therefore,
$$0=\phi_{\omega}=\phi_{\gamma_1}\cdots\phi_{\gamma_s}=\varphi_{p_1}\cdots\varphi_{p_t}=\varphi_{p_1\cdots p_t},$$
where, for each $1\leq i\leq t$, either $p_i=p_{\alpha}^-$ for some $\alpha\in\AinQ(k)$ or $[p_i]\in {\bf(T4)}$.
Hence $p_1\cdots p_t\in\langle I\rangle$. Therefore, there exists $1\leq l\leq t-1$ such that $p_lp_{l+1}\in\langle I\rangle$, equivalently, $0=\varphi_{p_l}\varphi_{p_{l+1}}$. According to the types of $p_l$ and $p_{l+1}$, we obtain the following cases.
\begin{enumerate}
    \item If $[p_l],[p_{l+1}]\in\bf(T4)$, then $0=\varphi_{p_l}\varphi_{p_{l+1}}=\phi_{[p_l]}\phi_{[p_{l+1}]}=\phi_{\gamma_i}\phi_{\gamma_{i+1}}$ for some $1\leq i\leq s-1$, hence $\gamma_i\gamma_{i+1}=[p_l][p_{l+1}]\in I_\mb$.
    \item If $p_l=p_{\alpha}^-$ for some $\alpha\in\AinQ(k)$ and  $[p_{l+1}]\in\bf(T4)$, then $0=\varphi_{p_l}\varphi_{p_{l+1}}=\phi_{[p_{\alpha}^-\alpha]}\phi_{\alpha^*}\phi_{[p_{l+1}]}=\phi_{\gamma_i}\phi_{\gamma_{i+1}}\phi_{\gamma_{i+2}}$  for some $1\leq i\leq s-2$, hence $[p_{\alpha}^-\alpha]\alpha^*[p_{l+1}]\in\langle I_T\rangle$.
    By Lemma~\ref{l: minimal relation 3 means 2}, $\alpha^*[p_{l+1}]\in I_\mb$. Thus
     $\gamma_{i+1}\gamma_{i+2}=\alpha^*[p_{l+1}]\in I_\mb$.
 \item  If $[p_{l}]\in\bf(T4)$ and $p_{l+1}=p_{\alpha}^-$ for some $\alpha\in\AinQ(k)$, then  $0=\varphi_{p_l}\varphi_{p_{\alpha}^-}=\varphi_{p_l}\varphi_{[p_{\alpha}^-\alpha]}\varphi_{\alpha^*}=\phi_{[p_l]}\phi_{[p_{\alpha}^-\alpha]}\phi_{\alpha^*}=\phi_{\gamma_i}\phi_{\gamma_{i+1}}\phi_{\gamma_{i+2}}$  for some $1\leq i\leq s-2$.
Hence $[p_l][p_{\alpha}^-\alpha]\alpha^*\in I_{T}$.  By Lemma~\ref{l: minimal relation 3 means 2}, $[p_l][p_{\alpha}^-\alpha]\in I_{\mb}$, thus
$\gamma_i\gamma_{i+1}=[p_l][p_{\alpha}^-\alpha]\in I_{\mb}$.
\item If $p_l=p_{\alpha}^-$  and $p_{l+1}=p_{\alpha'}^-$ for some $\alpha,\alpha'\in\AinQ(k)$, then $0=\varphi_{p_{\alpha}^-}\varphi_{p_{\alpha'}^-}=\phi_{[p_{\alpha}^-\alpha]}\phi_{\alpha^*}\phi_{[p_{\alpha'}^-\alpha']}\phi_{\alpha'^*}=\phi_{\gamma_i}\phi_{\gamma_{i+1}}\phi_{\gamma_{i+2}}\phi_{\gamma_{i+3}}$  for some $1\leq i\leq s-3$,
hence $[p_{\alpha}^-\alpha]\alpha^*[p_{\alpha'}^-\alpha']\alpha'^*\in\langle I_T\rangle$.
By Lemma~\ref{l: minimal relation 3 means 2}, $\alpha^*[p_{\alpha'}^-\alpha']\in I_\mb$. Hence
$\gamma_{i+1}\gamma_{i+2}=\alpha^*[p_{\alpha'}^-\alpha']\in I_\mb$.
\end{enumerate}

\noindent\textbf{Subcase 1.2.} Suppose that $s(\gamma_1)=k^*$ or $t(\gamma_s)=k^*$.
If $t(\gamma_s)=k^*$, then $\gamma_s=[p_{\alpha}^-\alpha]$ for some $\alpha\in\AinQ(k)$, set $\gamma_{s+1}=\alpha^*$. Otherwise, set $\gamma_{s+1}=e_{t(\gamma_s)}$.

\begin{itemize}
\item If $s(\gamma_1)\neq k^*$, set $\omega'=\omega\gamma_{s+1}$.
Then $k\notin\{s(\omega'),t(\omega')\}$. Since $\omega'\in\langle I_T\rangle$, Subcase 1.1 yields $1\leq i\leq s$ such that $\gamma_i\gamma_{i+1}\in I_{\mb}$.
In either case, the terminal pair introduced above is not in $I_\mb$;
hence $1\leq i\leq s-1$, as required.

\item If $s(\gamma_1)=k^*$, then $\gamma_1=\alpha^*$ for some $\alpha\in\AinQ(k)$ and $s(\gamma_2)\neq k^*$. If $\gamma_2\cdots\gamma_{s+1}\in\langle I_T\rangle$, the preceding argument gives $2\leq i\leq s-1$ such that $\gamma_i\gamma_{i+1}\in I_\mb$.
If $\gamma_2\cdots\gamma_{s+1}\notin \langle I_T\rangle$,
we have
$$0=\phi_{\omega\gamma_{s+1}}=\phi_{\alpha^*}\phi_{\gamma_1}\cdots\phi_{\gamma_{s+1}}=\phi_{\alpha^*}\varphi_{p_1}\cdots\varphi_{p_t}=\phi_{\alpha^*}\varphi_{p_1\cdots p_t}=\phi_{\alpha^*}\varphi_p$$
for some nonzero path $p$ in $Q$ with $k\notin\{s(p),t(p)\}$. Lemma~\ref{l: condition for a*p=0} then gives $\alpha=\Ini(p)=\Ini(p_1)$, so $\gamma_2=[\alpha p_{\alpha}^+]$ or $\gamma_2=[p_{\alpha'}^-\alpha']$  with $\Ini(p_{\alpha'}^-)=\alpha$, where $\alpha'\in\AinQ(k)$. In either case, $\gamma_1\gamma_2\in I_\mb$.
\end{itemize}

\noindent\textbf{Case 2.}
Assume that $k$ carries a loop $\rho$, and let $\alpha$ denote the unique non-loop incoming arrow at $k$.
If there exists $1\leq i\leq s-1$ such that $\gamma_i\gamma_{i+1}=[p_{\alpha}^-\alpha]\alpha^*$ or $\gamma_i\gamma_{i+1}=\rho^*\rho^*$, then $\gamma_i\gamma_{i+1}\in I_\mb$, and we are done.
Hence we may assume that no such pair occurs.

\noindent\textbf{Subcase 2.1.}  $s(\gamma_1)\neq k^*$ and $t(\gamma_s)\neq k^*$. 
By the types of arrows at $k^*$, the block $[p_{\alpha}^-\alpha]\rho^*\alpha^*$ must occur consecutively whenever one of $[p_{\alpha}^-\alpha],\rho^*,\alpha^*$ occurs.
For each $\gamma_{i-1}\gamma_{i}\gamma_{i+1}=[p_{\alpha}^-\alpha]\rho^*\alpha^*$,
By Proposition~\ref{p: path ka is T irreducible}, $\phi_{\gamma_{i-1}}\phi_{\gamma_i}\phi_{\gamma_{i+1}}=\varphi_{[p_{\alpha}^-\alpha]}\varphi_{\rho}^*\varphi_{\alpha}^*=\varphi_{p_{\alpha}^-}$. For each $\gamma_i=[p]\in {\bf(T4)}$, we have $\phi_{[p]}=\varphi_p$. The same argument as in Subcase~1.1, together with Lemma~\ref{l: minimal relation 3 means 2}, yields an index $1\leq i\leq s-1$ such that $\gamma_i\gamma_{i+1}\in I_{\mb}$.
\noindent\textbf{Subcase 2.2.} Suppose that $s(\gamma_1)=k^*$ or $t(\gamma_s)=k^*$.
If $\gamma_s=[p_{\alpha}^-\alpha]$, set $\gamma_{s+1}\gamma_{s+2}=\rho^*\alpha^*$. If $\gamma_s=\rho^*$, set $\gamma_{s+1}\gamma_{s+2}=e_{k^*}\alpha^*$. Otherwise, set $\gamma_{s+1}\gamma_{s+2}=e_{t(\omega)}e_{t(\omega)}$.

\begin{itemize}
    \item If $s(\gamma_1)\neq k^*$, set $\omega'=\omega\gamma_{s+1}\gamma_{s+2}$.
By the same argument as in
\textbf{Subcase 1.2}, together with $\rho^*\alpha^*\notin I_\mb$ and $[p_{\alpha}^-\alpha]\rho^*\notin I_\mb$, the result follows.
\item If $s(\gamma_1)=k^*$, then $\gamma_1\gamma_2=\rho^*\alpha^*$ or $\gamma_1=\alpha^*$. The same argument as in \textbf{Subcase 1.2}, together with Lemma~\ref{l: condition for a*p=0}, gives the result.
\end{itemize}
\end{proof}

\begin{lem}\label{l: p1+p2 in I means pi in I}
Let $c,d\in(Q_T)_0$, and let $\omega_1,\ldots,\omega_t$ be distinct paths in $Q_T$ from $c$ to $d$, where $t\geq2$. Suppose that $\omega=\lambda_1\omega_1+\cdots+\lambda_t\omega_t\in\langle I_T\rangle$ for nonzero scalars $\lambda_i\in\bbk$. Then $\omega_l\in\langle I_T\rangle$ for every $1\leq l\leq t$.
\end{lem}
\begin{proof}
We only prove the result when $k$ carries a loop. The loop-free case is analogous. Denote the unique loop by $\rho$ and the unique non-loop incoming arrow at $k$ by $\alpha$.

For each $1\leq l\leq t$, write $\omega_l=\gamma_{l1}\cdots\gamma_{l,s_l}$, where each $\gamma_{lj}\in(Q_T)_1$.
If for some $1\leq l\leq t$, there exists $1\leq i\leq s_l-1$ such that $\gamma_{li}\gamma_{l,i+1}=[p_{\alpha}^-\alpha]\alpha^*$ or $\gamma_{l,i}\gamma_{l,i+1}=\rho^*\rho^*$, then $\gamma_{l,i}\gamma_{l,i+1}\in I_\mb$, so $\omega_l\in\langle I_T\rangle$. Repeating this argument, we may assume that no such pair occurs in any $\omega_l$ for $1\le l\le t$.

Note that  $c=k^*$ if and only if  $\gamma_{l1}\in\{\rho^*,\alpha^*\}$ for every $l$.
Similarly, $d=k^*$ if and only if  $\gamma_{l,s_l}\in\{\rho^*,[p_{\alpha}^-\alpha]\}$ for every $l$.

\noindent\textbf{Case 1.}
Assume that $c\neq k^*$ and $d\neq k^*$.
Then the block $[p_{\alpha}^-\alpha]\rho^*\alpha^*$ must occur consecutively whenever one of $[p_{\alpha}^-\alpha],\rho^*,\alpha^*$ occurs. For each $i,l$ such that  ${\gamma_{l,i-1}}{\gamma_{li}}{\gamma_{l,i+1}}=[p_{\alpha}^-\alpha]\rho^*\alpha^*$,
by Proposition~\ref{p: path ka is T irreducible}, $\phi_{\gamma_{l,i-1}}\phi_{\gamma_{li}}\phi_{\gamma_{l,i+1}}=\varphi_{[p_{\alpha}^-\alpha]}\varphi_{\rho}^*\varphi_{\alpha}^*=\varphi_{p_{\alpha}^-}$. For each $\gamma_{li}=[p]\in {\bf(T4)}$, we have $\phi_{[p]}=\varphi_p$. Therefore,
$$0=\lambda_1\phi_{\omega_1}+\cdots+\lambda_t\phi_{\omega_t}=\lambda_1\varphi_{q_1}+\cdots+\lambda_t\varphi_{q_t}=\varphi_{\lambda_1{q_1}+\cdots+\lambda_t{q_t}}$$
for some paths $q_1,\ldots,q_t$ in $Q$. Since $A$ is gentle, Lemma~\ref{l: p_1+p_2 in I infer pi in I} gives $q_l\in\langle I\rangle$ for every $1\leq l\leq t$. Thus $\phi_{\omega_l}=\varphi_{q_l}=0$, so $\omega_l\in\langle I_T\rangle$ for every $l$.

\noindent\textbf{Case 2.}
Assume that $c=k^*$ and $d=k^*$. Then $\gamma_{l1}\in\{\rho^*,\alpha^*\}$ and $\gamma_{l,s_l}\in\{\rho^*,[p_{\alpha}^-\alpha]\}$ for every $1\le l\le t$.

\noindent\textbf{Subcase 2.1.}
Assume that $\gamma_{l1}=\gamma_{j1}$ and $\gamma_{l,s_l}=\gamma_{j,s_j}$ for all $1\leq l,j\leq t$.

First suppose that
$\gamma_{l1}=\rho^*$ and $\gamma_{l,s_l}=[p_{\alpha}^-\alpha]$ for every $1\leq l\leq t$. Then $\gamma_{l2}=\alpha^*$ for every $l$.
Set $\gamma_{l,s_l+1}=\rho^*$ and $\gamma_{l,s_l+2}=\alpha^*$, and let $\omega'=\omega\rho^*\alpha^*$.
Then $$\omega'=\omega\rho^*\alpha^*=\rho^*\alpha^*(\lambda_1\omega_1'+\cdots+\lambda_t\omega_t')\in\langle I_T\rangle.$$
If $\lambda_1\omega_1'+\cdots+\lambda_t\omega_t'\in\langle I_T\rangle$,
then $s(\omega_l')=t(\alpha^*)\neq k^*$ and $t(\omega_l')=t(\alpha^*)\neq k^*$ for every $l$.
By Case~1, $\omega_l'\in\langle I_T\rangle$ for every $1\leq l\leq t$.
Lemma~\ref{l:p=0-> ab=0} then gives $1\leq j_l\leq s_l+1$ such that $\gamma_{l,j_l}\gamma_{l,j_l+1}\in I_\mb$. Since $\gamma_{l,s_l}\gamma_{l,s_l+1}=[p_{\alpha}^-\alpha]\rho^*\notin I_{\mb}$ and $\gamma_{l,s_l+1}\gamma_{l,s_l+2}=\rho^*\alpha^*\notin I_{\mb}$, we must have $1\leq j_l\leq s_l-1$. Hence $\omega_l\in\langle I_T\rangle$ for every $l$.

If $\lambda_1\omega_1'+\cdots+\lambda_t\omega_t'\notin\langle I_T\rangle$, then
$$0=\phi_{\omega'}=\phi_{\rho^*\alpha^*(\lambda_1\omega_1'+\cdots+\lambda_t\omega_t')}=\varphi_{\rho}^*\varphi_{\alpha}^*\varphi_{\lambda_1q_1'+\cdots+\lambda_tq_t'}$$
for some paths $q_1',\ldots,q_t'$ in $Q$ with $s(q_1')=\cdots=s(q_t')\neq k$ and $t(q_1')=\cdots=t(q_t')\neq k$. By Lemma~\ref{l: condition for a*p=0}, $\alpha=\Ini(q_l')$ for every $l$. Hence $\gamma_{l3}=[\alpha p_{\alpha}^+]$ or $\gamma_{l3}=[p_{\alpha}^-\alpha]$ with $\Ini(p_{\alpha}^-)=\alpha$. Therefore $\gamma_{l2}\gamma_{l3}\in I_\mb$, and so $\omega_l\in\langle I_T\rangle$ for every $l$.
The remaining possibilities are treated in the same way.

\noindent\textbf{Subcase 2.2.}
Assume that there exist $l\neq j$ with $\gamma_{l1}\neq\gamma_{j1}$, while $\gamma_{l,s_l}=\gamma_{j,s_j}$ for all $l,j$.

Without loss of generality, suppose that $\gamma_{l1}=\rho^*$ for $1\leq l\leq m$ and $\gamma_{l1}=\alpha^*$ for $m+1\leq l\leq t$, where $1\leq m<t$.
Assume further that, for every $1\leq l\leq t$,
$\gamma_{l,s_l}=[p_{\alpha}^-\alpha]$.

Set $\gamma_{l0}=[p_{\alpha}^-\alpha]$, $\gamma_{l,s_l+1}=\rho^*$, $\gamma_{l,s_l+2}=\alpha^*$, and $\omega_l'=[p_{\alpha}^-\alpha]\omega_l\rho^*\alpha^*$ for every $1\leq l\leq t$.
Then $$\omega'=[p_{\alpha}^-\alpha]\omega\rho^*\alpha^*=\lambda_1\omega_1'+\cdots+\lambda_t\omega_t'\in\langle I_T\rangle.$$
Since $[p_{\alpha}^-\alpha]\alpha^*\in I_\mb$, we have
$\omega_l'\in\langle I_T\rangle$ for every $m+1\leq l\leq t$.
Therefore, $$\omega''=\omega'-\sum_{j=m+1}^t\lambda_j\omega_j'=\lambda_1\omega_1'+\cdots+\lambda_m\omega_m'\in\langle I_T\rangle.$$
Since $s(\omega'')\neq k^*$ and $t(\omega'')\neq k^*$, Case~1 gives
$\omega_l'\in\langle I_T\rangle$ for every $1\leq l\leq m$.
For each $1\leq l\leq m$, Lemma~\ref{l:p=0-> ab=0} gives $0\leq j_l\leq s_l+1$ such that $\gamma_{l,j_l}\gamma_{l,j_l+1}\in I_\mb$. Since $\gamma_{l,0}\gamma_{l,1}=[p_{\alpha}^-\alpha]\rho^*\notin I_{\mb}$, $\gamma_{l,s_l}\gamma_{l,s_l+1}=[p_{\alpha}^-\alpha]\rho^*\notin I_{\mb}$, and $\gamma_{l,s_l+1}\gamma_{l,s_l+2}=\rho^*\alpha^*\notin I_{\mb}$, we must have $1\leq j_l\leq s_l-1$. Hence $\omega_l\in\langle I_T\rangle$, for $1\le l\le m$.

It remains to show that $\omega_l\in\langle I_T\rangle$ for $m+1\leq l\leq t$.
Set $$\omega''=[p_{\alpha}^-\alpha]\rho^*\omega\rho^*\alpha^*=\lambda_1\omega_1''+\cdots+\lambda_t\omega_t''\in\langle I_T\rangle.$$
Using $\rho^*\rho^*\in I_\mb$ and the same argument as above, we obtain $\omega_l\in\langle I_T\rangle$ for every $m+1\leq l\leq t$.

\noindent\textbf{Subcase 2.3.}
Assume that $\gamma_{l,s_l}\neq\gamma_{j,s_j}$ for some $l\neq j$, while $\gamma_{l1}=\gamma_{j1}$ for all $l,j$. The argument is analogous to that of Subcase~2.2.

\noindent\textbf{Subcase 2.4.}
Assume that $\gamma_{a1}\neq\gamma_{b1}$ for some $a\neq b$ and that $\gamma_{l,s_l}\neq\gamma_{j,s_j}$ for some $l\neq j$.

Without loss of generality, assume that $\gamma_{a1}=\rho^*$, $\gamma_{b1}=\alpha^*$, it is obvious that $\gamma_{a2}=\alpha^*=\gamma_{b1}$. Similarly, assume that $\gamma_{l,s_l}=\rho^*,\gamma_{j,s_j}=[p_{\alpha}^-\alpha]$, furthermore, $\gamma_{l,s_l-1}=[p_{\alpha}^-\alpha]=\gamma_{j,s_j}$. 

We consider the cases where $\gamma_{a1}\ne\gamma_{b1},\gamma_{a,s_a}\ne\gamma_{b,s_b}$, the other cases are analogous to Subcase 2.2, 2.3.
Specifically, let $\gamma_{a1}=\rho^*,\gamma_{b_1}=\alpha^*$, and $\gamma_{a,s_a}, \gamma_{b,s_b}\in\{\rho^*,[p_\alpha^-\alpha]\}$ with $\gamma_{a,s_a}\ne \gamma_{b,s_b}$. 

When  $\gamma_{a,s_a}=\rho^*,\gamma_{b,s_b}=[p_\alpha^-\alpha]$,  we obtain $\omega_a=\rho^*\omega_b\rho^*$. 
When $\gamma_{a,s_a}=[p_\alpha^-\alpha],\gamma_{b,s_b}=\rho^*$, then  $\omega_a=\rho^*\omega_a',\omega_b=\omega_b'\rho^*,$ for $\omega_a'=\gamma_{a,2}\dots\gamma_{a,s_a}=\omega_{b1}\dots\omega_{b,s_b-1}=\omega_b'$. 
 No matter in which cases the arguments are analogous to that of Subcase 2.2.

\noindent\textbf{Case 3.}
Assume that $c=k^*$ and $d\neq k^*$. The argument is analogous to that of Case~2.

\noindent\textbf{Case 4.}
Assume that $c\neq k^*$ and $d=k^*$. The argument is analogous to that of Case~2.
\end{proof}

We can now state the main result.
\begin{thm}\label{t: qt-ib is a gentle pair}
$(Q_T,I_\mb)$ is a gentle pair and
    $$\End_A(T)\cong \bbk Q_T/\langle I_\mb\rangle.$$
\end{thm}
\begin{proof}
By definition, $I_\mb\subseteq I_T$. Lemmas~\ref{l: Q' satisfy g2-g3}, \ref{l:p=0-> ab=0}, and \ref{l: p1+p2 in I means pi in I} imply that $\langle I_T\rangle=\langle I_\mb\rangle$. This proves the theorem.
\end{proof}

\section{A combinatorial description of tilting mutations}\label{s:s7 ex}

Let $(Q,I)$ be a gentle pair. In this section, let $k\in Q_0$ be a vertex at which the tilting mutation exists. For an incoming non-loop arrow $\alpha$ at $k$, we recall the following paths associated with $\alpha$.
\begin{itemize}
    \item $p_{\alpha}^-$ denotes the shortest nonzero path ending at $s(\alpha)$ such that $p_{\alpha}^- \alpha \in \langle I\rangle $ and $s(p_{\alpha}^-) \neq k$ (if such a path exists).
     \item $p_{\alpha}^+$ denotes the shortest path starting at $k$ such that $\alpha p_{\alpha}^+\notin\langle I\rangle$,  $t(p_{\alpha}^+)\neq k$ and for every arrow $\beta: t(p_{\alpha}^+)\to k$, $p_{\alpha}^+ \beta \notin \langle I\rangle $ (if such a path exists).
\end{itemize}
\subsection{General case}
\begin{definition}\label{d: tilting mutation in general}
 We define a new pair $(Q',I')=\mu_k^+(Q,I)$, called the \emph{tilting mutation} of $(Q,I)$ at $k$, as follows.
  \begin{itemize}
 \item[\textbf{Step 1}:] Set $Q'=Q\cup \{k^*\}$ and $I'=I$;
  \item[\textbf{Step 2}:] For each arrow $ a\xra{\alpha} k$ in $Q_1$, 
  \begin{itemize}
      \item replace $\alpha$ in $Q_1'$ by a new arrow $\alpha^*:k^*\ra a$;
      \item delete  $\alpha\beta$ from $I'$ if there is an outgoing arrow $\beta$ at $k$ such that $\alpha\beta\in I$;
      \item if $\alpha$ is a loop $\rho$, delete $\rho^2$ from $I'$ and add $(\rho^*)^2$ to $I'$.
  \end{itemize}
    \item[\textbf{Step 3}:]
    For each arrow $\alpha \colon a \to k$  (with $a \neq k$) in $Q_1$, if $p_{\alpha}^-$ exists, let $\gamma_{\alpha}=\Ter(p_{\alpha}^-)$.
    \begin{itemize}
    \item replace  $\gamma_{\alpha}$ in $Q_1'$ by a new arrow $[p_{\alpha}^-\alpha]\colon s(p_{\alpha}^-)\to k^*$;
    \item delete the relation $\gamma_{\alpha}\alpha$ from $I'$;
    \item if there is no loop at $k$, add the relation $[p_{\alpha}^-\alpha]\alpha'^*$ to $I'$ (if the other incoming arrow $\alpha'$ in $Q_1$ at $k$ exists); 
    \item if there is a loop at $k$, add the relation $[p_{\alpha}^-\alpha]\alpha^*$ to $I'$;
    \item if $p_{\alpha}^-=\gamma_{\alpha}$ and a relation $\beta\gamma_{\alpha}\in I'$ exists, 
    replace it by $\beta[p_{\alpha}^-\alpha]$;
    \item if $\Ini(p_{\alpha}^-)=\alpha''$ for some incoming arrow $\alpha''$ at $k$, add the relation $\alpha''^*[p_{\alpha}^-\alpha]$ to $I'$.
\end{itemize}

 \item[\textbf{Step 4}:] 
For each arrow $\alpha \colon a \to k$ (with $a \neq k$) in $Q_1$, if $p_{\alpha}^+$ exists,
let $\omega_{\alpha}=\Ter(p_{\alpha}^+)$. 
   \begin{itemize}
        \item replace $\omega_{\alpha}$ in $Q_1'$ by a new arrow $[\alpha p_{\alpha}^+]:a\ra t(p_{\alpha}^+)$;
       \item add the relation $\alpha^*[\alpha p_{\alpha}^+]$ to $I'$;
       \item 
    if a relation $\omega_{\alpha}\gamma\in I'$ exists, replace it by $[\alpha p_{\alpha}^+]\gamma$.
   \end{itemize}
\item[\textbf{Step 5}:] Delete $k$ from $Q'$ and replace $k^*$ by $k$.

\end{itemize}
We denote the resulting gentle pair by $(Q',I')=\mu_k^+(Q,I)=(\mu_k^+(Q),\mu_k^+(I))$.

\end{definition}

The main result of this section is the following.
\begin{thm}\label{t: tilitng mutation in general}
 Let $A=\bbk Q/\langle I\rangle$ be a gentle algebra. 
Let $k\in Q_0$ be such that, for every outgoing arrow $\beta$ at $k$, there is an incoming arrow $\alpha$ at $k$ with $\alpha\beta\notin I$. Let $T=\overline{T}\oplus P(k)^*$, then
\begin{enumerate}
    \item $\mu_k^+(Q,I)$ is a gentle pair;
    \item $(Q_T,I_T)=\mu_k^+(Q,I)$, i.e., $\End(T)\cong \bbk\mu_k^+(Q)/\langle \mu_k^+(I)\rangle$.
\end{enumerate}

\end{thm}

\begin{proof}  
    Set $(Q',I')=\mu_k(Q,I)$. Since $(Q_T,I_T)=(Q_T,I_\mb)$, it is enough to prove that
 $(Q_T,I_\mb)=(Q',I')$.   By construction, $Q_T=Q'$. Indeed, in the construction of $Q_T$, $\bf (T1)$ and $\bf(T2)$ correspond to \textbf{Step 2},
    $\bf(T3)$ corresponds to \textbf{Step 3}, and $\bf(T4)$ corresponds to \textbf{Step 4}.

For the relations, it is clear that $I'\subset I_\mb$. For the reverse inclusion, by the construction of $I'$ and Proposition~\ref{p: construction of Ib}, we only need to prove that when there is no loop at $k$, 
$$ \{[\alpha' p_{\alpha'}^+][p_{\alpha}^-\alpha]\mid \alpha',\alpha\in\AinQ(k),  t(p_{\alpha'}^+)=s(p_{\alpha}^-), p_{\alpha'}^+p_{\alpha}^-\in \langle I\rangle\} \subset I'.$$
Let $\alpha':a'\to k$ and $\alpha:a\to k$ be non-loop arrows in $\AinQ(k)$ such that $$a\neq a',  t(p_{\alpha'}^+)=s(p_{\alpha}^-),  p_{\alpha'}^+p_{\alpha}^-\in \langle I\rangle.$$ Since $p_{\alpha'}^+$ ends at $s(p_\alpha^-)$, by the definition of $p_{\alpha'}^+$, 
$\Ini(p_{\alpha}^-)$ is not an incoming arrow at $k$. Then by $p_{\alpha}^-\in\mpik(s(p_{\alpha}^-),a))$, one can get that  $p_{\alpha}^-$ is an arrow. 
Then after \textbf{Step 3}, we have $\Ter(p_{\alpha'}^+)[p_{\alpha}^-\alpha]\in I'$. After \textbf{Step 4}, we have $$[\alpha' p_{\alpha'}^+][p_{\alpha}^-\alpha]\in I'.$$
Therefore $I_T=I_\mb=I'$. Thus
$$(Q_T,I_T)=(Q',I')=\mu_k^+(Q,I).$$
In particular, $(Q',I')=\mu_k^+(Q,I)$ is a gentle pair since $(Q_T,I_T)$ is a gentle pair.
\end{proof}
\begin{remark}
Holm, Schr\"oer and Zimmermann~\cite{HSZ01} (see also the appendix of~\cite{AA08}) have described combinatorial transformations of the quiver and relations of a gentle algebra that yield derived equivalences via tilting complexes.  The construction here is a version of generalized BB-tilting module.
\end{remark}

\subsection{Without a loop or a 2-cycle}

Suppose that there is neither a loop nor a 2-cycle at $k$. Then the tilting mutation at $k$ simplifies as follows:   
\begin{itemize}
 \item[\textbf{Step 1}:] Set $Q'=Q$ and $I'=I$;
  \item[\textbf{Step 2}:] For each arrow $ a\xra{\alpha} k$ in $Q_1$, 
  \begin{itemize}
      \item replace $\alpha$ in $Q_1'$ by $\alpha^*:k\ra a$;
      \item delete $\alpha\beta$ from $I'$ if there is an outgoing arrow $\beta$ at $k$ such that $\alpha\beta\in I$;
  \end{itemize} 
    
    \item[\textbf{Step 3}:] For each pair of arrows $c\xra{\gamma} a\xra{\alpha} k$ in $Q_1$ with $\gamma\alpha\in I$, 
    \begin{itemize}
          \item replace  $\gamma$ in $Q_1'$ by a new arrow $[\gamma\alpha]: c\ra k$;
        \item delete $\gamma\alpha$ from $I'$;
        \item add $[\gamma\alpha]\alpha'^*$  to $I'$ if there is another  incoming arrow $\alpha'$   in $Q_1$  at $k$; 
          \item   replace each $\beta\gamma$  in $I'$ by $\beta[\gamma\alpha]$.
    
    \end{itemize}

 \item[\textbf{Step 4}:] For each pair of arrows $a\xra{\alpha}k\xra{\beta}c$ in $Q_1$ with $\alpha\beta\notin I$, 
   \begin{itemize}
        \item replace $\beta$ in $Q_1'$ by a new arrow $[\alpha\beta]: a\ra c$;
       \item add a new relation $\alpha^*[\alpha\beta]$ to $I'$;
    \item replace each $\beta\delta$ in  $I'$ by $[\alpha\beta]\delta$; 
   \end{itemize}

\end{itemize}

\begin{example}\label{eg:mutation in without loop and 2-cycle}
 Let $A=\bbk  Q/\langle I\rangle$ be a gentle algebra as follows:
\[\begin{tikzcd}[sep=small]
	&&&& b && \\
	\\
	c && a && k && n \\
	\\
	&&&& d
	\arrow["{\alpha_2}", from=1-5, to=3-5]
	\arrow["\gamma", from=3-1, to=3-3]
	\arrow["{\alpha_1}", from=3-3, to=3-5]
	\arrow["{\beta_1}", from=3-5, to=3-7]
	\arrow["{\beta_2}", from=3-5, to=5-5]
\end{tikzcd}\]
where  $I=\{\gamma\alpha_1,\alpha_1\beta_1,\alpha_2\beta_2\}$.
There is neither a loop nor a 2-cycle at $k$, then according to the above rules, we have $(Q',I')$ as follows:
\begin{itemize}
    \item [\textbf{Step 1}:] $Q'=Q, I'=I$;
    \item [\textbf{Step 2}:] After this step, $Q'$ is 
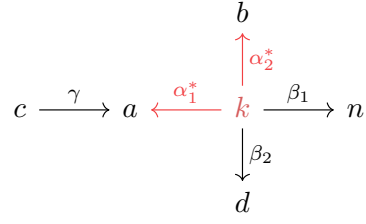
\[\begin{tikzcd}[sep=small]
	&&&& b && \\
	\\
	c && a &&\textcolor{rgb,255:red,214;green,92;blue,92} k && n \\
	\\
	&&&& d
	\arrow["{{\alpha_2^*}}", color={rgb,255:red,242;green,64;blue,67}, tail reversed, no head, from=1-5, to=3-5]
	\arrow["\gamma", from=3-1, to=3-3]
	\arrow["{{\alpha_1^*}}", color={rgb,255:red,242;green,64;blue,67}, tail reversed, no head, from=3-3, to=3-5]
	\arrow["{{\beta_1}}", from=3-5, to=3-7]
	\arrow["{{\beta_2}}", from=3-5, to=5-5]
\end{tikzcd}\]
and  $I'=\{\gamma\alpha_1\}$.
\item [\textbf{Step 3}:]
Using $\gamma\alpha_1\in I$, one can get the new quiver $Q'$ is
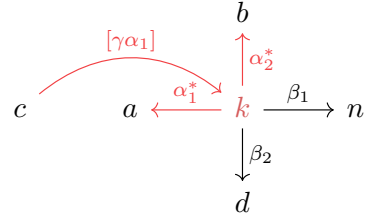
\[\begin{tikzcd}[sep=small]
	&&&& b && \\
	\\
	c && a &&\textcolor{rgb,255:red,214;green,92;blue,92} k && n \\
	\\
	&&&& d
	\arrow["{{\alpha_2^*}}", color={rgb,255:red,242;green,64;blue,67}, tail reversed, no head, from=1-5, to=3-5]
	\arrow["{[\gamma\alpha_1]}", color={rgb,255:red,242;green,64;blue,67}, curve={height=-24pt}, from=3-1, to=3-5]
	\arrow["{{\alpha_1^*}}", color={rgb,255:red,242;green,64;blue,67}, tail reversed, no head, from=3-3, to=3-5]
	\arrow["{{\beta_1}}", from=3-5, to=3-7]
	\arrow["{{\beta_2}}", from=3-5, to=5-5]
\end{tikzcd}\]
and $I'=\{[\gamma\alpha_1]\alpha_2^*\}$.
\item [\textbf{Step 4}:] Since $\alpha_1\beta_2\notin I, \alpha_2\beta_1\notin I$,  \textbf{Step 4} gives $[\alpha_1\beta_2],~[\alpha_2\beta_1]$, respectively: $Q'$ is
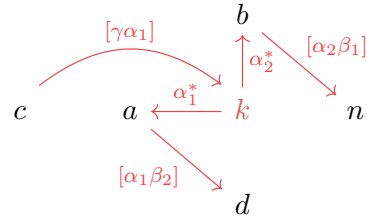
\[\begin{tikzcd}[sep=small]
	&&&& b && \\
	\\
	c && a && \textcolor{rgb,255:red,214;green,92;blue,92}{k} && n \\
	\\
	&&&& d
	\arrow["{{{\alpha_2^*}}}",color={rgb,255:red,242;green,64;blue,67}, tail reversed, no head, from=1-5, to=3-5]
	\arrow["{{{[\alpha_2\beta_1]}}}",color={rgb,255:red,242;green,64;blue,67}, from=1-5, to=3-7]
	\arrow["{{[\gamma\alpha_1]}}", shift left=2, color={rgb,255:red,242;green,64;blue,67}, curve={height=-18pt}, from=3-1, to=3-5]
	\arrow["{{{\alpha_1^*}}}",color={rgb,255:red,242;green,64;blue,67}, tail reversed, no head, from=3-3, to=3-5]
	\arrow["{{{[\alpha_1\beta_2]}}}"', color={rgb,255:red,242;green,64;blue,67}, from=3-3, to=5-5]
\end{tikzcd}\]
and $I'=\{[\gamma\alpha_1]\alpha_2^*,~\alpha_1^*[\alpha_1\beta_2],~\alpha_2^*[\alpha_2\beta_1]\}.$
\end{itemize}
 \end{example}

\subsection{With a loop but no 2-cycle}

Suppose that there is a loop $\rho$ but no 2-cycle at $k$. Denote the unique non-loop incoming arrow at $k$ by $\alpha: a\to k$.
Then the tilting mutation at $k$ simplifies as follows:   
\begin{itemize}
 \item[\textbf{Step 1}:] Set $Q'=Q$ and $I'=I$;
  \item[\textbf{Step 2}:] For $\rho$ and $\alpha$,
  \begin{itemize}
  \item replace $k\xra{\rho}k$ in $Q'$ by $k\xra{\rho^*}k$;
      \item  replace $\alpha: a\ra k$ in $Q'$ by $\alpha^*:k\ra a$;
      \item delete $\alpha\beta$ from $I'$ if there is an outgoing arrow $\beta$ at $k$ such that $\alpha\beta\in I$;
      \item replace $\rho^2$ in $I'$ by $(\rho^*)^2$.
  \end{itemize}
    \item[\textbf{Step 3}:] For each pair of arrows $c\xra{\gamma} a\xra{\alpha} k$ in $Q_1$ with $\gamma\alpha\in I$, 
    \begin{itemize}
          \item replace  $\gamma$ in $Q_1'$ by a new arrow $[\gamma\alpha]: c\ra k$;
        \item delete $\gamma\alpha$ from $I'$ and add a new relation $[\gamma\alpha]\alpha^*\in I'$;
          \item     replace each $\beta\gamma$  in $I'$ by $\beta[\gamma\alpha]$.
    \end{itemize}

 \item[\textbf{Step 4}:] For each configuration of arrows $a\xra{\alpha}k\xra{\rho}k\xra{\beta}c$ in $Q_1$, 
   \begin{itemize}
        \item replace  $\beta$ in $Q_1'$ by a new arrow $[\alpha\rho\beta]: a\ra c$;
       \item  add a new relation $\alpha^*[\alpha\rho\beta]$ to $I'$; 
       \item 
replace each $\beta\delta$ in  $I'$ by $[\alpha\rho\beta]\delta$. 
   \end{itemize}

\end{itemize}

\begin{example}\label{eg:mutation in with loop and no 2-cycle}
When $(Q,I)$ as follows:
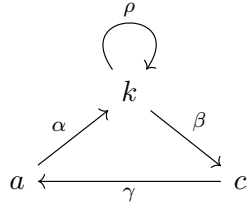
\[\begin{tikzcd}[sep=small]
	&& k && \\
	\\
	a &&&& c
	\arrow["\rho", from=1-3, to=1-3, loop, in=55, out=125, distance=10mm]
	\arrow["\beta", from=1-3, to=3-5]
	\arrow["\alpha", from=3-1, to=1-3]
	\arrow["\gamma", from=3-5, to=3-1]
\end{tikzcd}\]
with $I=\{\alpha\beta,\beta\gamma,\gamma\alpha,\rho^2\}$. Then $(Q',I')$ can be obtained from $(Q,I)$ by the following steps:
\begin{itemize}
  \item[\textbf{Step 1}:] $Q'=Q, I'=I$;
  \item[\textbf{Step 2}:]There is a non-loop arrow $ a\xra{\alpha} k$ and a loop $ k\xra{\rho} k$ in $Q_1$, thus after this step, $Q'$ is 
  
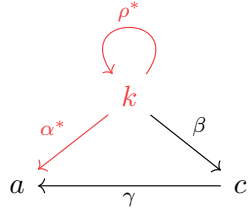
\[\begin{tikzcd}[sep=small]
	&& \textcolor{rgb,255:red,244;green,62;blue,65}{k} && \\
	\\
	a &&&& c
	\arrow["{\rho^*}", color={rgb,255:red,244;green,62;blue,65}, tail reversed, no head, from=1-3, to=1-3, loop, in=55, out=125, distance=10mm]
	\arrow["{\alpha^*}"', color={rgb,255:red,244;green,62;blue,65}, from=1-3, to=3-1]
	\arrow["\beta", from=1-3, to=3-5]
	\arrow["\gamma", from=3-5, to=3-1]
\end{tikzcd}\]
and $I'=\{\beta\gamma,\gamma\alpha,(\rho^*)^2\}$.
    \item[\textbf{Step 3}:] There is $c\xra{\gamma} a\xra{\alpha} k$ with $\gamma\alpha\in I$, thus after this step, $Q'$ is 
  
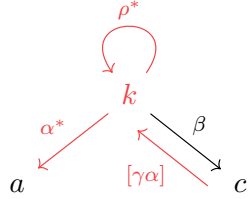
\[\begin{tikzcd}[sep=small]
	&& \textcolor{rgb,255:red,244;green,62;blue,65}{k} && \\
	\\
	a &&&& c
	\arrow["{\rho^*}", color={rgb,255:red,244;green,62;blue,65}, tail reversed, no head, from=1-3, to=1-3, loop, in=55, out=125, distance=10mm]
	\arrow["{\alpha^*}"', color={rgb,255:red,244;green,62;blue,65}, from=1-3, to=3-1]
	\arrow["\beta", from=1-3, to=3-5]
	\arrow["{[\gamma\alpha]}", shift left=3, color={rgb,255:red,244;green,62;blue,65}, from=3-5, to=1-3]
\end{tikzcd}\]
and $I'=\{\beta[\gamma\alpha],[\gamma\alpha]\alpha^*,(\rho^*)^2\}$,

 \item[\textbf{Step 4}:] There is a non zero path $a\xra{\alpha}k\xra{\rho}k\xra{\beta}c$ in $Q_1$, thus after this step, $Q'$ is 
   
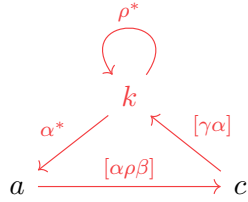
\[\begin{tikzcd}[sep=small]
	&& \textcolor{rgb,255:red,244;green,62;blue,65}{k} && \\
	\\
	a &&&& c
	\arrow["{\rho^*}", color={rgb,255:red,244;green,62;blue,65}, tail reversed, no head, from=1-3, to=1-3, loop, in=55, out=125, distance=10mm]
	\arrow["{\alpha^*}"', color={rgb,255:red,244;green,62;blue,65}, from=1-3, to=3-1]
	\arrow["{[\alpha\rho\beta]}", color={rgb,255:red,244;green,62;blue,65}, from=3-1, to=3-5]
	\arrow["{[\gamma\alpha]}"', color={rgb,255:red,244;green,62;blue,65}, from=3-5, to=1-3]
\end{tikzcd}\]
and $I'=\{\alpha^*[\alpha\rho\beta],[\alpha\rho\beta][\gamma\alpha],[\gamma\alpha]\alpha^*,(\rho^*)^2\}$.
\end{itemize}
 \end{example}

\subsection{With a 2-cycle but no loop}

\begin{example}\label{eg:mutation in without loop and have 2-cycle}
     Let $(Q,I)$ be a gentle pair, where $Q$ is as follows:
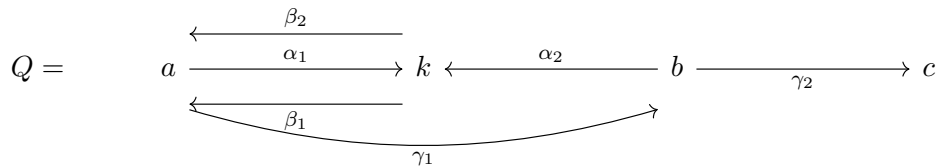
\[\begin{tikzcd}
	{Q=} & a &&& k &&& b &&& c
	\arrow["{{{{\alpha_1}}}}", from=1-2, to=1-5]
	\arrow["{{{{\beta_2}}}}", shift left=5, tail reversed, no head, from=1-2, to=1-5]
	\arrow["{\gamma_1}"', shift right=5, curve={height=18pt}, from=1-2, to=1-8]
	\arrow["{{{{\beta_1}}}}", shift left=5, from=1-5, to=1-2]
	\arrow["{{{{\alpha_2}}}}", tail reversed, no head, from=1-5, to=1-8]
	\arrow["{\gamma_2}"', from=1-8, to=1-11]
\end{tikzcd}\]
with  $I=\{\alpha_2\beta_2,\beta_2\alpha_1,\alpha_1\beta_1,\beta_1\gamma_1,\gamma_1\alpha_2\}$. Then $(Q',I')$ is obtained from the following steps:
\begin{itemize}
 \item[\textbf{Step 1}:] $Q'=Q, I'=I$;
  \item[\textbf{Step 2}:] Since there are two non-loop incoming arrows $\alpha_1,\alpha_2$ at $k$ with $\alpha_1\beta_1,\alpha_2\beta_2\in I$, it follows that
  \[\begin{tikzcd}
	{Q'=} & a &&& \textcolor{rgb,255:red,214;green,92;blue,92}k &&& b &&& c
	\arrow["{{{{\alpha_1^*}}}}", color={rgb,255:red,240;green,66;blue,78}, tail reversed, no head, from=1-2, to=1-5]
	\arrow["{{{{\beta_2}}}}", shift left=5, tail reversed, no head, from=1-2, to=1-5]
	\arrow["{\gamma_1}"', shift right=5, curve={height=18pt}, from=1-2, to=1-8]
	\arrow["{{{{\beta_1}}}}", shift left=5, from=1-5, to=1-2]
	\arrow["{{{{\alpha_2^*}}}}", color={rgb,255:red,240;green,66;blue,78}, from=1-5, to=1-8]
	\arrow["{\gamma_2}"', from=1-8, to=1-11]
\end{tikzcd}\]
where $I'=\{\beta_2\alpha_1,\beta_1\gamma_1,\gamma_1\alpha_2\}$;
\item[\textbf{Step 3}:] Since there is a path $p_{\alpha_1}^-=\alpha_1\beta_2\in\mathcal{P}_k(a,a)$ with $p_{\alpha_1}^-\alpha_1\in \langle I\rangle$ and $p_{\alpha_2}^-=\gamma_1$ with $p_{\alpha_2}^-\alpha_2\in\langle I\rangle$, we obtain:
\[\begin{tikzcd}
	{Q'=} & a &&& \textcolor{rgb,255:red,214;green,92;blue,92}k &&& b &&& c
	\arrow["{{{{{\alpha_1^*}}}}}", color={rgb,255:red,240;green,66;blue,78}, tail reversed, no head, from=1-2, to=1-5]
	\arrow["{[\alpha_1\beta_2\alpha_1]}" , shift left=5, color={rgb,255:red,240;green,66;blue,78}, from=1-2, to=1-5]
	\arrow["{[\gamma_1\alpha_2]}"', shift right=5, color={rgb,255:red,240;green,66;blue,78}, curve={height=24pt}, from=1-2, to=1-5]
	\arrow["{{{{{\beta_1}}}}}", shift left=5, from=1-5, to=1-2]
	\arrow["{{{{{\alpha_2^*}}}}}", color={rgb,255:red,240;green,66;blue,78}, from=1-5, to=1-8]
	\arrow["{{\gamma_2}}"', from=1-8, to=1-11]
\end{tikzcd}\]
$I'=\{[\alpha_1\beta_2\alpha_1]\alpha_2^*,\alpha_1^*[\alpha_1\beta_2\alpha_1],[\gamma_1\alpha_2]\alpha_1^*,\beta_1[\gamma_1\alpha_2]\}.$
 \item[\textbf{Step 4}:] Since there is a path $\beta_1:k\to a$ such that $\alpha_2\beta_1\notin\langle I\rangle$ and for $\alpha_1:a\to k$, $\beta_1\alpha_1\notin\langle I\rangle$, we obtain:
 \[\begin{tikzcd}
	{Q'=} & a &&& \textcolor{rgb,255:red,240;green,66;blue,78}{k} &&& b &&& c
	\arrow["{{{{{\alpha_1^*}}}}}", color={rgb,255:red,240;green,66;blue,78}, tail reversed, no head, from=1-2, to=1-5]
	\arrow["{[\alpha_1\beta_2\alpha_1]}", shift left=5, color={rgb,255:red,240;green,66;blue,78}, from=1-2, to=1-5]
	\arrow["{[\gamma_1\alpha_2]}"', shift right=5, color={rgb,255:red,240;green,66;blue,78}, from=1-2, to=1-5]
	\arrow["{{{{{\alpha_2^*}}}}}", color={rgb,255:red,240;green,66;blue,78}, from=1-5, to=1-8]
	\arrow["{[\alpha_2\beta_1]}"', shift left=4, color={rgb,255:red,240;green,66;blue,78}, curve={height=-30pt}, from=1-8, to=1-2]
	\arrow["{{\gamma_2}}"', from=1-8, to=1-11]
\end{tikzcd}\]
with $I'=\{\alpha_2^*[\alpha_2\beta_1],[\alpha_1\beta_2\alpha_1]\alpha_2^*,\alpha_1^*[\alpha_1\beta_2\alpha_1],[\gamma_1\alpha_2]\alpha_1^*,[\alpha_2\beta_1][\gamma_1\alpha_2]\}.$
\end{itemize}
 \end{example}
\begin{example}
     When the gentle pair $(Q,I)$ are as follows: 
   \[\begin{tikzcd}[sep=small]
	{Q_1=} & a &&& k &&& \\
	\\
	{Q_2=} & a &&& k &&& b
	\arrow["\gamma"', shift right=3, tail reversed, no head, from=1-2, to=1-5]
	\arrow["{\alpha}", shift left=4, from=1-2, to=1-5]
	\arrow["{\alpha_1}", shift left=3, from=3-2, to=3-5]
	\arrow["{\gamma_1}", shift left=3, from=3-5, to=3-2]
	\arrow["{\alpha_2}", shift left=3, tail reversed, no head, from=3-5, to=3-8]
	\arrow["{\gamma_2}"', shift right=3, from=3-5, to=3-8]
\end{tikzcd}\]
with $I_1=\{\gamma\alpha\}$ and $I_2=\{\alpha_1\gamma_1,\gamma_1\alpha_1,\alpha_2\gamma_2,\gamma_2\alpha_2\}$ or $I_2=\{\gamma_1\alpha_1,\alpha_1\gamma_2,\gamma_2\alpha_2,\alpha_2\gamma_1\}.$
Then $$\mu_k^+(Q_i,I_i)=(Q_i,I_i),$$ for each $i=1,2$.
 \end{example}


\subsection{With a loop and a 2-cycle}

\begin{prop}
Suppose that there is a loop $\rho$ at $k$ and a 2-cycle at $k$. Then $\mu_k^+(Q,I)=(Q,I)$, i.e., $$\End_A(T)\cong A=\bbk Q/\langle I\rangle.$$
\end{prop}
\begin{proof}
   Suppose that the loop and the 2-cycle incident to $k$ are as follows:
\[\begin{tikzcd}[sep=small]
Q=	&\cdots\cdots  a &&& k
	\arrow["\alpha",shift left=2, from=1-2, to=1-5]
	\arrow["\beta", shift left=2, from=1-5, to=1-2]
	\arrow["\rho", from=1-5, to=1-5, loop, in=55, out=125, distance=10mm]
\end{tikzcd}\]
Since $A$ is finite-dimensional and there is a loop $\rho$ at $k$, we obtain 
 $\{\alpha\beta,\beta\alpha,\rho^2\}\subset I$. Denote by $\mu_k^+(Q,I)=(Q',I')$.
 
 Note that since there is a loop $\rho$ and a non-loop incoming arrow $\alpha$ at $k$,  it follows that after $\textbf{Step 2}$, there is a loop $\rho^*$ and an outgoing arrow $\alpha^*$ at $k$ in the new quiver $Q'$ with $(\rho^*)^2\in I'$ and $I'=I\setminus\{\alpha\beta,\rho^2\}\cup\{(\rho^*)^2\}.$ Furthermore, because there is a nonzero path $p_{\alpha}^-=\alpha\rho\beta\in\mathcal{P}_k(a,a)$ such that $p_\alpha^-\alpha\in\langle I\rangle$ with $\beta=\Ter(p_{\alpha}^-)$, then  after $\textbf{Step 3}$,  there is an arrow $[p_\alpha^-\alpha]:a\to k$ in $Q'$, 
thus , 
 \[\begin{tikzcd}[sep=small]
Q'=	&\cdots\cdots  a &&& k
	\arrow["{[p_{\alpha}^-\alpha]}",shift left=2, from=1-2, to=1-5]
	\arrow["{\alpha^*}", shift left=2, from=1-5, to=1-2]
	\arrow["{\rho^*}", from=1-5, to=1-5, loop, in=55, out=125, distance=10mm]
\end{tikzcd}\]
and \[I'=(I\setminus\{\rho^2,\alpha\beta,\beta\alpha\})\cup\{(\rho^*)^2,\alpha^*[p_\alpha^-\alpha],[p_\alpha^-\alpha]\alpha^*\}.\]
Because $p_{\alpha}^+$ doesn't exist, thus there is no $\textbf{Step 4}$. 
Thus we get the three arrows $\rho^*,\alpha^*,[p_\alpha^-\alpha]$ correspond to $\rho,\alpha,\beta$ respectively, and the three quadratic relations correspond to $\rho^2,\alpha\beta,\beta\alpha$ respectively. Hence $\mu_k^+(Q,I)=(Q,I)$ directly.
\end{proof}

\section{Cotilting mutation of gentle algebras}\label{s:s8 cotilting}

\subsection{Cotilting mutation}\label{ss:cotilting-mutation}
Let $A=\bbk Q/\langle I\rangle$ be a gentle algebra.
Recall that a finitely generated $A$-module $C$ is a \emph{cotilting module} if
\begin{enumerate}
    \item $\operatorname{id}_A C\leq 1$;
    \item $\Ext_A^1(C,C)=0$;
    \item there exists an exact sequence
    \[
    0 \longrightarrow C_1 \longrightarrow C_0 \longrightarrow D(A) \longrightarrow 0
    \]
    with $C_0,C_1\in\add C$, where $D=\Hom_{\bbk}(-,\bbk)$ denotes the standard $\bbk$-duality.
\end{enumerate}

Let $k\in Q_0$ and decompose the minimal injective cogenerator as
\[
D(A) = I(k) \oplus \overline C,
\qquad
\overline C := \bigoplus_{\ell \neq k} I(\ell).
\]
Suppose there exists an indecomposable $A$-module $I(k)^*\not\simeq I(k)$ such that
\[
C := \overline C \oplus I(k)^*
\]
is a basic cotilting $A$-module.

\begin{definition}
Under these assumptions, we call $C$ the \emph{cotilting mutation of the injective cogenerator $D(A)$ at $I(k)$}. The module $I(k)^*$ is called the \emph{exchange complement} of $I(k)$, and the algebra
\[
\mu_k^-(A):=\End_A(C)
\]
is called the \emph{cotilting mutation of $A$ at $k$}.
\end{definition}

\begin{remark}
The cotilting mutation $\mu_k^-(A)$ exists if and only if the tilting mutation $\mu_k^+(A^{\operatorname{op}})$ exists. This follows directly from the standard duality between tilting and cotilting modules over opposite algebras.
\end{remark}

\subsection{Combinatorial description and relation to tilting mutation}
When $(Q,I)$ is a gentle pair, the opposite pair $(Q^{\operatorname{op}}, I^{\operatorname{op}})$ is also gentle. This allows us to describe cotilting mutation combinatorially in terms of the tilting mutation established above.

\begin{definition}\label{def:cotilting_pair}
Let $(Q,I)$ be a gentle pair. We define the \emph{cotilting mutation} of $(Q,I)$ at $k$ as the gentle pair
\[
\mu_k^-(Q,I) := \bigl(\mu_k^+(Q^{\operatorname{op}}, I^{\operatorname{op}})\bigr)^{\operatorname{op}}.
\]
\end{definition}

We recall the fundamental duality between tilting and cotilting modules, which connects the algebraic and combinatorial constructions.

\begin{theorem}[Standard $\bbk$-Duality]\label{thm:happel-duality}
Let $A$ be a finite-dimensional $\bbk$-algebra and $D=\Hom_{\bbk}(-,\bbk)$ the standard duality. For any finitely generated $A$-module $T$, $T$ is a tilting $A$-module if and only if $D(T)$ is a cotilting $A^{\operatorname{op}}$-module. Moreover, there is a natural algebra isomorphism
\[
\End_A(T) \cong \End_{A^{\operatorname{op}}}(D(T))^{\operatorname{op}}.
\]
\end{theorem}

Using this duality together with the combinatorial description of tilting mutation, we obtain the following result.

\begin{theorem}\label{t:cotilting_mutation_isomorphism}
Let $(Q,I)$ be a gentle pair, let $A=\bbk Q/\langle I\rangle$, and set $(Q',I')=\mu_k^-(Q,I)$. Then the cotilting mutation algebra satisfies
\[
\mu_k^-(A)\cong \bbk Q'/\langle I'\rangle
\]
as $\bbk$-algebras.
\end{theorem}


\begin{thebibliography}{99}
\bibitem{A}
  T.~Aihara, Derived equivalences between symmetric special biserial algebras,
 \emph{J. Pure Appl. Algebra}, \textbf{219} (2015), no.~5, 1800--1825.



  
\bibitem{AI12}
T.~Aihara and O.~Iyama,
Silting mutation in triangulated categories,
\emph{J. Lond. Math. Soc. (2)} \textbf{85} (2012), no.~3, 633--668.

\bibitem{AS87}
I.~Assem and A.~Skowro\'nski,
Iterated tilted algebras of type $\widetilde{\mathbb A}_n$,
\emph{Math. Z.} \textbf{195} (1987), no.~2, 269--290.

\bibitem{AA08}
D. Avella-Alaminos, Derived classification of gentle algebras with two cycles, \emph{Bol. Soc. Mat. Mexicana (3)} {\bf 14} (2008), no.~2, 177--216; MR2667326

\bibitem{AG08}
D. Avella-Alaminos and C. Geiss, Combinatorial derived invariants for gentle algebras, \emph{J. Pure Appl. Algebra} {\bf 212} (2008), no.~1, 228--243.




\bibitem{APR}
M.~Auslander, M.~I.~Platzeck, and I.~Reiten,
Coxeter functors without diagrams,
\emph{Trans. Amer. Math. Soc.} \textbf{250} (1979), 1--46.

\bibitem{BS21}
K.~Baur and R.~Coelho Sim\~oes,
A geometric model for the module category of a gentle algebra,
\emph{Int. Math. Res. Not. IMRN} \textbf{2021}, no.~15, 11357--11392.

\bibitem{BGP}
J.~H.~Bernstein, I.~M.~Gel'fand, and V.~A.~Ponomarev, 
 Coxeter functors and Gabriel's theorem,
  \emph{Russian Math. Surveys},
\textbf{28} (1973), no.~2, 17--32.

 \bibitem{BB12}
G.~Bobi\'nski and A.~B.~Buan,
The algebras derived equivalent to gentle cluster tilted algebras,
\emph{J. Algebra Appl.} \textbf{11} (2012), no.~1,
1250012, 26~pp.

\bibitem{BM08}
G. Bobi\'nski and P. Malicki, On derived equivalence classification of gentle two-cycle algebras, \emph{Colloq. Math.} {\bf 112} (2008), no.~1, 33--72.


\bibitem{BB1980}
S.~Brenner and M.~C.~R.~Butler,
Generalizations of the Bernstein–Gelfand–Ponomarev reflection functors,
in \emph{Representation Theory II},
Lecture Notes in Math. \textbf{832},
Springer, Berlin, 1980, 103--169.

\bibitem{BIRS}
 A.~B.~Buan, O.~Iyama, I.~Reiten, and D.~Smith,
 Mutation of cluster-tilting objects and potentials,
 \emph{Amer. J. Math.} \textbf{133} (2011), no.~4, 835--887.

\bibitem{BMR}
A.~B.~Buan, B.~R.~Marsh, I.~Reiten,
Cluster mutation via quiver representations,
\emph{Comment. Math. Helv. } \textbf{83} (2008), no.~1,  143--177.

\bibitem{BMRRT}
A.~B.~Buan, B.~R.~Marsh, M.~Reineke, I.~Reiten, and G.~Todorov,
Tilting theory and cluster combinatorics,
\emph{Adv. Math.} \textbf{204} (2006), no.~2, 572--618.





\bibitem{Chang24}
W.~Chang,
Tilting-completion for gentle algebras,
arXiv:2412.13971, 2024.

\bibitem{CSc23}
W.~Chang and S.~Schroll,
A geometric realization of silting theory for gentle algebras,
\emph{Math. Z.} \textbf{303} (2023), no.~3, Paper No.~67, 37~pp.

\bibitem{CHU94}
F.~U.~Coelho, D.~Happel, and L.~Unger,
Complements to partial tilting modules,
\emph{J. Algebra} \textbf{170} (1994), no.~1, 184--205.

\bibitem{DGL26}
D.~Deng, S.~Geng, and P.~Liu,
Geometric models for endomorphism algebras of tilting modules over gentle algebras,
arXiv:2607.19945, 2026.

\bibitem{Fosse}
D.~Fosse,
A combinatorial procedure for tilting mutation,
arXiv:2112.08129, 2021.


\bibitem{HU1989} 
D.~Happel and L.~Unger,
Almost complete tilting modules,
\emph{Proc. Amer. Math. Soc.} \textbf{107} (1989), no.~3, 603--610.


\bibitem{HU1998}
D.~Happel and L.~Unger,
Complements and the generalized Nakayama conjecture,
in \emph{Algebras and Modules, II (Geiranger, 1996)},
CMS Conf. Proc., vol.~24, Amer. Math. Soc., Providence, RI, 1998, 293--310.

\bibitem{HSZ01}
 T. Holm, J. Schr\"oer, A. Zimmermann, Combinatorial derived equivalences between gentle algebras, manuscript (2001), 1-14.


\bibitem{IR}
 O.~Iyama and I.~Reiten,
 Fomin--Zelevinsky mutation and tilting modules over Calabi--Yau Algebras,
 \emph{Amer. J. Math.}  \textbf{130} (2008), no.~4, 1087--1149.
 
\bibitem{IY}
O.~Iyama and Y.~Yoshino,
Mutation in triangulated categories and rigid Cohen--Macaulay modules,
\emph{Invent. Math.} \textbf{172} (2008),  no.~1, 117--168.

\bibitem{KY2011}
B.~Keller and D.~Yang,
Derived equivalences from mutations of quivers with potential,
\emph{Adv. Math.} \textbf{226} (2011),  no.~3, 2118--2168.

\bibitem{KY2014}
S.~K\"onig and D.~Yang,
   Silting objects, simple-minded collections, t-structures and co-t-structures for finite-dimensional algebras, 
   \emph{Doc. Math.} \textbf{19} (2014), 403--438.
 
\bibitem{Ladkani11}
S.~Ladkani,
Mutation classes of certain quivers with potentials as derived equivalence classes,
arXiv:1102.4108, 2011.

\bibitem{Ni}
		Z.~Ni,
		{On BB-tilting-cotilting equivalence of
    permanently representation-finite algebras},
		\emph{Master's thesis}, Nanjing University, Nanjing, China, 2026.
\bibitem{Opp17}
S.~Oppermann,
Quivers for silting mutation,
\emph{Adv. Math.} \textbf{307} (2017), 684--714.

\bibitem{OPS}
S.~Opper, P.-G.~Plamondon and S.~Schroll,
 A geometric model for the derived category of gentle algebras,
arXiv:1801.09659, 2018.

\bibitem{RS91}
C.~Riedtmann and A.~Schofield,
On a simplicial complex associated with tilting modules,
\emph{Comment. Math. Helv.} \textbf{66} (1991), 70--78.

\bibitem{Saleh26}
I.~Saleh,
On the mutations of gentle quivers and admissible ideals,
arXiv:2607.28767, 2026.

\bibitem{S99}
J.~Schr\"oer,
Modules without self-extensions over gentle algebras,
\emph{J. Algebra} \textbf{216} (1999), no.~1, 178--189.

\bibitem{SZ03}
J.~Schr\"oer and A.~Zimmermann,
Stable endomorphism algebras of modules over special biserial algebras,
\emph{Math. Z.} \textbf{244} (2003), no.~3, 515--530.

\end{thebibliography}
\end{document}